\documentclass[]{article}

\usepackage{floatrow}
\newfloatcommand{capbtabbox}{table}[][\FBwidth]

\usepackage{subfig}

\usepackage{amsfonts}
\usepackage{titling}
\usepackage{amssymb}
\usepackage[margin=1.0in]{geometry}
\usepackage{multicol}
\usepackage[utf8]{inputenc}
\usepackage{graphicx}
\usepackage{setspace}
\usepackage[titletoc,title]{appendix}
\usepackage{lettrine}
\usepackage{changepage}
\usepackage[english]{babel}
\usepackage[section]{placeins}
\usepackage{gensymb}
\usepackage{amsmath,amsthm}
\usepackage{bm}
\usepackage{hyperref}
\usepackage{enumerate}
\usepackage{polynom}
\usepackage{enumitem}
\usepackage{mathtools}
\usepackage[mathscr]{euscript}
\usepackage{tikz}
\usepackage{graphicx}
\usepackage{authblk}
\usepackage[numbers,sort]{natbib}

\usepackage{array}
\newcolumntype{C}[1]{>{\centering\let\newline\\\arraybackslash\hspace{0pt}}m{#1}}

\usepackage{hhline}
\usepackage{tikz}
\usepackage{amssymb,amsthm}
\usepackage[T1]{fontenc}
\usepackage{lmodern}

\usepackage{algorithm}
\usepackage{algpseudocode}

\usepackage{booktabs}

\newcommand{\R}{\mathbb{R}}

\DeclareMathOperator*{\argmin}{arg\,min}

\newenvironment{prooftext}[1]{\par\noindent{\bf #1\ }}{\hfill\BlackBox\\[2mm]}

\mathtoolsset{showonlyrefs}

\newcommand{\N}{\mathbb{N}}
\newcommand{\calF}{\mathcal{F}}
\newcommand{\calM}{\mathcal{M}}
\newcommand{\calO}{\mathcal{O}}

\newcommand{\prox}{\textnormal{prox}}

\newcommand{\E}{\mathbb{E}}
\newcommand{\Ek}{\mathbb{E}_k}

\newcommand{\defeq}{\stackrel{\mathrm{def}}{=}}
\newcommand{\symnum}[3][2]{\stackrel{\mathclap{\raisebox{.5pt}{\footnotesize \textcircled{\raisebox{-.9pt} {#2}}}}}{#3}}

\newcommand{\numcirc}[1]{$\raisebox{.5pt}{\footnotesize \textcircled{\raisebox{-.9pt} {#1}}}$}

\newcommand{\BlackBox}{$\blacksquare$}

\newcommand{\norm}[1]{{|\kern-1.125pt|} #1 {|\kern-1.125pt|}}
\newcommand{\iprod}[2]{\langle #1,\,#2 \rangle}
\newcommand{\msum}{\mathbin{\scalebox{1.25}{\ensuremath{\sum}}}}

\newcommand{\beq}{\begin{equation}}
\newcommand{\eeq}{\end{equation}}
\newcommand{\beqn}{\begin{equation*}}
\newcommand{\eeqn}{\end{equation*}}
\newcommand{\beqs}{\begin{equation}\begin{small}}
\newcommand{\eeqs}{\end{small}\end{equation}}
\renewcommand{\numcirc}[1]{{\raisebox{.5pt}{\footnotesize \textcircled{\raisebox{-.9pt} {#1}}}}}

\newcommand{\pa}[1]{(#1)}
\newcommand{\Pa}[1]{\big({#1}\big)}
\newcommand{\bPa}[1]{{\Big(}{#1}{\Big)}}

\newcommand{\ba}[1]{\{#1\}}
\newcommand{\Ba}[1]{\big\{#1\big\}}
\newcommand{\bBa}[1]{\Big\{#1\Big\}}
\newcommand{\ca}[1]{[#1]}
\newcommand{\Ca}[1]{\big[#1\big]}
\newcommand{\bCa}[1]{\Big[#1\Big]}

\newcommand{\tnabla}{\widetilde{\nabla}}
\newcommand{\tnablak}[1]{\widetilde{\nabla}_{#1}}

\newcommand{\seq}[1]{\{#1\}_{k\in\N}}

\newcommand{\xsol}{x^{\star}}
\newcommand{\xbar}{\bar{x}}
\newcommand{\xkm}{x_{k-1}}

\newcommand{\xk}{x_{k}}
\newcommand{\xkp}{x_{k+1}}

\newcommand{\sfrac}[2]{ {{\frac{\raisebox{-0.175em}{\footnotesize $#1$}}{\raisebox{0.175em}{\footnotesize $#2$}}}} }

\usepackage{xspace}

\newcommand{\gradSGD}{\widetilde{\nabla}^{\textnormal{\tiny SGD}}}
\newcommand{\gradSAG}{\widetilde{\nabla}^{\textnormal{\tiny SAG}}}
\newcommand{\gradSAGA}{\widetilde{\nabla}^{\textnormal{\tiny SAGA}}}
\newcommand{\gradBSAGA}{\widetilde{\nabla}^{\textnormal{\tiny B-SAGA}}}
\newcommand{\gradSVRG}{\widetilde{\nabla}^{\textnormal{\tiny SVRG}}}
\newcommand{\gradBSVRG}{\widetilde{\nabla}^{\textnormal{\tiny B-SVRG}}}
\newcommand{\gradSARAH}{\widetilde{\nabla}^{\textnormal{\tiny SARAH}}}
\newcommand{\gradSARGE}{\widetilde{\nabla}^{\textnormal{\tiny SARGE}}}

\newtheoremstyle{dotless}{}{}{\itshape}{}{\bfseries}{}{ }{}
\theoremstyle{dotless}
\newtheorem*{proposition*}{Proposition}

\allowdisplaybreaks

\title{On Biased Stochastic Gradient Estimation}

\author{Derek Driggs\thanks{\texttt{d.driggs@damtp.cam.ac.uk}}}
\author{Jingwei Liang\thanks{\texttt{jl993@cam.ac.uk}},}
\author{Carola-Bibiane Sch\"onlieb\thanks{\texttt{cbs31@cam.ac.uk}}}
\affil{Department of Applied Mathematics and Theoretical Physics, Cambridge University}

\usetikzlibrary{positioning}

\usepackage{stmaryrd}

\mathtoolsset{showonlyrefs}

\newtheorem{theorem}{Theorem}
\newtheorem{lemma}[theorem]{Lemma}
\newtheorem{corollary}[theorem]{Corollary}

\newtheorem{remark}{Remark}
\newtheorem{definition}{Definition}

\allowdisplaybreaks

\begin{document}

\maketitle

\begin{abstract}
We present a uniform analysis of biased stochastic gradient methods for minimizing convex, strongly convex, and non-convex composite objectives, and identify settings where bias is useful in stochastic gradient estimation. The framework we present allows us to extend proximal support to biased algorithms, including SAG and SARAH, for the first time in the convex setting. 
We also use our framework to develop a new algorithm, Stochastic Average Recursive GradiEnt (SARGE), that achieves the oracle complexity lower-bound for non-convex, finite-sum objectives and requires strictly fewer calls to a stochastic gradient oracle per iteration than SVRG and SARAH. 
We support our theoretical results with numerical experiments that demonstrate the benefits of certain biased gradient estimators.
\end{abstract}

\section{Introduction}\label{sec:introduction}

In this paper, we focus on the following composite minimisation problem:
\beq\label{eq:mainopt}
\min_{x \in \R^p} \Ba{ F(x) \defeq f(x) + g(x) } .
\eeq
Throughout, we assume:
\begin{itemize}[itemsep=-1ex,topsep=2pt]
\item $g$ is proper and closed such that its proximity operator (see \eqref{eq:prox} in Section \ref{sec:prelims}) is well posed, 

\item $f$ admits finite-sum structure $f(x) \defeq \frac{1}{n} \sum_{i=1}^n f_i(x)$, and for all $i \in \{1,2,\cdots,n\}$, $\nabla f_i$ is $L$-Lipschitz continuous for some $L > 0$.
\end{itemize}
We place no further restrictions on $f_i$ or $g$ unless stated otherwise. 

Problems of this form arise frequently in many areas of science and engineering, such as machine learning, statistics, operations research, and imaging. 
For instance, in machine learning, these problems often arise as empirical risk minimisation problems from classification and regression tasks. Examples include ridge regression, logistic regression, Lasso, and $\ell_1$-regularized logistic regression \cite{bishop2006pattern}. Principal component analysis (PCA) can also be formulated as a problem with this structure, where the functions $f_i$ are non-convex \cite{pcanoncon,svrgplusplus}. In imaging, $\ell_1$ or total variation regularization is often combined with differentiable data discrepancy terms that appear in both convex and non-convex instances \cite{bredies2018mathematical}.

\subsection{Stochastic gradient methods}

Two classical approaches to solve \eqref{eq:mainopt} are the proximal gradient descent method (PGD) \cite{lions1979splitting} and its accelerated variants, including inertial PGD \cite{liang2017activity} and FISTA \cite{fista}. For these deterministic approaches, the full gradient of $f$ must be evaluated at each iteration, which often requires huge computational resources when $n$ is large. Such a drawback makes these schemes unsuitable for large-scale machine learning tasks. 

By exploiting the finite sum structure of $f$, stochastic gradient methods enjoy low per-iteration complexity while achieving comparable convergence rates. These qualities make stochastic gradient methods the standard approach to solving many problems in machine learning, and are gaining popularity in other areas such as image processing \cite{spdhg}. Stochastic gradient descent (SGD) was first proposed in the 1950's \cite{sgd} and has experienced a renaissance in the past decade, with numerous variants of SGD proposed in the literature (see, for instance, \cite{SAG,SVRG,SAGA} and references therein). Most of these algorithms can be summarised into one general form, which is described below in Algorithm \ref{alg:general_sgd}.

\begin{center}
\begin{minipage}{0.95\linewidth}

\begin{algorithm}[H]
\caption{Stochastic gradient descent framework}
\begin{algorithmic}[1]
\Require starting point $x_0 \in \R^p$, gradient estimator $\tnabla$.
\For{$k = 0,1,\cdots, T-1$}
\State \textrm{Compute stochastic gradient $\tnablak{k}$ at current iteration $k$.}
\State \textrm{Choose step size/learning rate $\eta_k$.}
\State \textrm{Update $\xkp$:}
\beq\label{eq:general_sgd}
x_{k+1} \leftarrow \prox_{\eta_{k} g} \pa{ x_k - \eta_{k} \tnablak{k} } .
\eeq
\EndFor
\end{algorithmic}
\label{alg:general_sgd}
\end{algorithm}

\end{minipage}
\end{center}

{\noindent}Below we summarize several popular stochastic gradient estimators $\tnablak{k}$:
\begin{itemize}[leftmargin=2em, label={\small{${\blacktriangleright}$}}, itemsep=-1pt]
\item {\bf SGD} Classic stochastic gradient descent \cite{sgd} uses the following gradient estimator at iteration $k$:
\beq\label{eq:grad_sgd}
\begin{aligned}
&\left\lfloor \begin{aligned}
& \textrm{sample $j_k$ uniformly at random from $\{1,...,n\}$} , \\
& \gradSGD_{k} = \nabla f_{j_{k}} (\xk).
\end{aligned}
\right. \\
\end{aligned}
\eeq
At each step, SGD uses the gradient of the sampled function $\nabla f_{j_k}(\xk)$ as a stochastic approximation of the full gradient $\nabla f(\xk)$. It is an unbiased estimate as $\Ek\ca{\nabla f_{j_k}(x_k)} = \nabla f(\xk)$. It is also \emph{memoryless}: every update of $x_{k+1}$ depends only upon $x_k$ and the random variable $j_k$. The variance of SGD is does not vanish as $\xk$ converges. 

\item {\bf SAG} To deal with the non-vanishing variance of SGD, in \cite{roux2012stochastic,SAG} the authors introduce the SAG gradient estimator, which uses the gradient history to decrease its variance. With $\nabla f_i(\varphi^i_0) = 0, i=1,...,n$, the SAG gradient estimator is computed using the following procedure:
 \beq\label{eq:grad_sag}
 \begin{aligned}
 &\left\lfloor \begin{aligned}
 & \textrm{sample $j_k$ uniformly at random from $\{1,...,n\}$} , \\
 & \gradSAG_{k} = \tfrac{1}{n} \pa{ \nabla f_{j_{k}} (\xk) - \nabla f_{j_{k}} (\varphi_{k}^{j_k}) } + \tfrac{1}{n} \msum_{i=1}^{n} \nabla f_{i} (\varphi_{k}^{i}) , \\
 & \textrm{update the gradient history: $\varphi_{k+1}^{j_{k}} = \xk$ and } \nabla f_{i} (\varphi_{k+1}^{i}) = \left\{ \begin{aligned} &\nabla f_{i}(\xk) && \textrm{if~~} i = j_{k} , \\ & \nabla f_{i} (\varphi_{k}^{i}) && \textrm{o.w.} \end{aligned} \right. 
 \end{aligned}
 \right. \\
 \end{aligned}
 \eeq
Here, for each $i \in \ba{1,...,n}$, $\nabla f_{i} (\varphi_{k}^{i})$ is a stored gradient of $\nabla f_i$ from a previous iteration. 
With the help of memory, the variance of the SAG gradient estimator diminishes as $\xk$ approaches the solution; estimators that satisfy this property are known as \emph{variance-reduced} estimators. In contrast to the SGD estimator, $\gradSAG_{k}$ is a {\em biased} estimate of $\nabla f(\xk)$.

\item {\bf SAGA} Based on \cite{roux2012stochastic,SAG}, \cite{SAGA} propose the unbiased gradient estimator SAGA, which is computed using the procedure below.
 \beq\label{eq:grad_saga}
 \begin{aligned}
 &\left\lfloor \begin{aligned}
 & \textrm{Sample $j_k$ uniformly at random from $\{1,...,n\}$} , \\
 & \gradSAGA_{k} = { \nabla f_{j_{k}} (\xk) - \nabla f_{j_{k}} (\varphi_{k}^{j_k}) } + \tfrac{1}{n} \msum_{i=1}^{n} \nabla f_{i} (\varphi_{k}^{i}) , \\
 & \textrm{update the gradient history: $\varphi_{k+1}^{j_{k}} = \xk$ and } \nabla f_{i} (\varphi_{k+1}^{i}) = \left\{ \begin{aligned} &\nabla f_{i}(\xk) && \textrm{if~~} i = j_{k} , \\ & \nabla f_{i} (\varphi_{k}^{i}) && \textrm{o.w.} \end{aligned} \right. 
 \end{aligned}
 \right. \\
 \end{aligned}
 \eeq
Compared to $\gradSAG$, the SAGA estimator gives less weight to stored gradients. With this adjustment, $\gradSAGA$ is unbiased while maintaining the variance reduction property. Similar gradient estimators can be found in Point-SAGA \cite{pointsaga}, Finito \cite{finito}, MISO \cite{miso}, SDCA \cite{sdca}, and those in \cite{neighbors}.

\item {\bf SVRG}
Another popular variance-reduced estimator is SVRG \cite{SVRG}. The SVRG gradient estimator is computed as follows:
\beq\label{eq:grad_svrg}
\begin{aligned}
&\left\lfloor  \begin{aligned}
&\textrm{For $s=0,\dotsm,S$} \\
& \textstyle \nabla f(\varphi_s) = \tfrac{1}{n} \sum_{i=1}^{n} \nabla f_{i}(\varphi_s), \\
&\textrm{For $k=1,\dotsm,m$} \\
&\left\lfloor  \begin{aligned}
   & \textrm{Sample $j_{k}$ uniformly at random from $\ba{1,\dotsm,n}$} , \\
   & \gradSVRG_{k} = \nabla f_{j_{k}}(x_{k}) - \nabla f_{j_{k}}(\varphi_s) + \nabla f(\varphi_s) , \\
  \end{aligned}
\right. \\
\end{aligned}
\right. \\
\end{aligned}
\eeq
where $\varphi_s$ is a ``snapshot'' point updated every $m$ steps. The algorithms prox-SVRG \cite{proxSVRG}, Katyusha \cite{katyusha}, KatyushaX \cite{katyushax}, Natasha \cite{natasha}, Natasha2 \cite{natasha2}, 
MiG \cite{MiG}, ASVRG \cite{ASVRG}, and VARAG \cite{varag} use the SVRG gradient estimator.

\item {\bf SARAH} In \cite{sarah} the authors proposed a recursive modification to SVRG.
\beq\label{eq:grad_sarah}
\begin{aligned}
&\left\lfloor  \begin{aligned}
&\textrm{For $s=0,\dotsm,S$} \\
& \textstyle \gradSARAH_{k-1} = \tfrac{1}{n} \sum_{i=1}^{n} \nabla f_{i}(\varphi_s), \\
&\textrm{For $k=1,\dotsm,m$} \\
&\left\lfloor  \begin{aligned}
   & \textrm{Sample $j_{k}$ uniformly at random from $\ba{1,\dotsm,n}$} , \\
   & \gradSARAH_{k} = \nabla f_{j_{k}}(x_{k}) - \nabla f_{j_{k}}(\xkm) + \gradSARAH_{k-1} , \\
  \end{aligned}
\right. \\
\end{aligned}
\right. \\
\end{aligned}
\eeq
Like SAG, SARAH is a biased gradient estimator. 
It is also used in prox-SARAH \cite{proxsarah},
SPIDER \cite{spider}, SPIDERBoost \cite{spiderboost} and SPIDER-M \cite{spiderm}. 

\end{itemize}
We refer to algorithms employing (un)biased gradient estimators as (un)biased stochastic algorithms, respectively.
The body of work on biased algorithms is stunted compared to the enormous literature on unbiased algorithms, leading to several gaps in the development of biased stochastic gradient methods. We list a few below.
\begin{itemize}[leftmargin=2em, itemsep=-2pt,topsep=2pt]
\item \textbf{Complex convergence proofs.} It is commonly believed that the relationship $\Ek [ \tnablak{k} ] = \nabla f(x_k)$ is essential for a simple convergence analysis (see, for example, the discussion in \cite{SAGA}). The convergence proof of the biased algorithm SAG is especially complex, requiring computational verification \cite{roux2012stochastic,SAG}.

\item \textbf{Sub-optimal convergence rates.} In the convex setting with $g\equiv 0$, SARAH achieves a complexity bound of $\mathcal{O}\pa{ \frac{\log(1/\epsilon )}{\epsilon} }$ \cite{sarah} for finding a point $\bar{x}_{k}$ such that $\E [F(\bar{x}_{k}) - F(\xsol)] \le \epsilon$. In comparison, SAGA and SVRG achieve a complexity bound of $\mathcal{O}( 1 / \epsilon )$ which is the same as deterministic proximal gradient descent.

\item \textbf{Lack of proximal support.} Bias also makes it difficult to handle non-smooth functions. To the best of our knowledge, there are no theoretical guarantees for biased algorithms to solve \eqref{eq:mainopt} with $g \not \equiv 0$ that take advantage of convexity when it is present.

\end{itemize}
Despite the above shortcomings, there are notable exceptions that suggest biased algorithms are worth further consideration. Recently, \cite{proxsarah,spider,spiderboost,spiderm} proved that algorithms using the SARAH gradient estimator require $\mathcal{O} \pa{ \sqrt{n} / \epsilon^2 }$ stochastic gradient evaluations to find an $\epsilon$-first order stationary point. This matches the complexity lower-bound for non-convex, finite-sum optimisation for smooth functions $f_i$ and $n \le \mathcal{O}(\epsilon^{-4})$ \cite{spider}. For comparison, the best complexity bound obtained for SAGA and SVRG in this setting is $\mathcal{O} \pa{ n^{2/3} / \epsilon^2 }$ \cite{reddi,zhunoncon}. 
A detailed summary of existing complexity bounds for the variance-reduced gradient estimators mentioned above is provided in Table \ref{tab:rates1} for convex, strongly convex and non-convex settings.

\renewcommand{\arraystretch}{1.25}

\begin{table}[!ht]
\begin{minipage}{\textwidth}
\centering
\begin{tabular}{| c | c | c | c | c |}
\hline
& Convex & Strongly Convex & Non-Convex & Proximal Support? \\
\hline
SAGA & $\mathcal{O}( \frac{n L}{\epsilon} )$ & $\mathcal{O}( (n + \kappa) \log(1 / \epsilon))$ & $\mathcal{O} ( \frac{n L}{\epsilon^2} )^{b}$ & Yes \\
\hline
SVRG & $\mathcal{O}( \frac{n L}{\epsilon} )$\footnote[1]{The algorithm SVRG++ reduces this rate to $\mathcal{O}( n \log( 1 /  \epsilon ) + L / \epsilon )$ using an epoch-doubling procedure \cite{svrgplusplus}.} & $\mathcal{O} ( (n + \kappa) \log (1 / \epsilon )  )$ & $\mathcal{O}( \frac{n L}{\epsilon^2} )$\footnote[2]{Mini-batching reduces the dependence on $n$ to $n^{2/3}$ \cite{reddi,zhunoncon}, and these rates are proven only in the case $g$ is convex.} & Yes \\
\hline
SAG & $\mathcal{O} ( \frac{n L}{\epsilon} )$ & $\mathcal{O} ( (n + \kappa) \log(1 / \epsilon) )$ & \textbf{Unknown} & \textbf{No} \\
\hline
SARAH & $\mathcal{O} ( \frac{L\log(1 / \epsilon) }{\epsilon}  )$ & $\mathcal{O} ( (n + \kappa) \log(1 / \epsilon))$ & $\mathcal{O} ( \frac{\sqrt{n} L}{\epsilon^2} )$ & \textbf{Non-Convex Only} \\
\hline
\end{tabular}
\caption{Existing complexity bounds for stochastic gradient estimators under different settings. These complexities represent the number of stochastic gradient oracle calls required to find a point satisfying $\E [ F(\bar{x}_{k}) - F(\xsol) ] \le \epsilon$ for the convex case, $\E \ca{ \|x_k - \xsol \|^2 } \le \epsilon$ for the strongly convex case, and an $\epsilon$-approximate stationary point (as in Definition \ref{def:gengrad}) in the non-convex case. In the strongly convex case, $\mu$ is the strong convexity constant, and $\kappa = L/\mu$ is the condition number.}
\label{tab:rates1}
\end{minipage}
\end{table}

\renewcommand{\arraystretch}{1.0}

\subsection{Contributions}

{\noindent}This work provides three main contributions: 
\begin{enumerate}[itemsep=-1pt,topsep=2pt]
\item We introduce a framework for the systematic analysis of a large class of stochastic gradient methods and investigate a bias-variance tradeoff arising from our analysis. Our analysis provides proximal support to biased algorithms for the first time in the convex setting.

\item We apply our framework to derive convergence rates for SARAH and biased versions of SAGA and SVRG on convex, strongly convex, and non-convex problems.

\item We design a new recursive gradient estimator, Stochastic Average Recursive GradiEnt (SARGE), that achieves the same convergence rates as SARAH but never computes a full gradient, giving it a strictly smaller per-iteration complexity than SARAH. In particular, we show that SARGE achieves the oracle complexity lower bound for non-convex finite-sum optimisation.

\end{enumerate}
To study the effects of bias on the SAGA and SVRG estimators, we introduce Biased SAGA (B-SAGA) and Biased SVRG (B-SVRG). For $\theta > 0$ the B-SAGA and B-SVRG gradient estimators are

\begin{itemize}[itemsep=-1pt,topsep=2pt]
\item B-SAGA: 
 \beq
\gradBSAGA_{k} \defeq \tfrac{1}{\theta} \pa{ \nabla f_{j_k}(x_k) - \nabla f_{j_k}(\varphi_k^{j_k}) } + \tfrac{1}{n} \msum_{i = 1}^n \nabla f_i(\varphi_k^i),
\eeq

\item B-SVRG: 
 \beq
\gradBSVRG_{k} \defeq \tfrac{1}{\theta} \pa{ \nabla f_{j_k}(x_k) - \nabla f_{j_k}(\varphi_s) } + \nabla f(\varphi_s).
\eeq

\end{itemize}
In both B-SAGA and B-SVRG, the bias parameter $\theta$ adjusts how much weight is given to stored gradient information. When $\theta = n$, $\gradBSAGA_{k}$ recovers the SAG gradient estimator.

Motivated by the desirable properties of SARAH, we propose a new gradient estimator, Stochastic Average Recursive GradiEnt (SARGE), which is defined below
\beq
\gradSARGE_k \defeq \nabla f_{j_k}(x_k) - \psi_k^{j_k} + \tfrac{1}{n} \msum_{i=1}^n \psi_k^i - \pa{ 1 - \tfrac{1}{n} } \pa{ \nabla f_{j_k}(x_{k-1}) - \gradSARGE_{k-1} } ,
\eeq
where the variables $\psi_k^i$ follow the update rule $\psi_{k+1}^{i_k} = \nabla f_{j_k}(x_k) - \pa{1 - \tfrac{1}{n}} \nabla f_{j_k} (x_{k-1})$. 
Similar to SAGA, SARGE uses stored gradient information to avoid having to compute the full gradient, a computational burden that SVRG and SARAH require for variance reduction.

A summary of the complexity results obtained from our analysis for SAG/B-SAGA, B-SVRG, SARAH, and SARGE are provided in Table \ref{tab:rates2}. Note that the result for SAG is included in B-SAGA.

\renewcommand{\arraystretch}{1.25}

\begin{table}[!ht]
\begin{minipage}{\textwidth}
\centering
\begin{tabular}{| c | c | c | c | c |}
\hline
& Convex & Strongly Convex & Non-Convex & Proximal Support? \\
\hline
B-SAGA\footnote[3]{Mini-batching reduces the dependence on $n$ to $n^{2/3}$ as in \cite{reddi,zhunoncon} giving these algorithms a lower complexity than full-gradient methods, but we do not consider mini-batching in this work.} & $\mathcal{O}( \frac{n L}{\epsilon} )$ & $\mathcal{O}( n \kappa \log(1 / \epsilon))$ & $\mathcal{O}( \frac{n L}{\epsilon^2} )$ & Yes \\
\hline
B-SVRG\textsuperscript{c} & $\mathcal{O}( \frac{n L}{\epsilon} )$ & $\mathcal{O}( n \kappa \log(1 / \epsilon))$ & $\mathcal{O}( \frac{n L}{\epsilon^2} )$ & Yes \\
\hline
SARAH & $\mathcal{O}( \frac{\sqrt{n} L}{\epsilon} )$ & $\mathcal{O}( \max\{\sqrt{n} \kappa, n\} \log(1 / \epsilon))$ & $\mathcal{O}( \frac{\sqrt{n} L}{\epsilon^2} )$ & Yes \\
\hline
SARGE & $\mathcal{O}( \frac{\sqrt{n} L}{\epsilon} )$ & $\mathcal{O}( \max\{\sqrt{n} \kappa, n\} \log(1 / \epsilon))$ & $\mathcal{O}( \frac{\sqrt{n} L}{\epsilon^2} )$ & {Yes} \\
\hline
\end{tabular}
\caption{Complexity bounds obtained from our analysis framework. These complexities represent the number of stochastic gradient oracle calls required to find a point $\bar{x}_k$ satisfying $\E [ F(\bar{x}_k) - F(\xsol) ] \le \epsilon$ for the convex case, $\E \ca{ \|x_k - \xsol \|^2 } \le \epsilon$ for the strongly convex case, and an $\epsilon$-approximate stationary point in the non-convex case.}
\label{tab:rates2}
\end{minipage}
\end{table}

\renewcommand{\arraystretch}{1.0}

\paragraph{Paper organization}
Preliminary results and notations are provided in Section \ref{sec:prelims}. A discussion on the tradeoff between bias and variance in stochastic optimisation is included in Section \ref{sec:bvtradeoff}. 
Our main convergence results are presented in Section \ref{sec:rates}. 
In Section \ref{sec:experiments}, we substantiate our theoretical results using numerical experiments involving several classic regression tasks arising from machine learning. All the proofs of the main results are collected in the appendix.

\section{Preliminaries and notations}\label{sec:prelims}

Throughout the paper, $\R^p$ is a $p$-dimensional Euclidean space equipped with scalar inner product $\iprod{\cdot}{\cdot}$ and associated norm $\norm{\cdot}$. The sub-differential of a proper closed convex function $g$ is the set-valued operator defined by $\partial g \defeq \Ba{ v\in\R^n | g(x') \geq g(x) + \iprod{v}{x'-x} , \forall x' \in \R^n }$, the proximal map of $g$ is defined as
\beq\label{eq:prox}
\textstyle \prox_{\eta g}(y) \defeq \argmin_{x\in\R^n} \eta g(x) + \tfrac{1}{2}\norm{x-y}^2,
\eeq
where $\eta > 0$ and $y \in \R^p$. With $y^+ = \prox_{\gamma g}(y)$, \eqref{eq:prox} is equivalent to $y - y^+ \in \eta \partial g(y^+)$.

Below we summarize some useful results in convex analysis.

\begin{lemma}[{\cite[Thm 2.1.5]{nest2004}}]
\label{lem:2L}
Suppose $f$ is convex with an $L$-Lipschitz continuous gradient. We have for every $x,u \in \R^p$,
\beq
\|\nabla f(x) - \nabla f(u)\|^2 
\le 2 L \pa{ f(x) - f(u) - \langle \nabla f(u),x-u \rangle } .
\eeq
\end{lemma}

When $f$ is a finite sum as in \eqref{eq:mainopt}, Lemma \ref{lem:2L} is equivalent to the following result.

\begin{lemma}
\label{lem:varcan}
Let $f(x) = \tfrac{1}{n} \sum_{i=1}^n f_i(x)$, where each $f_i$ is convex with an $L$-Lipschitz gradient. Then for every $x, u \in \R^p$,
\beq
\textstyle \tfrac{1}{2 L n} \sum_{i=1}^n \|\nabla f_i(x) - \nabla f_i(u)\|^2 
\le f(x) - f(u) - \langle \nabla f(u), x - u \rangle .
\eeq
\end{lemma}

\noindent Lemma \ref{lem:varcan} is obtained by applying Lemma \ref{lem:2L} to each $f_i$ and averaging.

\begin{lemma}
\label{lem:prox}
Suppose $g$ is $\mu$-strongly convex with $\mu \ge 0$, and let $z = \prox_{\eta g} ( x - \eta d )$ for some $x, d \in \R^p$ and $\eta > 0$. Then, for any $y \in \R^p$,
\beq
\textstyle \eta \langle d, z - y \rangle 
\le \tfrac{1}{2} \|x - y\|^2 - \tfrac{1 + \mu \eta}{2} \|z - y\|^2 - \tfrac{1}{2} \|z - x\|^2 - \eta g(z) + \eta g(y).
\eeq
\end{lemma}

\begin{proof}
By the strong convexity of $g$,
\beq
g(z) - g(y) 
\le \langle \xi, z - y \rangle - \tfrac{\mu}{2} \|z - y\|^2 ,\quad \forall \xi \in \partial g(z) .
\eeq
From the definition of the proximal operator, we have that $\frac{1}{\eta} (x - z) - d \in \partial g(z)$. Therefore,
\beq
\begin{aligned}
\textstyle g(z) - g(y) 
&\textstyle \le \langle \xi, z - y \rangle - \tfrac{\mu}{2} \|z - y\|^2 \\
& =\textstyle \tfrac{1}{\eta} \langle x - z - \eta d, z - y \rangle - \tfrac{\mu}{2} \|z - y\|^2 \\
& =\textstyle - \iprod{ d }{ z - y } + \tfrac{1}{\eta} \iprod{ x - z}{ z - y } - \tfrac{\mu}{2} \|z - y\|^2 \\
& =\textstyle - \iprod{ d }{ z - y } - \tfrac{1}{2\eta} \norm{x-z}^2 - \tfrac{1}{2\eta} \norm{z-y}^2 + \tfrac{1}{2\eta} \norm{x-y}^2 - \tfrac{\mu}{2} \|z - y\|^2 .
\end{aligned}
\eeq
Multiplying by $\eta$ and rearranging yields the assertion.
\end{proof}

The next lemma is an analogue of the descent lemma for gradient descent when the gradient is replaced with an arbitrary vector $d$.

\begin{lemma}
\label{lem:descent}
Suppose $g$ is $\mu$-strongly convex for $\mu \ge 0$, and let $z = \prox_{\eta g} ( x - \eta d )$. The following inequality holds for any $\lambda > 0$.
\beq
\textstyle 0 \le \eta (F(x) - F(z)) + \tfrac{\eta}{2 L \lambda} \|d - \nabla f(x)\|^2 + (\tfrac{\eta L (\lambda + 1)}{2} - \tfrac{2 + \mu \eta}{2} ) \|z - x\|^2 .
\eeq

\end{lemma}

\begin{proof}
This follows immediately from Lemma \ref{lem:prox}.
\beq
\begin{aligned}
\textstyle 0 
&\textstyle = \eta \langle d, x - z \rangle + \eta \langle d, z - x \rangle \\
&\textstyle \symnum{1}{\le} \eta \langle d, x - z \rangle - \tfrac{2 + \mu \eta}{2} \|z - x\|^2 + \eta \pa{ g(x) - g(z) } \\
&\textstyle = \eta \langle \nabla f(x), x - z \rangle + \eta \langle d - \nabla f(x), x - z \rangle - \tfrac{2 + \mu \eta}{2} \|z - x\|^2 + \eta \pa{ g(x) - g(z) } \\
&\textstyle \symnum{2}{\leq} \eta ( F(x) - F(z) ) + \eta \langle d - \nabla f(x), x - z \rangle + (\tfrac{\eta L}{2} - \tfrac{2 + \mu \eta}{2} ) \|z - x\|^2 \\
&\textstyle \symnum{3}{\leq} \eta ( F(x) - F(z) ) + \tfrac{\eta}{2 L \lambda} \|d - \nabla f(x)\|^2 + (\tfrac{\eta L (\lambda + 1)}{2} - \tfrac{2 + \mu \eta}{2} ) \|z - x\|^2.
\end{aligned}
\eeq
Inequality \numcirc{1} is due to Lemma \ref{lem:prox} with $y = x$, \numcirc{2} is due to the Lipschitz continuity of $\nabla f_i$, and \numcirc{3} is Young's.
\end{proof}

The previous two lemmas require $g$ to be convex. Similar results hold in the non-convex case as well.

\begin{lemma}
\label{lem:nonconprox}
    Let $z = \prox_{\eta g} ( x - \eta d )$ for some $x, d \in \R^p$ and $\eta > 0$. Then, for any $y \in \R^p$,
    \beq
    \textstyle \eta \langle d, z - y \rangle 
    \le \tfrac{1}{2} \|x - y\|^2 - \tfrac{1}{2} \|z - x\|^2 - \eta g(z) + \eta g(y).
    \eeq
\end{lemma}

\begin{proof}
    By the Lipschitz continuity of $\nabla f$, we have the inequalities
    \begin{equation}
    \begin{aligned}
        f(x) - f(y) &\textstyle  \le \langle \nabla f(x), x - y \rangle + \frac{L}{2} \|x - y\|^2, \\
        f(z) - f(x) &\textstyle  \le \langle \nabla f(x), z - x \rangle + \frac{L}{2} \|z - x\|^2.
    \end{aligned}
    \end{equation}
    Furthermore, by the definition of $z$,
    \begin{equation}
        z \in \argmin_v \left\{ \langle d, v - x \rangle + \frac{1}{2 \eta} \|v - x\|^2 + g (v) \right\}.
    \end{equation}
    Taking $v = y$, we obtain
    \begin{equation}
        \textstyle g(z) - g(y) \le \langle d, y - z \rangle + \frac{1}{2 \eta} \left( \|x - y\|^2 - \|x - z\|^2 \right).
    \end{equation}
    Adding these three inequalities and multiplying by $\eta$ completes the proof.
\end{proof}

\begin{lemma}
\label{lem:noncon}
    Let $z = \prox_{\eta g} ( x - \eta d )$. Then
    \begin{equation}
        \textstyle F(z) \le F(y) + \langle \nabla f(x) - d, z - y \rangle + ( \frac{L}{2} - \frac{1}{2 \eta} ) \|x - z\|^2 + ( \frac{L}{2} + \frac{1}{2 \eta} ) \|x - y\|^2.
    \end{equation}
\end{lemma}

\begin{proof}
    By the Lipschitz continuity of $\nabla f$, we have the inequalities
    \begin{equation}
    \begin{aligned}
        f(x) - f(y) &\textstyle  \le \langle \nabla f(x), x - y \rangle + \frac{L}{2} \|x - y\|^2, \\
        f(z) - f(x) &\textstyle  \le \langle \nabla f(x), z - x \rangle + \frac{L}{2} \|z - x\|^2.
    \end{aligned}
    \end{equation}
    Furthermore, by Lemma \ref{lem:nonconprox},
    \begin{equation}
        \textstyle g(z) - g(y) \le \langle d, y - z \rangle + \frac{1}{2 \eta} \left( \|x - y\|^2 - \|x - z\|^2 \right).
    \end{equation}
    Adding these inequalities together completes the proof.
\end{proof}

In the non-convex setting, to measure convergence of the sequence to a first-order stationary point, we use the notion of a generalized gradient \cite{nest2004}.

\begin{definition}[{Generalized gradient map}]
\label{def:gengrad}
The {generalized gradient map} is defined as
\beqn%\label{eq:gengrad}
\textstyle \mathcal{G}_{\eta_k} (x_k) \defeq \tfrac{1}{\eta_k} \pa{ x_k - \prox_{\eta_k g} \pa{ x_k - \eta_k \nabla f(x_k) } } .
\eeqn
\end{definition}

When $g \equiv 0$, we have $\mathcal{G}_{\eta_k}(x_k) = \nabla f(x_k) \to 0$ if the sequence $\{\xk\}$ converges to some $\xsol \in \R^p$ such that $\nabla f(\xsol) = 0$. 
For nontrivial $g$, suppose $\inf_{k} \eta_k > 0$ and $\xk$ converges to some $\xsol$ such that $\xsol = \prox_{\eta g} \pa{ \xsol - \eta \nabla f(\xsol) }$, then $\mathcal{G}_{\eta_k}(x) \to 0$. 
Such a point $\xsol$ is called \emph{first-order stationary point} of \eqref{eq:mainopt} and an $\epsilon$-\emph{first-order stationary point} is a point satisfying $\|\mathcal{G}_{\eta}(x)\| \le \epsilon$ for some $\eta > 0$.

\section{A bias-variance tradeoff in stochastic gradient methods}\label{sec:bvtradeoff}

In this section, we discuss the effect of the bias and variance of a stochastic gradient estimator on the performance of Algorithm \ref{alg:general_sgd}. It is elementary that the mean-squared error (MSE) of a stochastic estimator can be decomposed into the sum of its variance and squared bias. In our setting,
\beq\label{eq:msedecomp}
\Ek \ca{ \|\tnablak{k} - \nabla f(x_k)\|^2 } = \| \Ek [\tnablak{k}] - \nabla f(x_k) \|^2 + \Ek \ca{ \| \tnablak{k} - \Ek [\tnablak{k}] \|^2 } .
\eeq
This decomposition shows that a biased estimator might have a smaller MSE than an unbiased estimator as long as the bias sufficiently diminishes the variance. This is the \emph{bias-variance tradeoff}. As we see below, a bias-variance tradeoff exists in our analysis of stochastic gradient methods, but with a slightly different form.

In what follows, we first discuss the bias-variance tradeoff in the convex settings and then the non-convex setting.

\subsection{Convex case}

Let $\xsol$ be a global minimizer of problem \eqref{eq:mainopt}. 
From the update \eqref{eq:general_sgd}, let $w_{k+1} \in \partial g(\xkp)$. We have the following bound on the suboptimality at $x_{k+1}$:
\beq\label{eq:framework}
\begin{aligned}
&\Ek [ F (x_{k+1}) - F(\xsol) ] \\
& = \Ek [ f(x_{k+1}) - f(x_k) + f(x_k) - f(\xsol) + g(x_{k+1}) - g(\xsol) ] \\
&\symnum{1}{\le} \tfrac{L}{2} \Ek \ca{ \|x_{k+1} - x_k\|^2 } + \Ek \ca{ \langle \nabla f(x_k), x_{k+1} - x_k \rangle } + \langle \nabla f(x_k), x_k - \xsol \rangle + \Ek [g(x_{k+1}) - g(\xsol)] \\
&= \tfrac{L}{2} \Ek \ca{ \|x_{k+1} - x_k\|^2 } + \Ek \ca{ \langle \nabla f(x_k) - \tnablak{k}, x_{k+1} - x_k \rangle } \\
& \qquad + \Ek \ca{ \langle \nabla f(x_k) - \tnablak{k}, x_k - \xsol \rangle } + \Ek \ca{ \langle \tnablak{k}, x_{k+1} - \xsol \rangle } + \Ek [g(x_{k+1}) - g(\xsol)] \\
&\symnum{2}{\le} \tfrac{\epsilon}{2} \Ek \ca{ \| \nabla f(x_k) - \tnablak{k} \|^2 } + ( \tfrac{L}{2} + \tfrac{1}{2 \epsilon} ) \Ek \ca{ \|x_{k+1} - x_k \|^2 } \\
& \qquad + \Ek \ca{ \langle \nabla f(x_k) - \tnablak{k}, x_k - \xsol \rangle } + \Ek \ca{ \langle \tnablak{k} + w_{k+1}, x_{k+1} - \xsol \rangle - \tfrac{\mu}{2} \|x_{k+1} - x^*\|^2 } \\
&\symnum{3}{\le} \tfrac{\epsilon}{2} \Ek \ca{ \| \nabla f(x_k) - \tnablak{k} \|^2 } + ( \tfrac{L}{2} + \tfrac{1}{2 \epsilon} - \tfrac{1}{2 \eta} ) \Ek \ca{ \|x_{k+1} - x_k \|^2 } \\
& \qquad + \Ek \ca{ \langle \nabla f(x_k) - \tnablak{k}, x_k - \xsol \rangle } - \tfrac{1 + \mu \eta}{2 \eta} \Ek \ca{ \|x_{k+1} - \xsol \|^2 } + \tfrac{1}{2 \eta} { \|x_k - \xsol \|^2 } .
\end{aligned}
\eeq
Inequality $\raisebox{.5pt}{\footnotesize \textcircled{\raisebox{-.9pt} {1}}}$ follows from the convexity of $f$ and Lipschitz continuity of $\nabla f$, $\raisebox{.5pt}{\footnotesize \textcircled{\raisebox{-.9pt} {2}}}$ follows from the (strong) convexity of $g$, and $\raisebox{.5pt}{\footnotesize \textcircled{\raisebox{-.9pt} {3}}}$ comes from the implicit definition of the proximal operator \eqref{eq:prox}.
For the last line of the inequality, we observe that the inner product term $\Ek \ca{ \iprod{ \nabla f(x_k) - \tnablak{k} }{ x_k - \xsol} }$ vanishes when $\tnablak{k}$ is an unbiased estimate of $\nabla f(\xk)$. When the estimator is biased, we must develop a new way to control this term, together with $\Ek \ca{\norm{\nabla f(\xk) - \tnablak{k}}^2}$.

Hence, the following terms arise in our convergence analysis from the bias and the variance of the gradient estimator:
\beq\label{eq:tradeoff}
\begin{aligned}
\textrm{Bias} &: \Ek \ca{ \iprod{ \nabla f(x_k) - \tnablak{k} }{ x_k - \xsol} } \enskip\textrm{and}\enskip \| \Ek [\tnablak{k}] - \nabla f(x_k) \|^2 , \\
\textrm{Variance} &: \Ek \ca{ \| \tnablak{k} - \Ek [\tnablak{k}] \|^2 } . 
\end{aligned}
\eeq

\begin{remark}[Non-composite case ${g=0}$]
When $g = 0$, for gradient descent, the descent property of $f$ yields 
\beq
\begin{aligned}
f(x_{k+1}) - f(\xsol) 
&\leq \pa{ \tfrac{L}{2} - \tfrac{1}{\eta} } \|x_{k+1} - x_k \|^2 + f(x_k) - f(\xsol) ,
\end{aligned}
\eeq
where $\eta \leq 2/L$. 
For stochastic gradient descent, we obtain the following relationship:
\beq
\label{eq:framework1}
\begin{aligned}
&\Ek [ f(x_{k+1}) - f(\xsol) ] \\
&= \Ek \ca{ f(x_{k+1}) - f(x_k) + f(x_k) - f(\xsol) } \\
&\le \Ek \ca{ \langle \nabla f(x_k) - {\tnablak{k}} , x_{k+1} - x_k \rangle } + \pa{ \tfrac{L}{2} - \tfrac{1}{\eta} } \Ek \ca{ \|x_{k+1} - x_k \|^2 } + f(x_k) - f(\xsol) \\
&\le \tfrac{\epsilon}{2} \Ek \ca{ \| \nabla f(x_k) - \tnablak{k} \|^2 } + \pa{ \tfrac{L}{2} + \tfrac{1}{2 \epsilon} - \tfrac{1}{\eta} } \Ek \ca{ \|x_{k+1} - x_k \|^2 } + f(x_k) - f(\xsol) .
\end{aligned}
\eeq
Compared to \eqref{eq:framework}, there is no inner product term in \eqref{eq:framework1}, which makes the analysis of the non-composite case much simpler. This is one reason why biased algorithms have been successfully studied in non-composite setting, but not in the composite setting.
\end{remark}

\subsection{Non-convex case}

The influence of bias is simpler in the non-convex setting and independent of $g$, which explains why biased algorithms have recently found success for these problems. To begin, let $\hat{x}_{k+1} \defeq \prox_{\eta / 2 g} \pa{ x_k - \eta / 2 \nabla f(x_k) }$. Applying Lemma \ref{lem:noncon} with $z = \hat{x}_{k+1}$, $y = x = x_k$ and $d = \nabla f(x_k)$, we have
\begin{equation}
    \textstyle F(\hat{x}_{k+1}) \le F(x_k) + ( \frac{L}{2} - \frac{1}{\eta} ) \|\hat{x}_{k+1} - x_k\|^2.
\end{equation}
Again, applying Lemma \ref{lem:noncon} with $z = x_{k+1}$, $y = \hat{x}_{k+1}$, $x = x_k$, and $d = \tnabla_k$, we obtain
\beqs
\begin{aligned}
F(x_{k+1}) &\le \textstyle  F(\hat{x}_{k+1}) + \langle \nabla f(x_k) - \tnabla_k, x_{k+1} - \hat{x}_{k+1} \rangle + ( \frac{L}{2} - \frac{1}{2 \eta} ) \|x_{k+1} - x_k\|^2 \\
&\textstyle \qquad + (\frac{L}{2} + \frac{1}{2 \eta} ) \|\hat{x}_{k+1} - x_k\|^2
\end{aligned}
\eeqs
Adding these two inequalities together gives
\beq\label{eq:saganoncon}
\begin{aligned}
F(x_{k+1}) 
&\le F(x_k) + \pa{ L - \tfrac{1}{2 \eta} } \|\hat{x}_{k+1} - x_k\|^2 + ( \tfrac{L}{2} - \tfrac{1}{2 \eta} ) \|x_{k+1} - x_k\|^2 \\
&\qquad + \langle \nabla f(x_k) - \tnablak{k}, x_{k+1} - \hat{x}_{k+1} \rangle \\
&\symnum{1}{\le} F(x_k) + \pa{ L - \tfrac{1}{2 \eta} } \|\hat{x}_{k+1} - x_k\|^2 + \pa{ \tfrac{L}{2} - \tfrac{1}{2 \eta} } \|x_{k+1} - x_k\|^2 + 2 \eta \| \nabla f(x_k) - \tnablak{k} \|^2\\
& \quad + \frac{1}{8 \eta} \|\hat{x}_{k+1} - x_{k+1}\|^2 \\
&\symnum{2}{\le} F(x_k) + \pa{ L - \tfrac{1}{4 \eta} } \|\hat{x}_{k+1} - x_k\|^2 + \pa{ \tfrac{L}{2} - \tfrac{1}{4 \eta} } \|x_{k+1} - x_k\|^2 + 2 \eta \| \nabla f(x_k) - \tnablak{k} \|^2.
\end{aligned}
\eeq
Inequality \numcirc{1} is Young's, and \numcirc{2} is the standard inequality $\|a - c\|^2 \le 2 \|a - b\|^2 + 2 \|b - c\|^2$. In the non-convex case, the inner-product bias term does not appear, so the bias-variance tradeoff is the classical one.

\subsection{General bounds on bias and variance}

To ensure convergence for a particular gradient estimator used in Algorithm \ref{alg:general_sgd}, we must bound the inner-product bias term $\Ek \ca{ \langle \nabla f(x_k) - \tnablak{k}, x_k - \xsol \rangle } $ and the MSE $\Ek \ca{ \| \nabla f(x_k) - \tnablak{k} \|^2 }$. Below we introduce general bounds on these terms that allow us to establish convergence rates for a variety of gradient estimators. The first of these is a bound on the MSE term.

\begin{definition}[Bounded MSE]\label{def:bmse}
The stochastic gradient estimator $\tnabla$ is said to satisfy the \emph{BMSE}$(M_1,M_2,\rho_M,$ $\rho_F,m)$ property with parameters $M_1,M_2 \ge 0$, $\rho_M,\rho_F \in (0,1]$ and $m \ge 1$ if there exist sequences $\calM_k$ and $\calF_k$ such that
\beq
\sum_{k = m s}^{m (s+1) - 1} \E \ca{ \|\tnablak{k} - \nabla f(x_k)\|^2 } 
\le \calM_{m s} ,
\eeq
and the following bounds hold: 
\beq\label{eq:mseb}
\begin{aligned}
\calM_{m s} 
&\le (1-\rho_M)^m \calM_{m (s-1)} + \calF_{m s} + \tfrac{M_1}{n} \sum_{k = m s}^{m(s+1)-1} \sum_{i=1}^n \E \ca{ \|\nabla f_i (x_{k+1}) - \nabla f_i(x_k)\|^2 } , \\
\calF_{m s} 
&\le \sum_{\ell = 0}^s \tfrac{M_2 (1-\rho_F)^{m (s - \ell)} }{n} \sum_{k = m s}^{m (s + 1)-1} \sum_{i=1}^n \E \ca{ \|\nabla f_i (x_{k+1}) - \nabla f_i(x_k)\|^2 } .
\end{aligned}
\eeq
\end{definition}

The constant $m$ is the epoch length of the gradient estimator, hence it is usually set to be $\calO\pa{n}$. The BMSE property allows these bounds to hold only on average over an epoch. 
This property is useful in convergence analyses because it bounds the MSE by a geometrically decaying sequence $\seq{\calM_{m k}}$ and a component that is proportional to the one-iteration progress of gradient descent $(1 / n \sum_{i=1}^n \|\nabla f_i(x_{k+1}) - \nabla f_i(x_k)\|^2)$.

\begin{remark}$~$
\begin{itemize}[itemsep=-1pt,topsep=2pt]

\item Most variance-reduced stochastic gradient estimators satisfy the BMSE property, including SAG, SAGA, SVRG, SARAH, and all the estimators in \cite{neighbors}. SGD does not satisfy this property, as its variance does not decay along the iterations.

\item Most existing work on the analysis of general stochastic gradient algorithms enforce bounds of this form on either the MSE or the moments of the stochastic estimator, with the crucial difference that existing works require the bounds to (i.e., dependent on only the previous iteration) \cite{nocedalreview}. In contrast, the BMSE property allows non-Markovian MSE bounds through the sequence $\calF_k$. This relaxation is crucial for the analysis of our new gradient estimator, SARGE.
\end{itemize}
\end{remark}

In order to bound the inner-product bias term, we require the gradient estimator to admit a certain structure in its bias. In biased estimators such as SAG, the bias depends on the stored gradient values:
\beq
\textstyle \nabla f(x_k) - \Ek \ca{ \gradSAG_{k} } = \pa{1 - \tfrac{1}{n}} \Pa{ \nabla f(x_k) - \tfrac{1}{n} \sum_{i=1}^n \nabla f_i(\varphi_k^i) }.
\eeq
We call estimators whose bias admits the above structure \emph{memory-biased} gradient estimators. These include SAG, and more generally B-SAGA and B-SVRG.

\begin{definition}[Memory-biased gradient estimator]
The stochastic gradient estimator $\tnabla$ is \emph{memory-biased} with parameters $\theta > 0$, $B_1 \ge 0$, and $m \ge 1$ if
\beq
\nabla f(x_k) - \Ek \ca{ \tnablak{k} } 
= \pa{ 1 - \tfrac{1}{\theta} } \bPa{ \nabla f(x_k) - \sfrac{1}{n} \sum_{i=1}^n \nabla f_i(\varphi_k^i)},
\eeq
for some $\{\varphi_k^i\}_{i=1}^n \subset \{ x_\ell \}_{\ell=0}^{k-1}$, and for any $s \in \mathbb{N}_0$,
\beq
\label{eq:membias}
\sum_{k = m s}^{m (s+1) - 1} \tfrac{1}{n} \sum_{i=1}^n \E \ca{ \|x_k - \varphi_k^i\|^2 } 
\le B_1 \sum_{k = m s}^{m (s+1)-1} \E \ca{ \| x_k - x_{k-1} \|^2 } .
\eeq
\end{definition}

B-SAGA is clearly a memory-biased estimator, and so is B-SVRG where $\varphi_k^i = \varphi_{m s}^i$ for all $k$ in epoch $s$. The parameter $\theta$ controls the amount of bias in the estimator, and $B_1$, in a sense, measures how ``stale'' the stored gradient information is. For memory-biased gradient estimators, the bias-term can be handled easily.
\begin{lemma}
\label{lem:couple1main}
Suppose $\tnabla$ is memory-biased with parameter $\theta \ge 1$ and that $F$ is $\mu$-strongly convex with $\mu \ge 0$. For any $\lambda > 0$, the following inequality holds:
\beqs
\begin{aligned}
\eta \Ek [ F(x_{k+1}) - F(\xsol) ] 
&\le \tfrac{\eta}{2 L \lambda} \Ek \ca{ \|\tnablak{k} - \nabla f(x_k) \|^2 } - \tfrac{1 + \mu \eta}{2} \Ek \ca{ \|x_{k+1} - \xsol\|^2 } + \tfrac{1}{2} \|x_k - \xsol\|^2 \\
&\qquad + \pa{ \tfrac{\eta L ( \lambda + 1 )}{2} - \tfrac{1}{2} } \Ek \ca{ \|x_{k+1} - x_k\|^2 } + \tfrac{\eta L}{2 n} \pa{ 1 - \tfrac{1}{\theta} } \sum_{i=1}^n \|x_k - \varphi_k^i\|^2.
\end{aligned}
\eeqs
\end{lemma}
The proof of Lemma \ref{lem:couple1main} can be found in Appendix \ref{sec:proof_mem}. The bound of Lemma \ref{lem:couple1main} is analogous to the bound in \eqref{eq:framework}, but the inner-product bias term is replaced with $\tfrac{\eta L}{2 n} \pa{ 1 - \tfrac{1}{\theta} } \sum_{i=1}^n \|x_k - \varphi_k^i\|^2$. This term is proportional to the progress of gradient descent (by \eqref{eq:membias}), so this provides the necessary control over the inner-product bias term.

\vskip1em

For estimators such as SARAH, the bias depends on the error in the previous gradient estimate, rather than previous stochastic gradients:
\beq
\nabla f(x_k) - \Ek \ca{ \gradSARAH_k } = \nabla f(x_{k-1}) - \gradSARAH_{k-1}.
\eeq
We refer to estimators of this type as \emph{recursively biased}.

\begin{definition}[Recursively biased gradient estimator]\label{def:recursively_biased}
For any sequence $\{x_k\}$, let $\tnablak{k}$ be a stochastic gradient estimator generated from the points $\{x_{\ell}\}_{\ell = 0}^k$. This estimator is \emph{recursively biased} with parameters $\rho_B \in (0,1]$ and $\nu \ge 1$ if
\beq
\nabla f(x_k) - \Ek \ca{ \tnablak{k} } = \begin{cases}
0 & \textnormal{for } k \in \nu \mathbb{N}_0, \\
( 1 - \rho_B ) \pa{ \nabla f(x_{k-1}) - \tnabla_{k-1} } & \textnormal{o.w.}
\end{cases}
\eeq
\end{definition}
The parameter $\nu$ represents how many steps occur between full gradient evaluations. For SARGE, $\nu = \infty$ because the full gradient is never computed.

\begin{lemma}
\label{lem:bias}
Suppose $\tnabla$ is a recursively biased gradient estimator with parameters $\nu \ge 1$ and $\rho_B \in (0,1]$. Then, for any $\epsilon > 0$,
\beqs\label{thm1:bias}
\begin{aligned}
&\sum_{k = \nu s + 1}^{\nu (s + 1) - 1} | \E \langle \nabla f(x_{k-1}) - \tnabla_{k-1}, x_k - \xsol \rangle | \\
&\le \min \bBa{ \nu, \tfrac{1}{\rho_B} } \sum_{k = \nu s}^{\nu (s + 1) - 1} \E \Ca{ \tfrac{\epsilon}{2} \| \nabla f(x_{k+1}) - \tnabla_{k+1} \|^2 + \tfrac{1}{2 \epsilon} \|x_{k+1} - x_k\|^2 } .
\end{aligned}
\eeqs
\end{lemma}
Lemma \ref{lem:bias} shows that, for recursively biased estimators, the inner-product bias term $\langle \nabla f(x_{k-1}) - \tnabla_{k-1}, x_k - \xsol \rangle$ is bounded from above by the MSE, implying that introducing bias to decrease the MSE is a reasonable approach to design improved gradient estimators.

\section{Convergence rates}\label{sec:rates}

In this section, we analyse the convergence rates for the stochastic gradient methods. We first provide very general convergence rates based on the bounds from the last section. Then, we specify the result to specific gradient estimators including memory-biased B-SAGA/B-SVRG, and recursively biased SARAH and SARGE.

\subsection{General convergence rates}

For Algorithm \ref{alg:general_sgd}, we consider a constant step size $\eta_k \equiv \eta > 0$. 
Given $T$ iterations of Algorithm \ref{alg:general_sgd}, define the average iterate $\xbar_{T} \defeq 1 / T \sum_{k=1}^T x_k$.

\subsubsection{Convex and strongly convex cases}

The following theorem establishes convergence rates for memory-biased estimators in the convex regime.

\begin{theorem}[{Memory-biased estimators}]\label{thm:rate_mem}
Let $\tnabla$ be a memory-biased gradient estimator parameterized by $\theta \geq 1$ and $B_1 \geq 0$, which satisfies the BMSE$(M_1,M_2,\rho_M,\rho_F,m)$ property. 
Let $\Theta = \frac{M_1 \rho_F + 2 M_2}{\rho_M \rho_F}$ and $\rho = \min\{ \rho_M, \rho_F \}$.
\begin{itemize}[itemsep=-1pt,topsep=2pt]
\item When $F$ is convex, let $\eta = \frac{1}{L (1 + 3 \sqrt{2 \Theta} )} $, then
 \beq
 \E [ F(\xbar_{T}) - F(\xsol) ] 
 \le \sfrac{1}{T} \left( \tfrac{L (1 + 3 \sqrt{2 \Theta} ) \|x_0 - \xsol\|^2}{2} + \max \bBa{ \tfrac{B_1 (1 - 1/\theta)}{\sqrt{2 \Theta}} - 1, 0 } \tfrac{ { F(x_0) - F(\xsol) }}{L (1 + 3 \sqrt{2 \Theta} ) } \right).
 \eeq
 
\item When $F$ is $\mu$-strongly convex with $\mu > 0$, let $\eta = \min \Ba{ \tfrac{1}{3 L (1 + 3 \sqrt{2 \Theta})}, \tfrac{\sqrt{2 \Theta}}{B_1 \mu (1 - 1/\theta)}, \tfrac{\rho}{2 \mu} }$.
 The iterate $x_T$ satisfies
 \beq
 \E \ca{ \|x_T - \xsol\|^2 } 
 \le (1 + \mu \eta )^{-T} \pa{ \tfrac{2}{\mu} \pa{ F(x_0) - F(\xsol) } + \|x_0 - \xsol\|^2 }.
 \eeq
\end{itemize}
\end{theorem}

The proof of Theorem \ref{thm:rate_mem} is provided in Appendix \ref{sec:proof_mem}. 
The next result establishes convergence rates for recursively biased gradient estimators whose proof is in Appendix \ref{sec:proof_recur}.

\begin{theorem}[{Recursively biased estimators}]\label{thm:rate_recurs}
Let $\tnabla$ be a recursively biased gradient estimator parameterized by $\rho_B \in (0,1)$ and $\nu \ge 1$, which satisfies the BMSE$(M_1,M_2,\rho_M,\rho_F,m)$ property. 
Let $B_2 \defeq \min\left\{ \nu, 1/\rho_B \right\}$, $\Theta = \frac{M_1 \rho_F + 2 M_2}{\rho_M \rho_F}$ and $\rho = \min\{\rho_M, \rho_F\}$.
\begin{itemize}[itemsep=-1pt,topsep=2pt]
\item When $F$ is convex, let $\eta = \tfrac{1}{L ( 4 \sqrt{2 \Theta} + 1 )}$, then 
 \beq
 \begin{aligned}
 \E [ F(\xbar_{T}) - F(\xsol) ] 
 \le \sfrac{1}{T} \left( \tfrac{L (4 \sqrt{2 \Theta} + 1)}{2} \|x_0 - \xsol\|^2 + \max\Ba{ (1-\rho_B) B_2 - 1, 0 } \tfrac{ F(x_0) - F(\xsol) }{L (4 \sqrt{2 \Theta} + 1)} \right) .
 \end{aligned}
 \eeq
 
\item When $F$ is $\mu$-strongly convex with $\mu > 0$, let $\eta = \min \Ba{ \frac{1}{3 L (4 \sqrt{2 \Theta} + 1)}, \frac{1}{\mu (1 - \rho_B) B_2}, \frac{\rho}{2 \mu} }$, then
\beq
\E \ca{ \|x_T - \xsol\|^2 } 
\le (1 + \mu \eta)^{-T} \pa{ \tfrac{2}{\mu} \pa{ F(x_0) - F(\xsol) } + \|x_0 - \xsol\|^2 } .
\eeq

\end{itemize}
\end{theorem}

\begin{remark} $~$
\begin{itemize}[itemsep=-1pt,topsep=2pt]
\item Both theorems hold true for smaller $\eta$; the choices in the theorems are the largest ones allowed by our analysis.

\item For B-SAGA and B-SVRG, $\Theta = \mathcal{O}(n^2)$, while for SARAH and SARGE, $\Theta = \mathcal{O}(n)$. This gives these recursive gradient estimators improved convergence rates and suggests that the bias in these estimators is more effective than the bias in SAGA and SVRG.
\end{itemize}
\end{remark}

\subsubsection{Non-convex case}

The analysis of biased gradient estimators is simpler for the non-convex setting than the convex ones due to the absence of the inner-product bias term in \eqref{eq:saganoncon}. Below we provide a uniform convergence guarantee for all gradient estimators satisfying the BMSE property, regardless of their bias. This suggests that in the non-convex setting, a large-bias, small-MSE gradient estimator is favourable over an estimator with small bias and large MSE.

\begin{theorem}
\label{thm:rate_ncvx}
Let $\tnabla$ be a gradient estimator that satisfies the BMSE$(M_1,M_2,\rho_M,\rho_F,m)$ property, let $\Theta = \frac{M_1 \rho_F + 2 M_2}{\rho_M \rho_F}$, and let $\alpha$ be a chosen uniformly at random from the set $\{0, 1, \cdots, T-1\}$. If $F$ is non-convex, set $\eta = \frac{\sqrt{16 \Theta + 1} - 1}{16 L \Theta}$ in Algorithm \ref{alg:general_sgd}, and the point $x_\alpha$ satisfies the following bound on its generalized gradient: 
\beq
\E \ca{ \|\mathcal{G}_{\eta / 2} (x_\alpha)\|^2 } 
\le \tfrac{16 ( F(x_0) - F(\xsol) )}{T \eta ( 1 - 4 \eta L ) }.
\eeq
\end{theorem}
The proof of this result is provided in Appendix \ref{sec:proof_ncvx}.

\begin{remark}
The convergence result of Theorem \ref{thm:rate_ncvx} does not depend on the bias except through the MSE of the gradient estimator, which implies that incorporating arbitrary amounts of bias for a smaller MSE improves the convergence rate. This fact is what allows the recursively biased estimators SARAH and SARGE to achieve the oracle complexity lower bound for non-convex optimisation when they are used in Algorithm \ref{alg:general_sgd}.
\end{remark}

\subsection{Convergence rates for specific gradient estimators}

In this section, we specialise the general convergence rates to analyse the performance of B-SAGA, B-SVRG, SARAH, and SARGE.

\subsubsection{Biased SAGA and SVRG}

B-SAGA and B-SVRG are examples of memory-biased gradient estimators, as their biases take the form
\beq
\nabla f(x_k) - \Ek \ca{ \tnablak{k} } 
= \pa{ 1 - \tfrac{1}{\theta} } \bPa{ \nabla f(x_k) - \tfrac{1}{n} \sum_{i=1}^n \nabla f_i(\varphi_k^i) },
\eeq
for some previous iterates $\varphi_k^i$. To establish convergence rates for B-SAGA and B-SVRG, we only need to show these estimators satisfy the BMSE property with suitable constants.
\begin{lemma}
\label{lem:sagamse}
The B-SAGA gradient estimator is memory-biased with $B_1 = 2 n (2 n + 1)$, and it satisfies the BMSE property with parameters $\rho_M = \frac{1}{2 n}$, $m = 1$, $M_2 = 0$, $\rho_F = 1$, and
\beq
M_1 = \begin{cases}
\frac{2 n + 1}{\theta^2} & \theta \in (0,2], \\
(2 n + 1) (1 - \frac{1}{\theta})^2 & \theta > 2.
\end{cases}
\eeq
\end{lemma}
The proof of Lemma \ref{lem:sagamse} uses a slight modification of existing variance bounds for the SAGA estimator, appearing in \cite{SAGA}, for example. We include the proof in Appendix \ref{sec:saga}. The B-SVRG gradient estimator satisfies the BMSE property with similar constants.
\begin{lemma}
\label{lem:svrgmse}
The B-SVRG gradient estimator is memory-biased with $B_1 = 3 m (m + 1)$, and it satisfies the BMSE property with parameters $\rho_M = 1$, $M_2 = 0$, $\rho_F = 1$, and
\beq
M_1 = \begin{cases}
\frac{3 m (m+1)}{\theta^2} & \theta \in (0,2], \\
3 m (m+1) \pa{ 1 - \tfrac{1}{\theta} }^2 & \theta > 2.
\end{cases}
\eeq
\end{lemma}
With these constants established, Theorem \ref{thm:rate_mem} provides rates of convergence.\footnotemark

\begin{corollary}[{Convergence rates for B-SAGA}]
\label{thm:saga}
\footnotetext{{We state the convergence rates without constants for simplicity. The complete result with constants is included in Appendix \ref{sec:saga}.}}Algorithm \ref{alg:general_sgd} achieves the following convergence guarantees using the B-SAGA gradient estimator:
\begin{itemize}[itemsep=-1pt,topsep=2pt]
\item 
If $F$ is convex, depending on the choice of $\theta$, set the step size to
\[
\eta
= \eta_{\theta} 
\defeq 
\left\{
\begin{aligned}
\textstyle \frac{1}{L (1 + \frac{6}{\theta} \sqrt{n (2 n + 1)})} &: \theta \in [1,2] , \\
\textstyle \frac{1}{L (1 + 6 {( 1 - \frac{1}{\theta}) }\sqrt{n ( 2 n + 1 )}) } &: \theta > 2,
\end{aligned}
\right.
\] 
and $\xbar_T$ satisfies $\E [ F(\xbar_{T}) - F(\xsol) ] = \calO(L n /T)  $.

\item
If $F$ is $\mu$-strongly convex, set {$\eta = \min \Ba{ \eta_{\theta}, \frac{1}{4 \mu n} }$}. Then $x_T$ satisfies $\E \ca{ \|x_T - \xsol\|^2 } = \calO( \pa{ 1 + \mu \eta }^{-T} ) $. 

\item 
If $F$ is non-convex, after $T$ iterations, the generalized gradient at $x_\alpha$ satisfies
\beq
\E \ca{ \|\mathcal{G}_{\eta / 2}(x_\alpha)\|^2 }
= 
\left\{
\begin{aligned}
\textstyle \mathcal{O} \left( \frac{L n}{T \theta} \right) &: \textstyle \eta = \frac{\theta}{2 L \sqrt{n (2 n + 1)}}, \ \theta \in (0,2], \\
\textstyle \mathcal{O} \left( \frac{L n}{T {( 1 - \frac{1}{\theta}) } } \right) &: \textstyle \eta = \frac{1}{2 L {( 1 - \frac{1}{\theta}) } \sqrt{n (2 n + 1) }}, \ \theta > 2.
\end{aligned}
\right.
\eeq

\end{itemize}
\end{corollary}

\begin{corollary}[{Convergence rates of B-SVRG}]
\label{thm:svrg}
Algorithm \ref{alg:general_sgd} achieves the following convergence guarantees using the B-SVRG gradient estimator:
\begin{itemize}[itemsep=-1pt,topsep=2pt]
\item 
When $F$ is convex, depending on the choice of $\theta$, set the step size to
\[
\eta
= \eta_{\theta} 
= 
\left\{
\begin{aligned}
\textstyle \frac{1}{L ( 1 + \frac{3}{\theta} \sqrt{6 m ( m + 1 )} )} &: \theta \in [1,2] , \\
\textstyle \frac{1}{L ( 1 + 3 ( 1 - \frac{1}{\theta})\sqrt{6 m ( m + 1 ) } )} &: \theta > 2 .
\end{aligned}
\right.
\] 
After $S$ epochs, the point $\xbar_{m S}$ satisfies $\E [ F(\xbar_{m S}) - F(\xsol) ] = \calO(L / S)  $.

\item If, moreover, $F$ is $\mu$-strongly convex, let $\eta = \min \{ \eta_{\theta}, \frac{1}{2 \mu} \}$. After $S$ epochs, $x_{m S}$ satisfies $\E \ca{ \|x_{m S} - \xsol\|^2 } = \calO( \pa{ 1 + \mu \eta }^{-m S} ) $.

\item
If $F$ is non-convex, after $S$ epochs, the generalized gradient at $x_\alpha$ satisfies
\beq
\E \ca{ \|\mathcal{G}_{\eta / 2}(x_\alpha)\|^2 }
= 
\left\{
\begin{aligned}
\textstyle \mathcal{O} \left( \frac{L m}{T \theta} \right) &: \textstyle \eta = \frac{\sqrt{2} \theta}{2 L \sqrt{3 m (m + 1)}}, \ \theta \in (0,2], \\
\textstyle \mathcal{O} \big( \frac{L m}{T {( 1 - {1}/{\theta}) } } \big) &: \textstyle \eta = \frac{\sqrt{2} \theta}{2 L {( 1 - \frac{1}{\theta}) } \sqrt{3 m (m + 1) }}, \ \theta > 2.
\end{aligned}
\right.
\eeq

\end{itemize}
\end{corollary}

\begin{remark} $~$
\begin{itemize}[itemsep=-1pt,topsep=2pt]
\item Our MSE bounds and convergence rates are optimised when $\theta = 2$. Numerical experiments (including those in Section \ref{sec:experiments}) suggest that setting $\theta$ in the range $1 < \theta \ll n$ gives the best performance, and B-SAGA prefers larger values of $\theta$ than B-SVRG.
\item
In the special case $\theta = 1$, Corollaries \ref{thm:saga} and \ref{thm:svrg} recover the state-of-the-art rates for SAGA and SVRG in the convex and non-convex regimes. For strongly convex problems, these rates are worse than existing convergence rates of $\mathcal{O} ( (1 + \min \left\{ \frac{\mu}{L}, \frac{1}{n} \right\} )^{-T} )$ proven for SAGA and SVRG \cite{SAGA,proxSVRG}. This difference is due to the generality of Theorem \ref{thm:rate_mem}, as some memory-biased estimators, including B-SVRG, exhibit poor performance on strongly convex problems when the bias is large.
\item 
Corollaries \ref{thm:saga} and \ref{thm:svrg} require step sizes that decrease with $n$, while existing results for SAG, SAGA, and SVRG allow step sizes that are independent of $n$. This is also due to the generality of Theorem \ref{thm:rate_mem}. For example, we find in practice that B-SAGA converges with step sizes that are independent of $n$, but B-SVRG requires smaller step sizes when the epoch length is larger.
\end{itemize}
\end{remark}

\subsubsection{SARAH and SARGE}

The SARAH and SARGE gradient estimators are recursively biased, with
\beq
\nabla f(x_k) - \Ek \ca{\gradSARAH_k} = \nabla f(x_{k-1}) - \gradSARAH_{k-1} .
\eeq
and
\beq
\textstyle \nabla f(x_k) - \Ek \ca{\gradSARGE_k} = (1 - \frac{1}{n} ) ( \nabla f(x_{k-1}) - \gradSARGE_{k-1} ).
\eeq
As we shall see, these biased estimators admit smaller MSE bounds than unbiased and memory-biased estimators, and this is reflected in their improved convergence rates. The following two lemmas establish the constants appearing in Theorem \ref{thm:rate_recurs} for these estimators.

\begin{lemma}
\label{lem:sarahmse}
The SARAH gradient estimator is recursively biased with parameters $\rho_B = 0$ and $\nu = m$, and it satisfies the BMSE property with parameters $M_1 = m$, $\rho_M = 1$, $\rho_F = 1$, and $M_2 = 0$.
\end{lemma}

\begin{lemma}
\label{lem:sargemse}
The SARGE gradient estimator is recursively biased with parameters $\rho_B = 1 / n$ and $\nu = \infty$, and it satisfies the BMSE property with $M_1 = 12$, $M_2 = 39 / n$, $\rho_M = \frac{1}{4 n}$, $\rho_F = \frac{1}{2 n}$, and $m = 1$.
\end{lemma}

\noindent Proofs of these results are included in Appendices \ref{sec:sarah} and \ref{sec:sarge}, respectively. It is enlightening to compare these BMSE constants to those of B-SVRG and B-SAGA. $M_1$ is a factor of $n$ smaller for the SARAH and SARGE estimators than for the B-SVRG and B-SAGA estimators (as long as $m = \mathcal{O}(n)$ in SARAH and B-SVRG). This translates to an $\mathcal{O}(\sqrt{n})$ improvement in the convergence rates for SARAH and SARGE.

\begin{corollary}[{Convergence rates for SARAH}]
\label{thm:sarah}
When using the SARAH gradient estimator in Algorithm \ref{alg:general_sgd}, 
\begin{itemize}[itemsep=-1pt,topsep=2pt]
\item If $F$ is convex, set $\eta = \tfrac{1}{L (4 \sqrt{2 m} + 1)}$. After $T$ iterations, $\xbar_T$ satisfies $\E [ F(\xbar_{T}) - F(\xsol) ] = \calO(L \sqrt{m} / T)$. 

\item
If $F$ is $\mu$-strongly convex, set $\eta = \min \{ \frac{1}{3 L (4 \sqrt{2 m} + 1)}, \frac{1}{\mu m} \}$,
then $\E \ca{ \|x_T - \xsol\|^2 } = \calO( \pa{ 1 + \mu \eta }^{-T} ) $.

\item 
If $F$ is non-convex, set $\eta = \frac{1}{L \sqrt{2 m}}$, then $\E \ca{ \|\mathcal{G}_{\eta / 2}(x_\alpha)\|^2 } \le \mathcal{O} \left( L \sqrt{m} / T \right)$
\end{itemize}
\end{corollary}

\begin{corollary}[{Convergence rates for SARGE}]\label{thm:sarge}
When using the SARGE gradient estimator in Algorithm \ref{alg:general_sgd},
\begin{itemize}[itemsep=-1pt,topsep=2pt]
\item If $F$ is convex, set $\eta = \tfrac{1}{L (16 \sqrt{3 (n + 13)} + 1)}$, then $\E [ F(\xbar_{T}) - F(\xsol) ] = \calO(L \sqrt{n} / T) $. 

\item
If $F$ is $\mu$-strongly convex, set $\eta = \min \{ \frac{1}{3 L (16 \sqrt{3 (n + 13)} + 1)}, \frac{1}{4 \mu n} \}$, then $\E \ca{ \|x_T - \xsol\|^2 } = \calO( \pa{ 1 + \mu \eta }^{-T} ) $.

\item
If $F$ is non-convex, set $\eta = \frac{1}{4 L \sqrt{3 (n+13)}}$, then $\E \ca{ \|\mathcal{G}(x_\alpha)\|^2 } \le \mathcal{O} \left( L \sqrt{n} / T \right)$.

\end{itemize}
\end{corollary}

These convergence rates for convex objectives represent a significant improvement over the performance of SAGA, SVRG, and full-gradient methods. Each of these algorithms require $\mathcal{O}(\frac{n L}{\epsilon})$ stochastic gradient evaluations to find a point satisfying $F(x_T) - F(\xsol) \le \epsilon$, while SARAH and SARGE require only $\mathcal{O}( \frac{\sqrt{n} L}{\epsilon} )$. These rates do not require the epoch-doubling procedure of \cite{svrgplusplus}, although epoch-doubling can potentially be used to improve the performance of SARAH just as it improves the performance of SVRG on non-strongly convex objectives.

This square-root dependence on $n$ is present in the convergence rates for strongly convex and non-convex objectives as well, which is a significant improvement over the dependence on $n$ in the convergence rates of B-SAGA and B-SVRG. This better dependence on $n$ is most significant in the non-convex regime, where these convergence rates imply that the SARAH and SARGE gradient estimators require only $\mathcal{O}( \frac{\sqrt{n} L}{\epsilon} )$ stochastic gradient evaluations to find an $\epsilon$-approximate stationary point, which is the oracle-complexity lower bound \cite{spider}. Similar results already exist for algorithms using the SARAH estimator \cite{spider,spiderm,spiderboost,proxsarah}. Our results for SARGE show that achieving this complexity is possible without ever computing the full gradient.

\section{Numerical Experiments}\label{sec:experiments}

In this section, we present numerical experiments testing B-SAGA, B-SVRG, SARAH, and SARGE for minimizing convex, strongly convex, and non-convex objectives. We include one set of experiments comparing different values of $\theta$ in B-SAGA and B-SVRG with a fixed step size and one set comparing SARAH and SARGE to B-SAGA and B-SVRG with the best values of $\theta$.

\subsection{Convex and strongly convex objectives}

Let $(h_i, l_i) \in \mathbb{R}^p \times \{\pm 1\} ,~ i=1,\cdots,n$ be the training set, where $h_{i} \in \mathbb{R}^p$ is the feature vector of each data sample, and $l_{i}$ is the binary label. Let $\beta > 0$ be a tuning parameter. The ridge regression problem takes the form
\beq
\min_{x \in \mathbb{R}^p } \quad \tfrac{1}{n} \msum_{i=1}^n ( h_i^\top x - l_i )^2 + \tfrac{\beta}{2} \norm{x}_{2}^{2}.
\eeq
LASSO is similar, but with the regulariser $\norm{x}_1$ replacing $\norm{x}_2^2$. These problems are of the form \eqref{eq:mainopt}, where we set $f_i = (h_i^\top x - l_i)^2$ and $g$ equal to the regulariser. In ridge regression, $g$ is strongly convex, and in LASSO, $g$ is only convex.

We consider four binary classification data sets: \texttt{australian}, \texttt{mushrooms}, \texttt{phishing}, and \texttt{ijcnn1} from LIBSVM\footnote{\url{https://www.csie.ntu.edu.tw/~cjlin/libsvmtools/datasets/}}. We rescale the value of the data to $[-1, 1]$, set $\beta = 1 / n$, and set the step size to $\eta = \frac{1}{5 L}$. 
To compare performance, we use the objective function value $F(x_k) - F(\xsol)$ is considered.

\paragraph{Comparison of B-SAGA}
We first compare the performance of B-SAGA under different choices of $\theta$ for solving ridge regression and LASSO problems. 
Four choices of $\theta$ are considered: $\theta \in \{1,10,100,n\}$, the results are provided below in Figures \ref{fig:saga-ridge} and \ref{fig:saga-lasso}, from which we observe that B-SAGA consistently performs better with moderate amounts of bias (i.e. $\theta \in (1,n)$). For the considered datasets, overall $\theta = 10$ provides the best performance.

\begin{figure}[!ht]
\centering
\subfloat[\texttt{australian}]{ \includegraphics[width=0.235\linewidth]{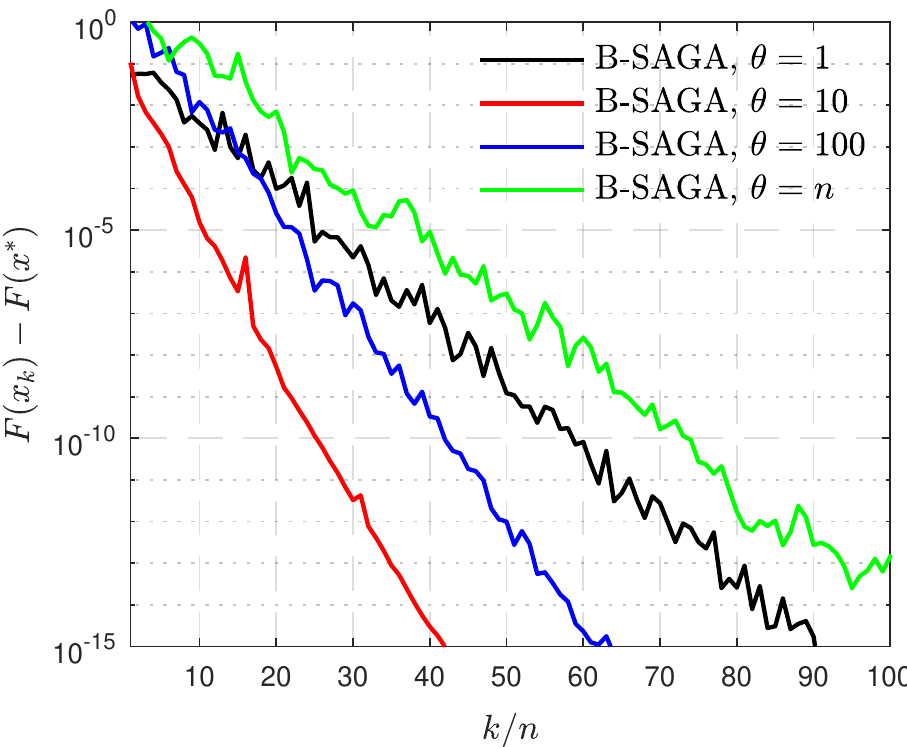} }  %\hspace{1pt}
\subfloat[\texttt{mushrooms}]{ \includegraphics[width=0.235\linewidth]{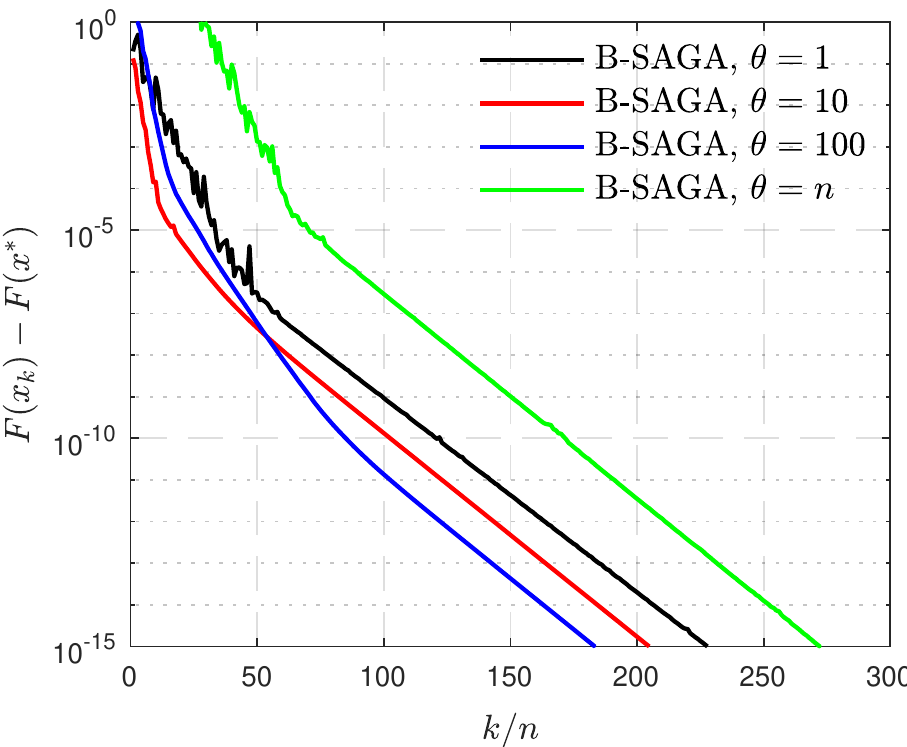} } 
%\hspace{1pt}
\subfloat[\texttt{phishing}]{ \includegraphics[width=0.235\linewidth]{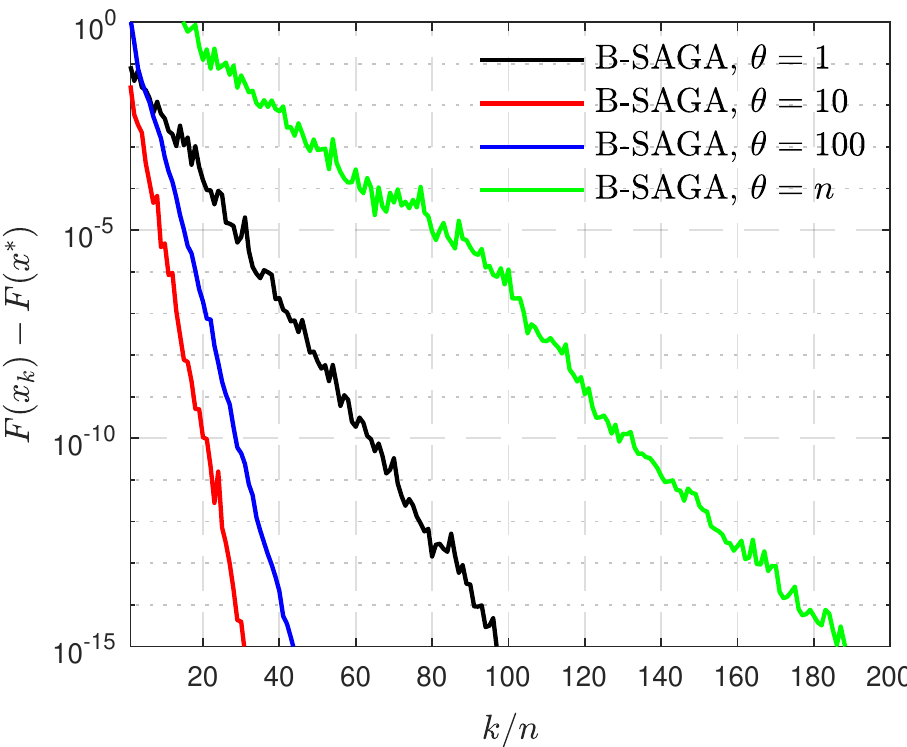} }   
%\hspace{1pt}
\subfloat[\texttt{ijcnn1}]{ \includegraphics[width=0.235\linewidth]{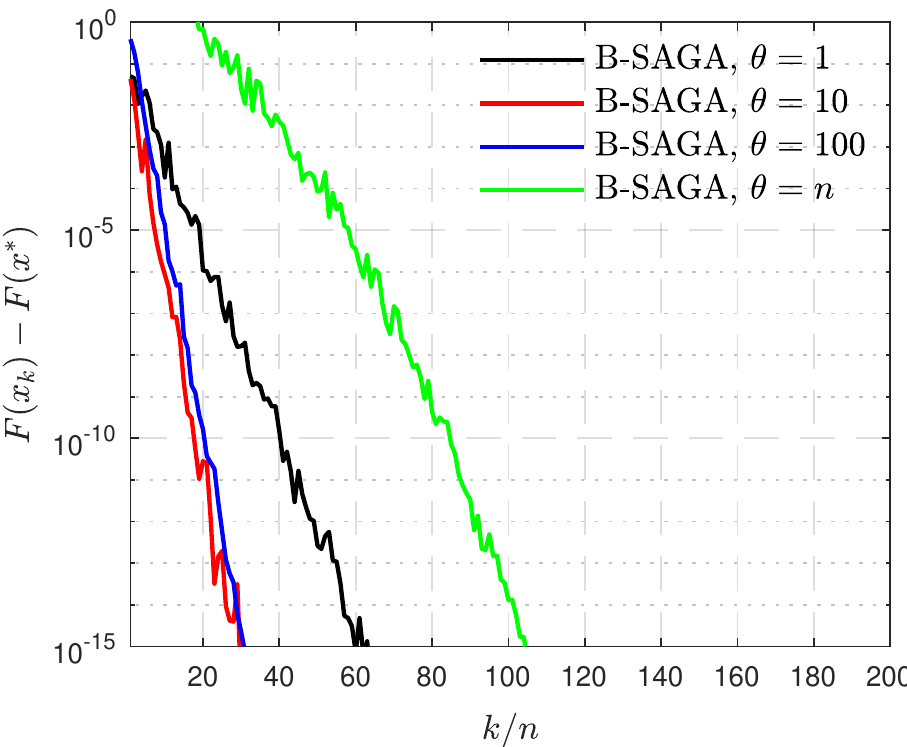} }  \\
%%%%%%%%
\caption{
Performance comparison fitting a ridge regression model for different choices of $\theta$ in B-SAGA. The step size for each case is set to $\eta = \frac{1}{5 L}$. 
}
\label{fig:saga-ridge}
\end{figure}

\begin{figure}[!ht]
\centering
\subfloat[\texttt{australian}]{ \includegraphics[width=0.235\linewidth]{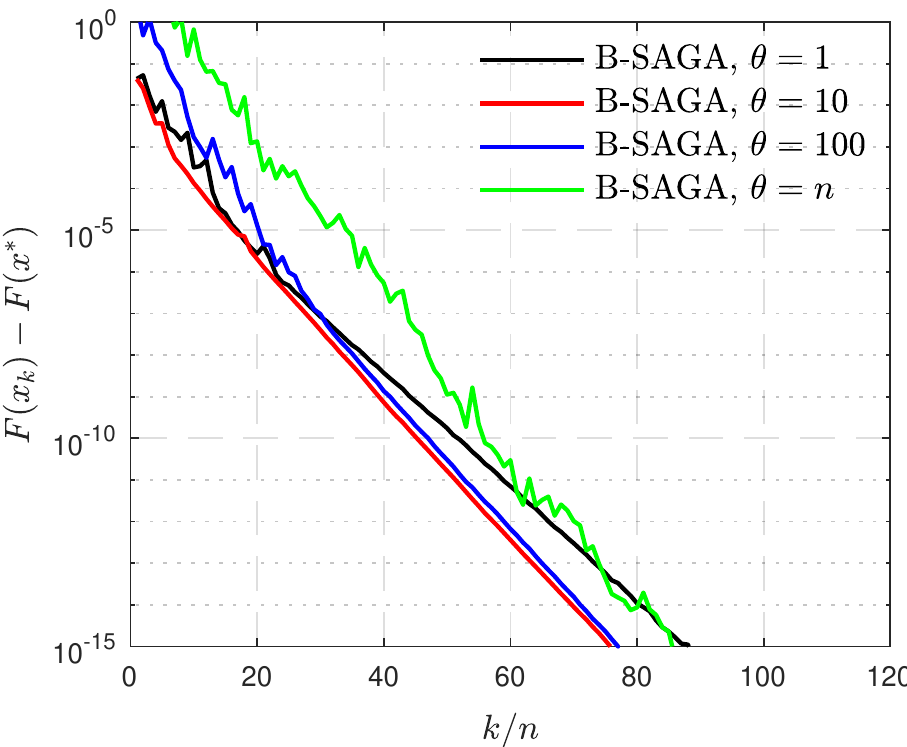} }  %\hspace{1pt}
\subfloat[\texttt{mushrooms}]{ \includegraphics[width=0.235\linewidth]{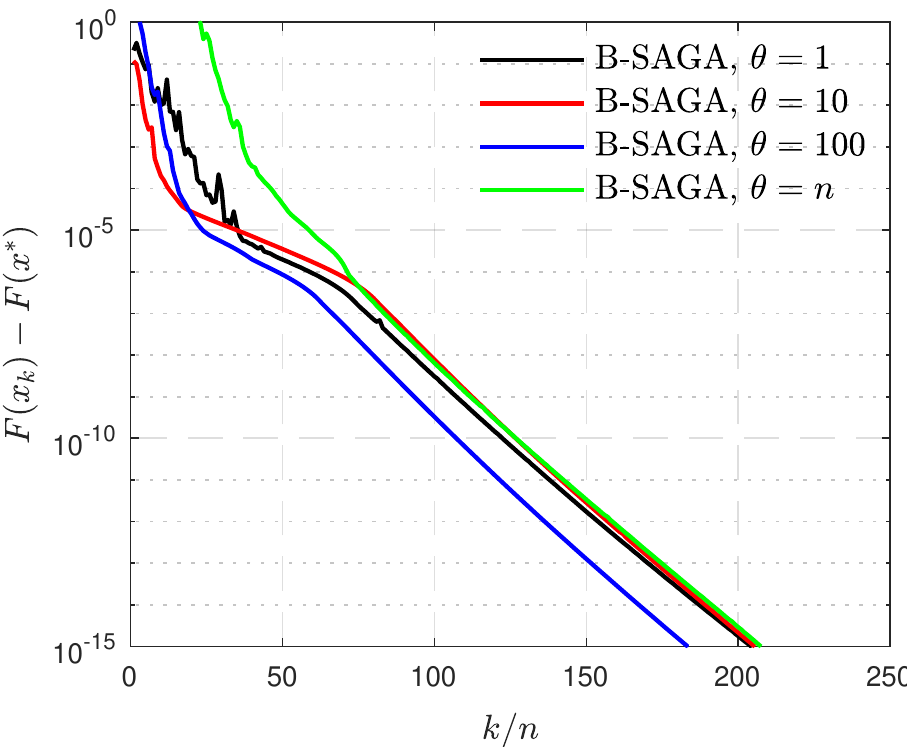} } %\hspace{1pt}
\subfloat[\texttt{phishing}]{ \includegraphics[width=0.235\linewidth]{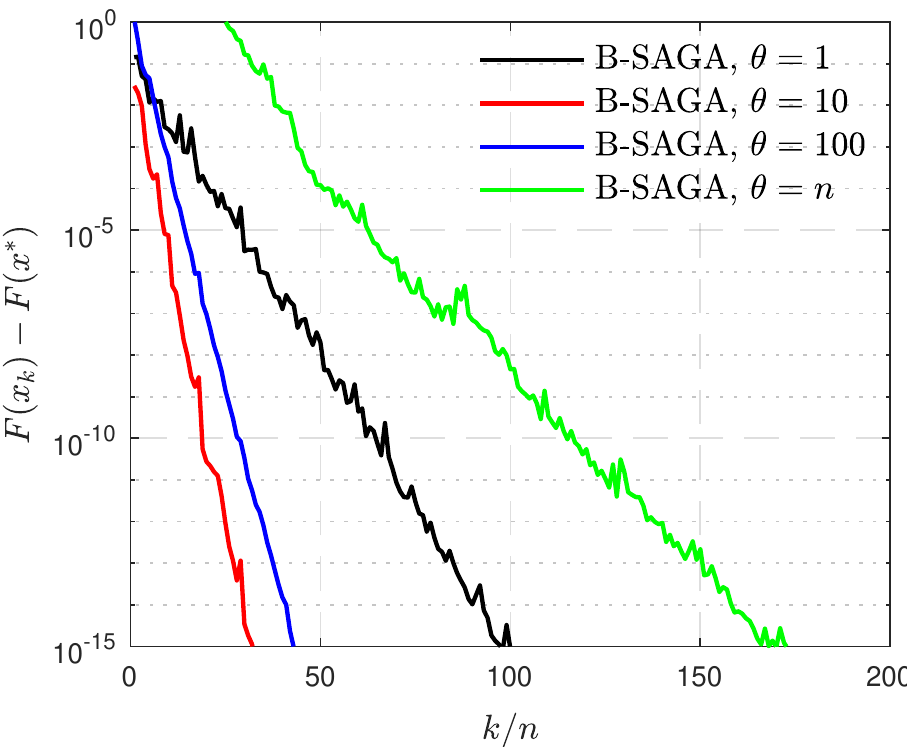} }   
%\hspace{1pt}
\subfloat[\texttt{ijcnn1}]{ \includegraphics[width=0.235\linewidth]{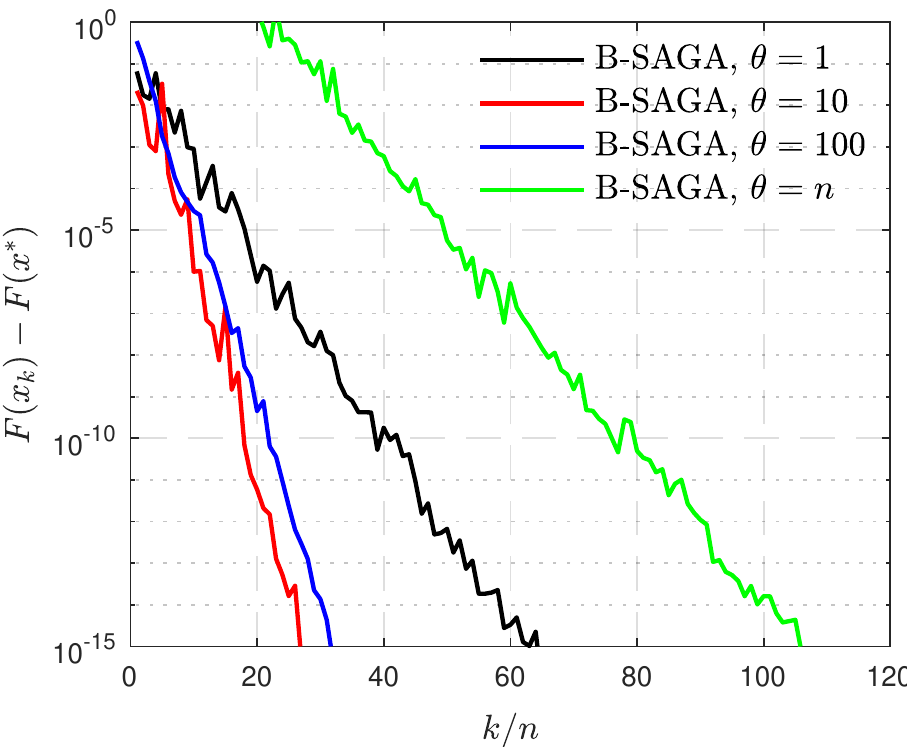} }  \\
%%%%%%%%
\caption{
Performance comparison fitting a LASSO model for different choices of $\theta$ in B-SAGA. The step size for each case is set to $\eta = \frac{1}{5 L}$. 
}
\label{fig:saga-lasso}
\end{figure}

\paragraph{Comparison of B-SVRG}
We also consider four choices of $\theta$ for B-SVRG, which are
$\theta \in \{ 0.5, 0.8, 1, 1.5\}$. The results are shown below in Figure \ref{fig:svrg-ridge} and \ref{fig:svrg-lasso}. We observe that B-SVRG is more sensitive to the choice of $\theta$; only small amounts of bias (i.e. $\theta \in [0.8, 1.5]$) can occasionally improve performance. 

\begin{figure}[!ht]
\centering
\subfloat[\texttt{australian}]{ \includegraphics[width=0.235\linewidth]{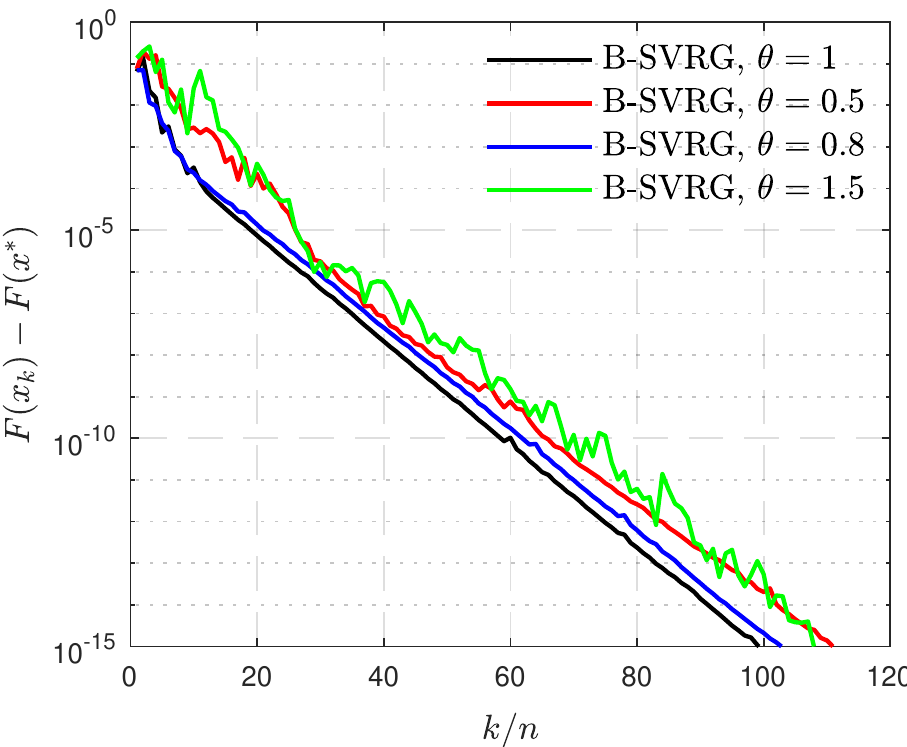} }  %\hspace{1pt}
\subfloat[\texttt{mushrooms}]{ \includegraphics[width=0.235\linewidth]{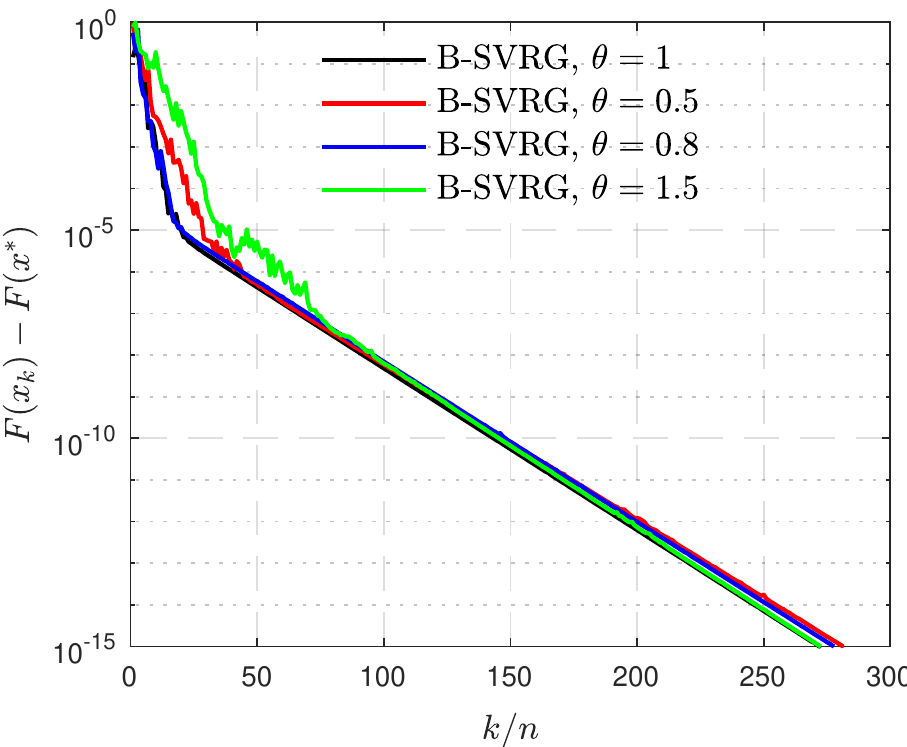} } 
%\hspace{1pt}
\subfloat[\texttt{phishing}]{ \includegraphics[width=0.235\linewidth]{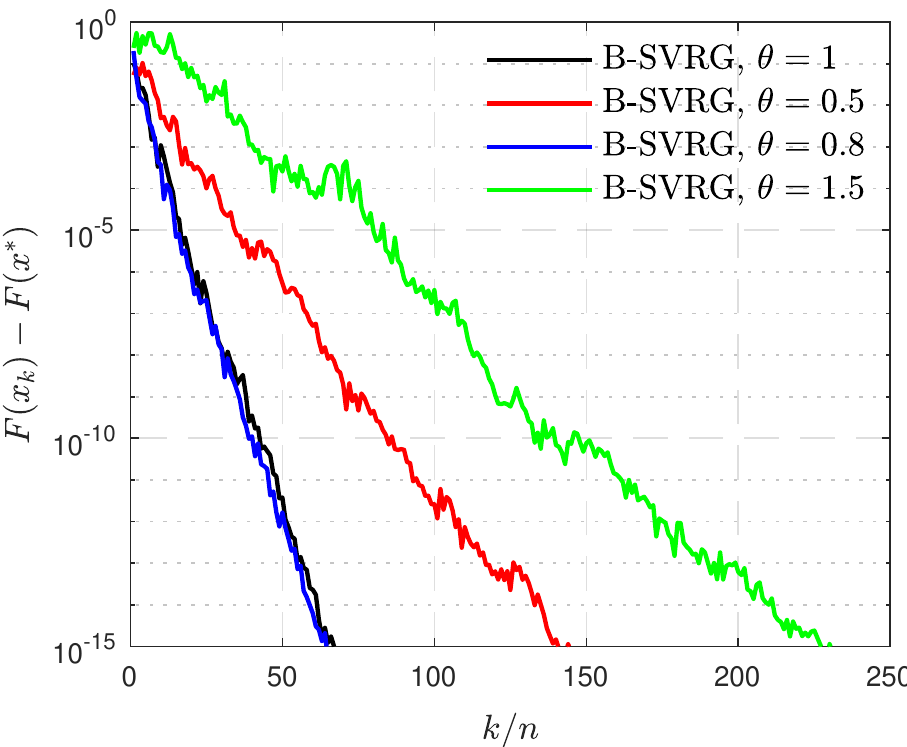} }   
%\hspace{1pt}
\subfloat[\texttt{ijcnn1}]{ \includegraphics[width=0.235\linewidth]{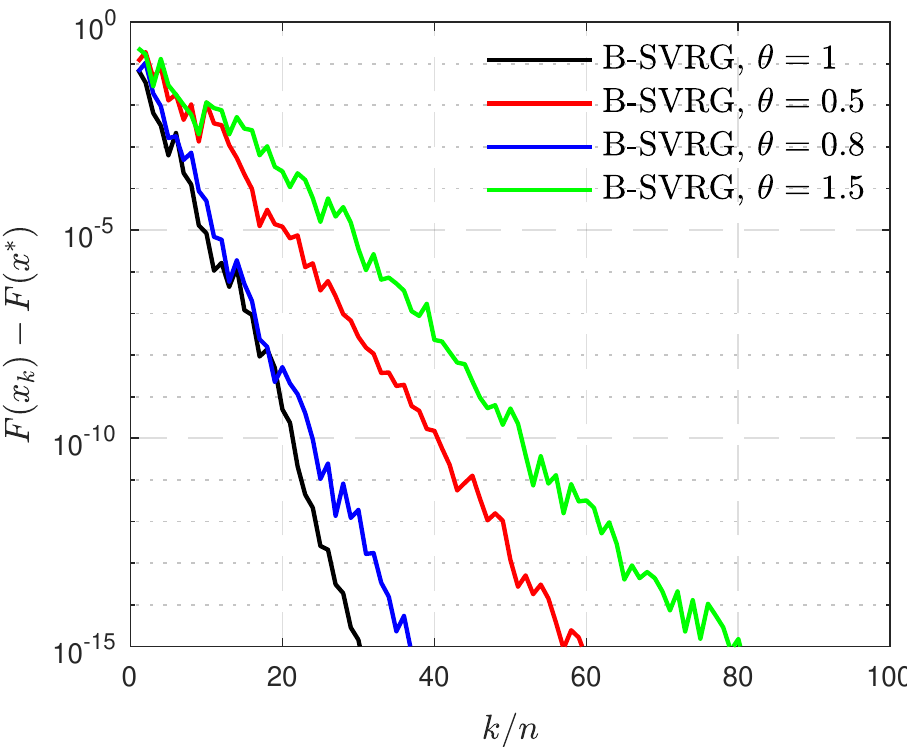} }  \\
%%%%%%%%
\caption{
Performance comparison fitting a ridge regression model for different choices of $\theta$ in B-SVRG. The step size for each case is set to $\eta = \frac{1}{5 L}$. 
}
\label{fig:svrg-ridge}
\end{figure}

\begin{figure}[!ht]
\centering
\subfloat[\texttt{australian}]{ \includegraphics[width=0.235\linewidth]{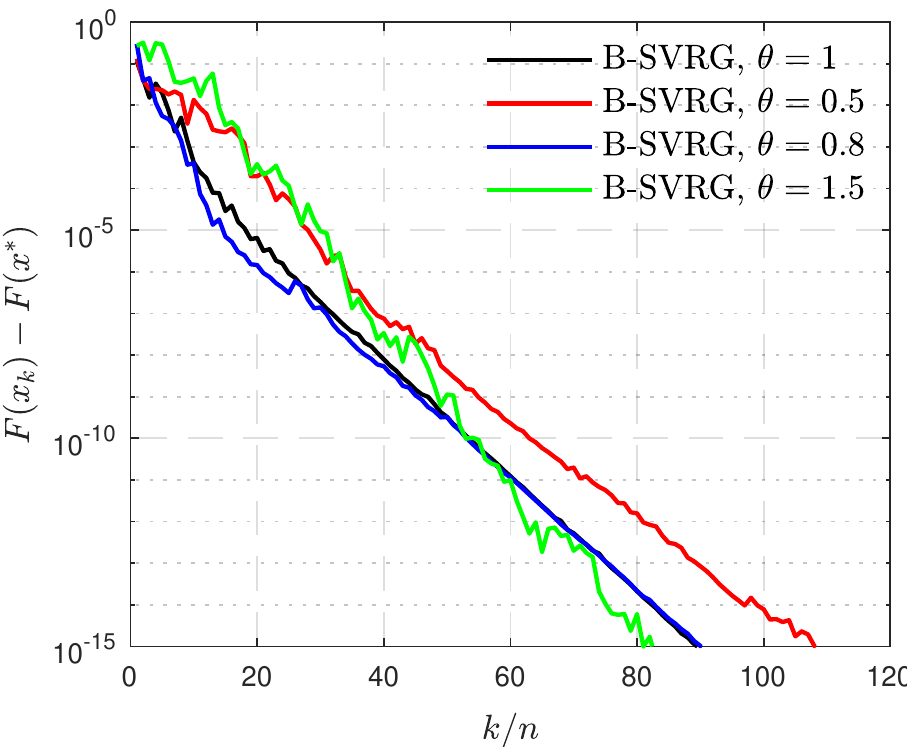} }  %\hspace{1pt}
\subfloat[\texttt{mushrooms}]{ \includegraphics[width=0.235\linewidth]{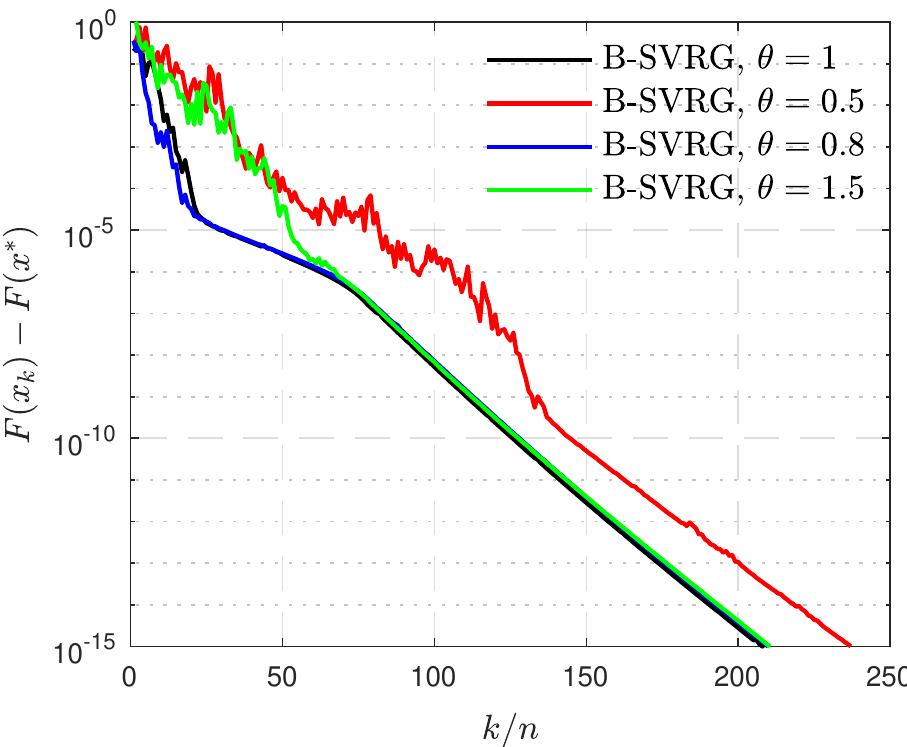} } %\hspace{1pt} 
\subfloat[\texttt{phishing}]{ \includegraphics[width=0.235\linewidth]{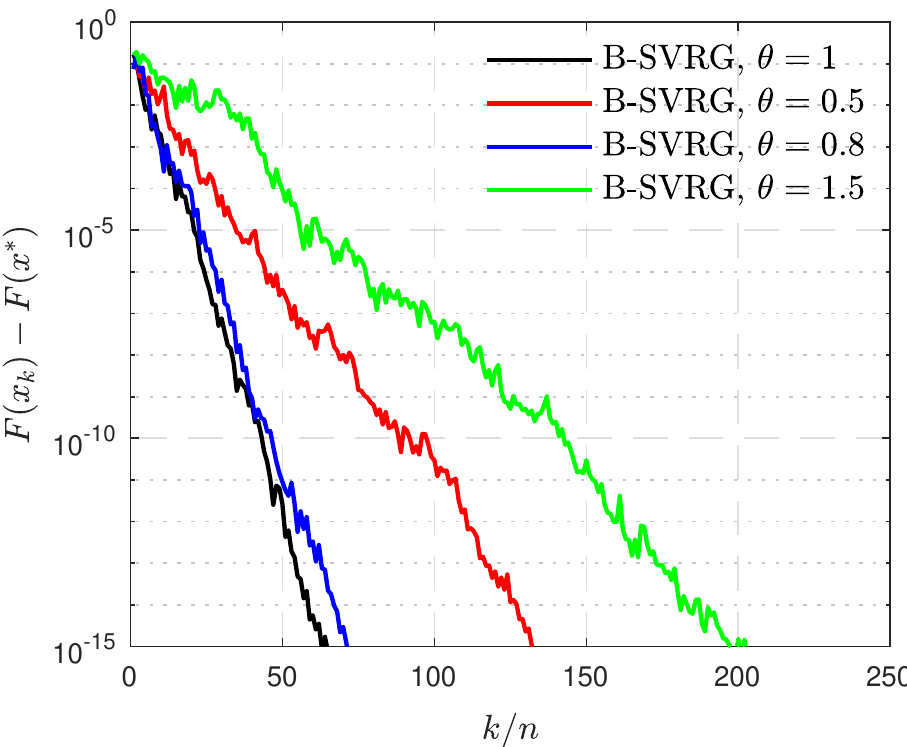} }   
%\hspace{1pt}
\subfloat[\texttt{ijcnn1}]{ \includegraphics[width=0.235\linewidth]{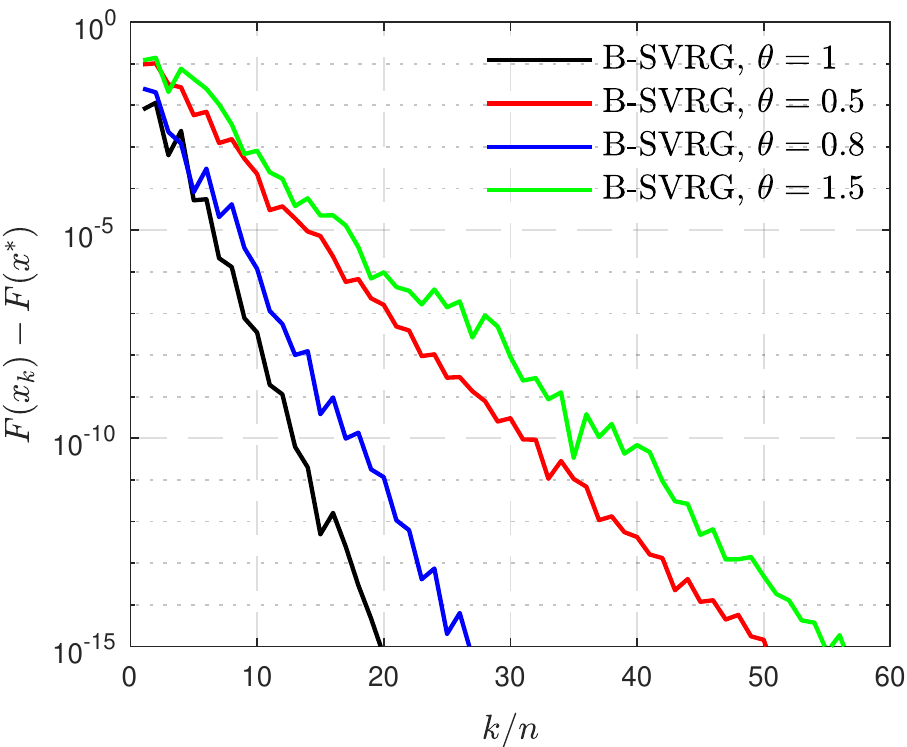} }  \\
%%%%%%%%
\caption{
Performance comparison fitting a LASSO model for different choices of $\theta$ in B-SVRG. The step size for each case is set to $\eta = \frac{1}{5 L}$. 
}
\label{fig:svrg-lasso}
\end{figure}

\paragraph{Comparison of different gradient estimators}
Finally, we provide comparison of SAGA, B-SAGA with $\theta = 10$, SVRG, SARAH and SARGE, the results are provided below in Figure \ref{fig:cmp-ridge} and \ref{fig:cmp-lasso} from which we observe that
\begin{itemize}[itemsep=-1pt,topsep=2pt]
\item SARAH performs similarly to SVRG, but is occasionally slower in early epochs.
\item SARGE consistently outperforms all other methods except for B-SAGA with $\theta = 10$.
\end{itemize}
The above observations indicate that, depending on the data, biased schemes can benefit from their biased gradient estimates. The free parameter $\theta$ reduces the MSE of the B-SAGA and B-SVRG gradient estimators leading to better performance, and the bias in SARAH and SARGE has a similar effect.

\begin{figure}[!ht]
\centering
\subfloat[\texttt{australian}]{ \includegraphics[width=0.235\linewidth]{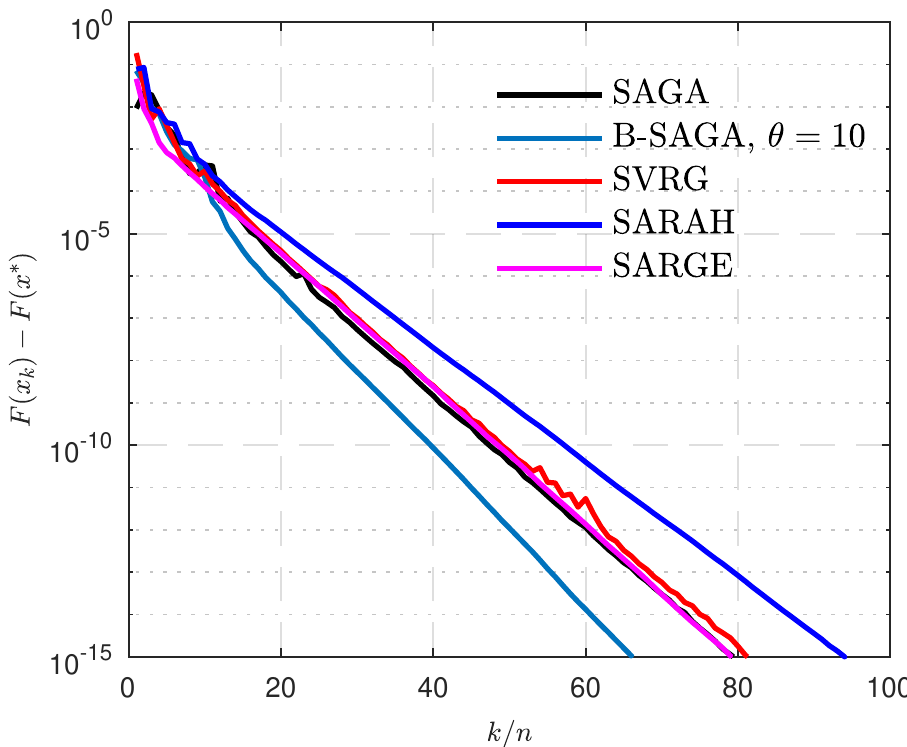} }  %\hspace{1pt}
\subfloat[\texttt{mushrooms}]{ \includegraphics[width=0.235\linewidth]{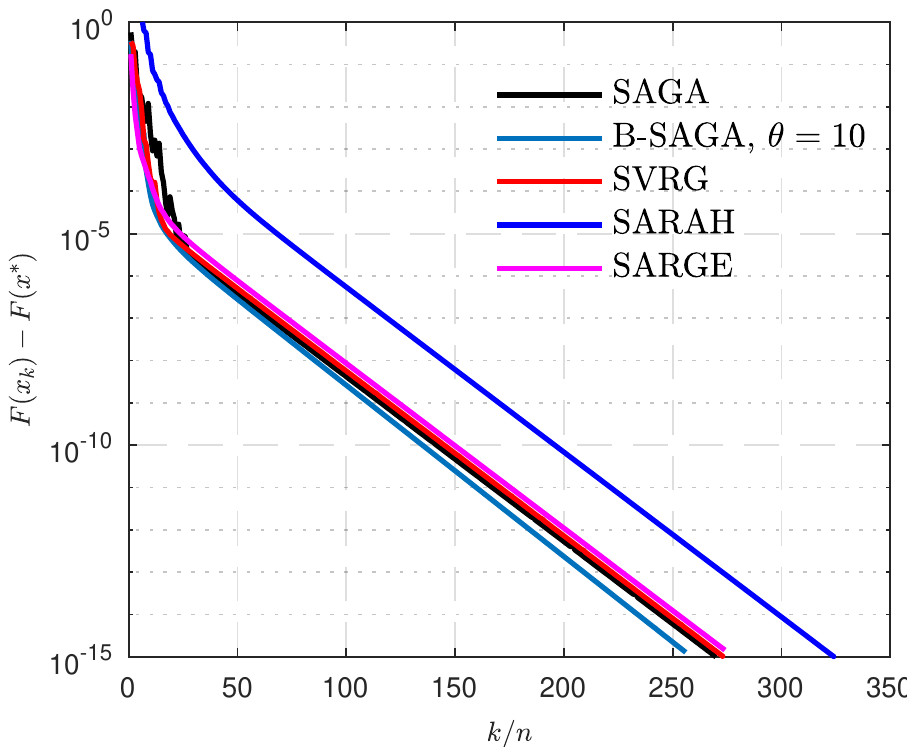} } %\hspace{1pt} 
\subfloat[\texttt{phishing}]{ \includegraphics[width=0.235\linewidth]{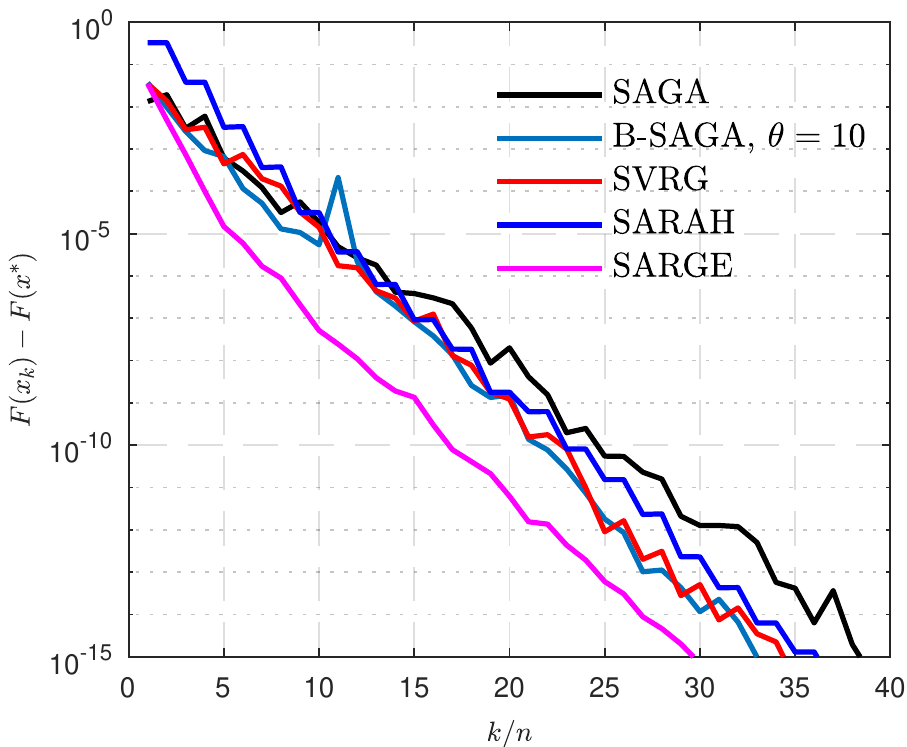} }   
%\hspace{1pt}
\subfloat[\texttt{ijcnn1}]{ \includegraphics[width=0.235\linewidth]{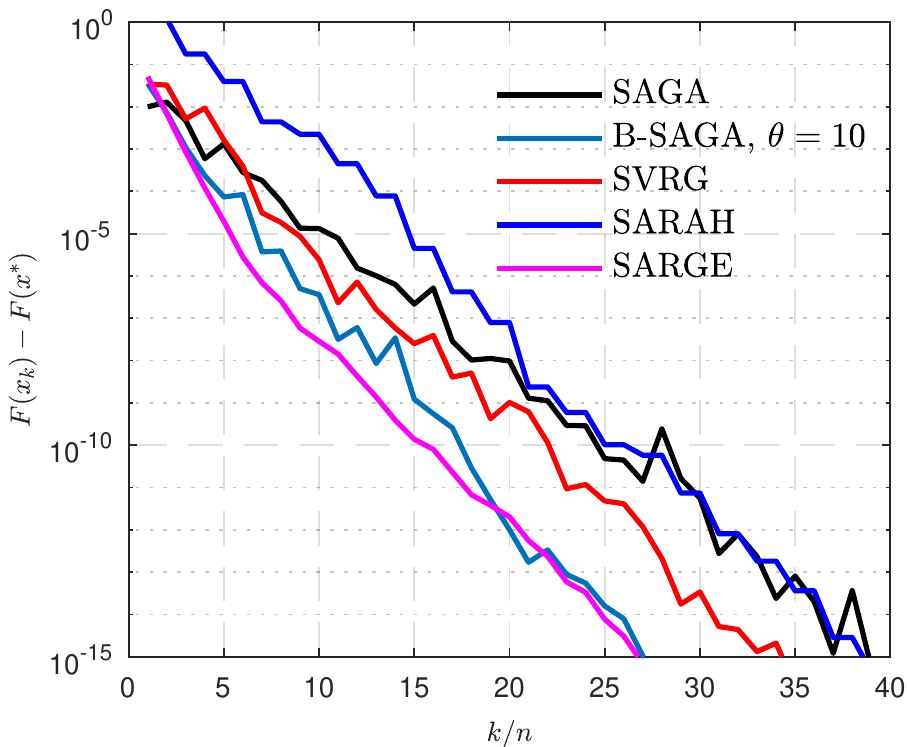} }  \\
%%%%%%%%
\caption{Performance comparison for solving ridge regression among different algorithms. Step sizes are tuned automatically to minimize the number of iterations required to reach a tolerance of $10^{-15}$.}
\label{fig:cmp-ridge}
\end{figure}

\begin{figure}[!ht]
\centering
\subfloat[\texttt{australian}]{ \includegraphics[width=0.235\linewidth]{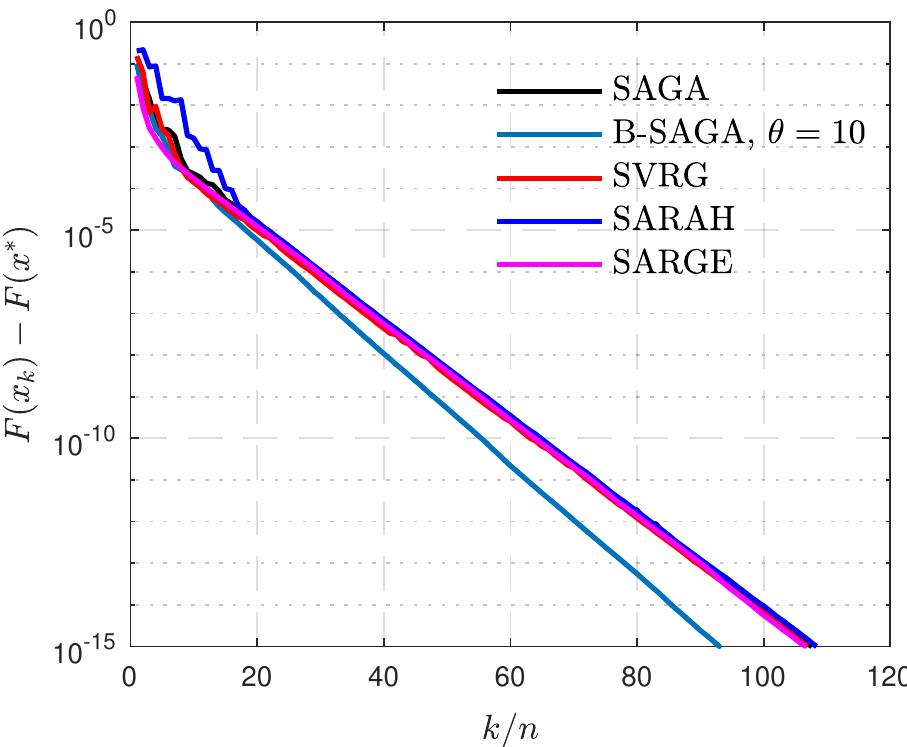} }  %\hspace{1pt}
\subfloat[\texttt{mushrooms}]{ \includegraphics[width=0.235\linewidth]{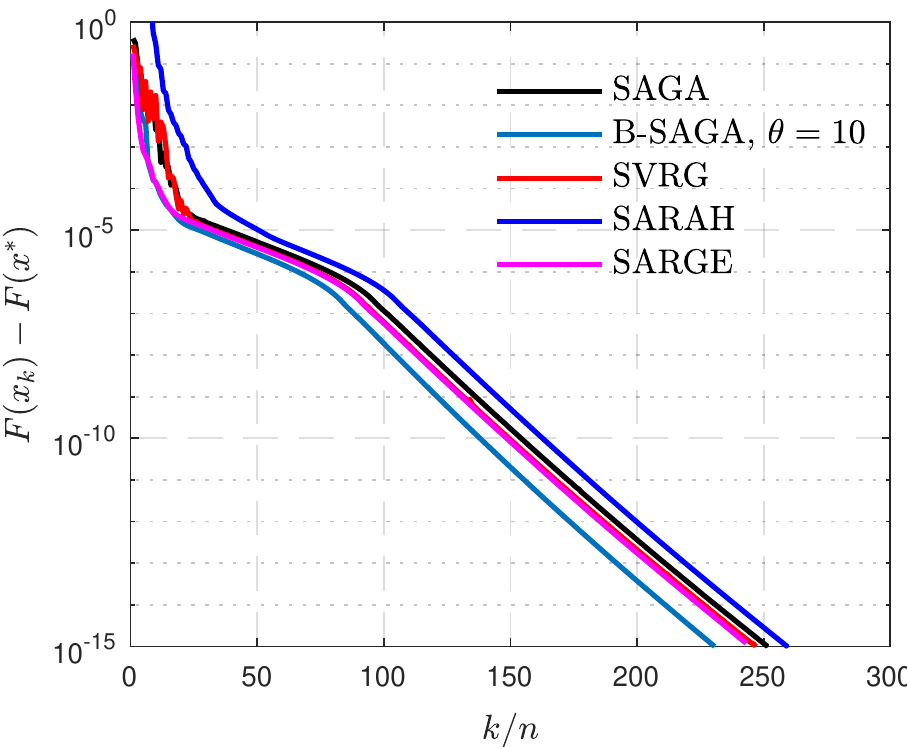} } %\hspace{1pt}
\subfloat[\texttt{phishing}]{ \includegraphics[width=0.235\linewidth]{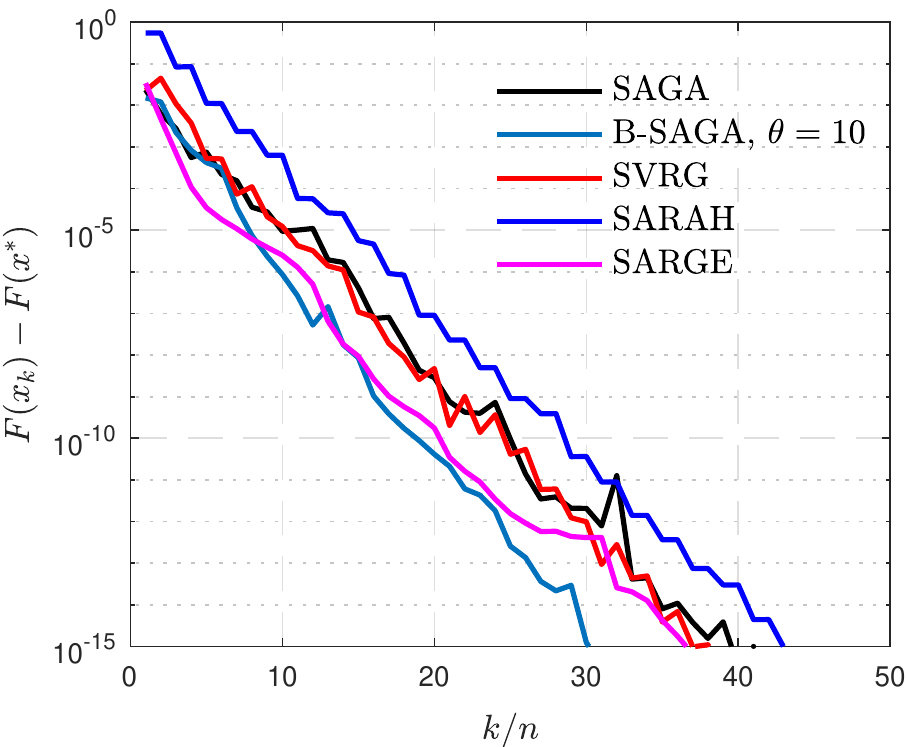} }   
%\hspace{1pt}
\subfloat[\texttt{ijcnn1}]{ \includegraphics[width=0.235\linewidth]{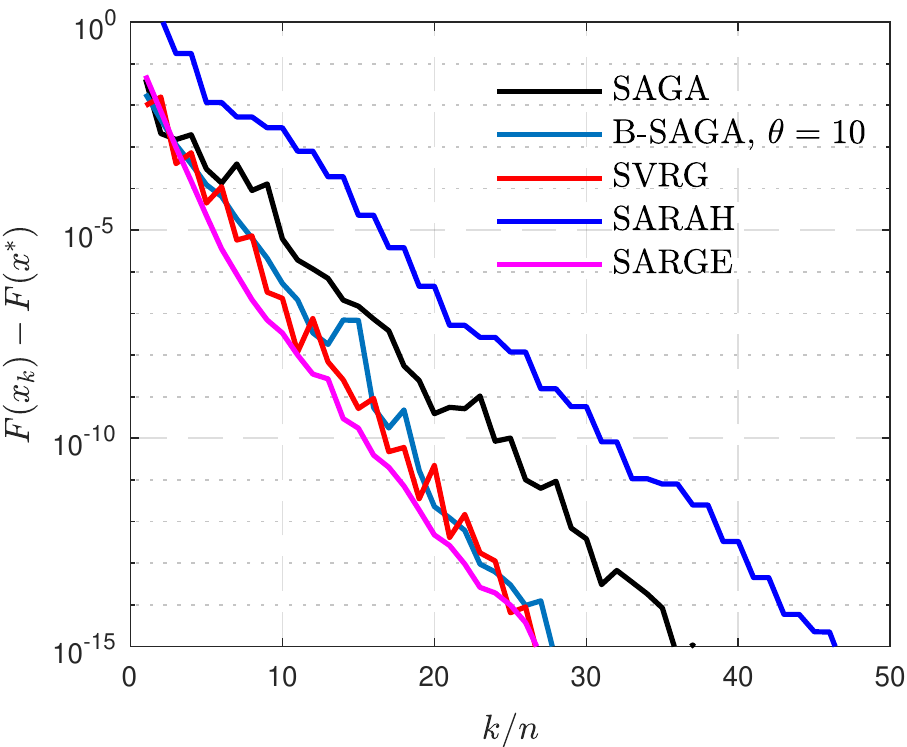} }  \\
%%%%%%%%
\caption{Performance comparison for solving LASSO regression among different algorithms. Step sizes are tuned automatically to minimize the number of iterations required to reach a suboptimality of $10^{-15}$.}
\label{fig:cmp-lasso}
\end{figure}

\subsection{Non-convex objectives}

To test the effect of bias in the non-convex setting, we consider the non-negative principal component analysis (NN-PCA) problems, which can be formulated as \cite{reddiprox}:
\beq
\min_{x \in \mathbb{R}^p } ~~ - \tfrac{1}{n} \msum_{i=1}^n ( h_i^\top x )^2 + \iota_C(x) ,
\eeq
where $C \defeq \{ x \in \R^p : \|x\| \le 1, \ x \ge 0 \}$ is a convex set and 
\begin{equation}
\iota_C(x) = \left\{ \begin{aligned} 0 &: x \in C \\ +\infty &: x \notin C  \end{aligned} \right.
\end{equation}
is the indicator function of $C$. Letting $g = \iota_C$, the operator $\prox_{\eta g}$ is the projection onto $C$, which can be computed efficiently.

As the problem is non-convex, we cannot measure convergence with respect to the global optimum $\xsol$, so we use many iterations of proximal gradient descent with a small step size ($\eta = \tfrac{1}{10 L n}$) to find a reference point $\xsol$. Every test is initialized using a random vector with normally distributed i.i.d. entries, and the same starting point is used for testing each value of $\theta$. We found that small step sizes generally lead to stationary points with smaller objective values, so we set $\eta = \frac{1}{5 n}$ for all our experiments. We report $F(x_k) - F(\xsol)$ averaged over every $n$ iterations. These experiments show that the performance of B-SAGA and B-SVRG varies significantly with $\theta$, with smaller values leading to better performance. SARAH and SARGE perform similarly to SAGA and SVRG in these experiments, see Figure \ref{fig:nnpca-saga} and \ref{fig:nnpca-svrg}.

\begin{figure}[!ht]
\centering
\subfloat[\texttt{australian}]{ \includegraphics[width=0.235\linewidth]{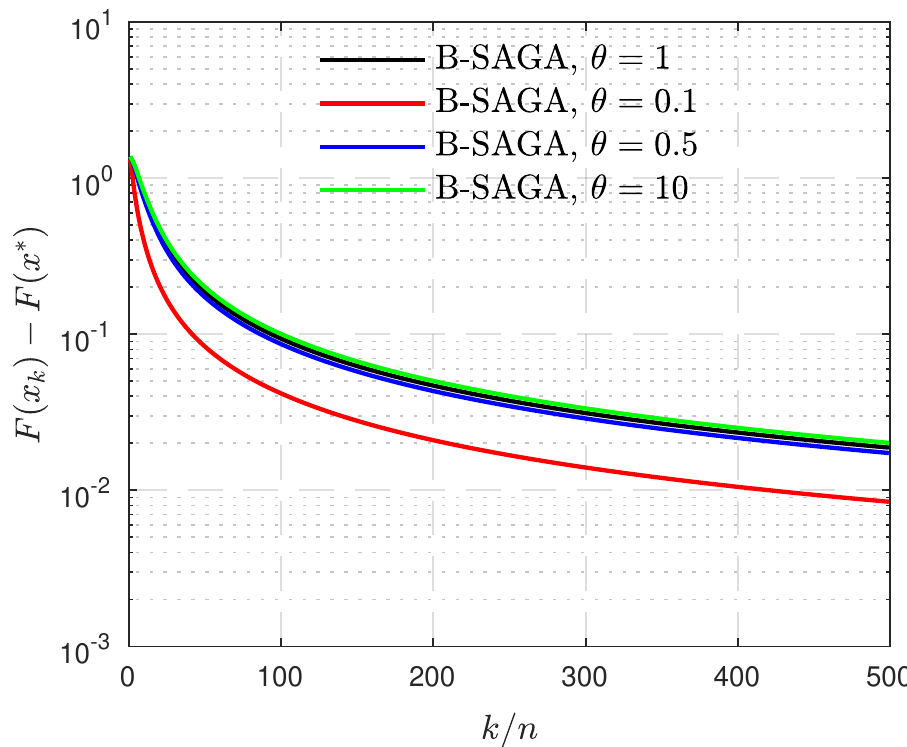} }  %\hspace{1pt}
\subfloat[\texttt{mushrooms}]{ \includegraphics[width=0.235\linewidth]{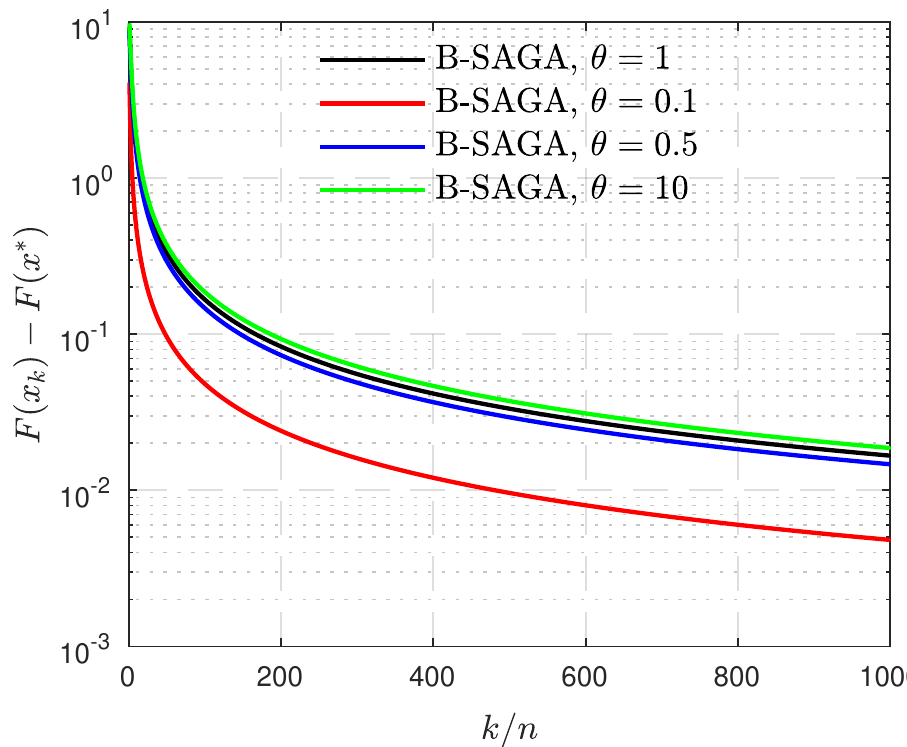} } %\hspace{1pt}
\subfloat[\texttt{phishing}]{ \includegraphics[width=0.235\linewidth]{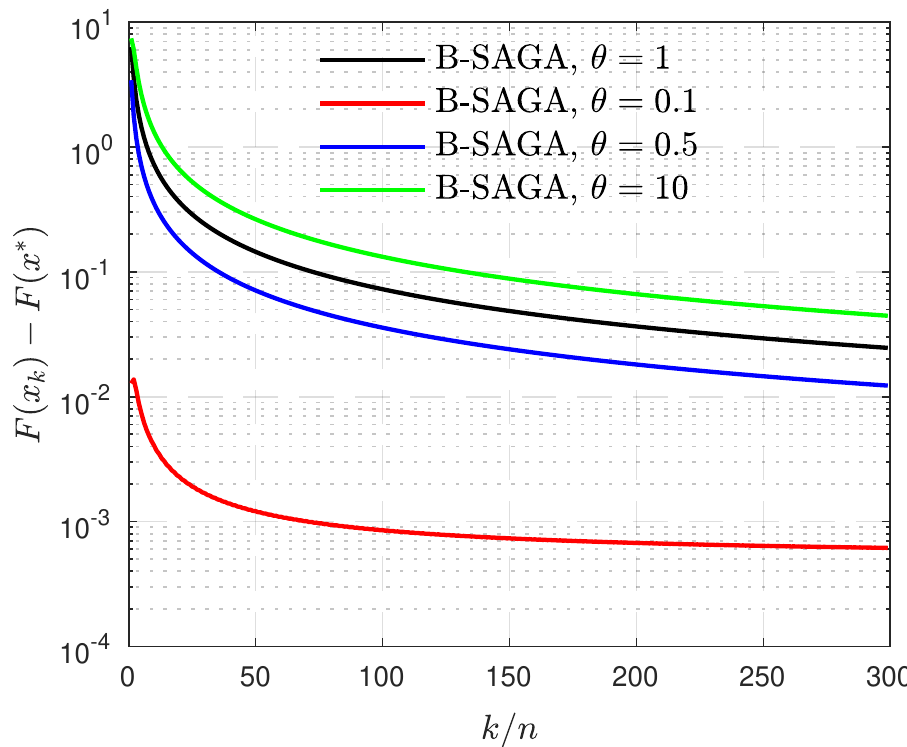} }   
%\hspace{1pt}
\subfloat[\texttt{ijcnn1}]{ \includegraphics[width=0.235\linewidth]{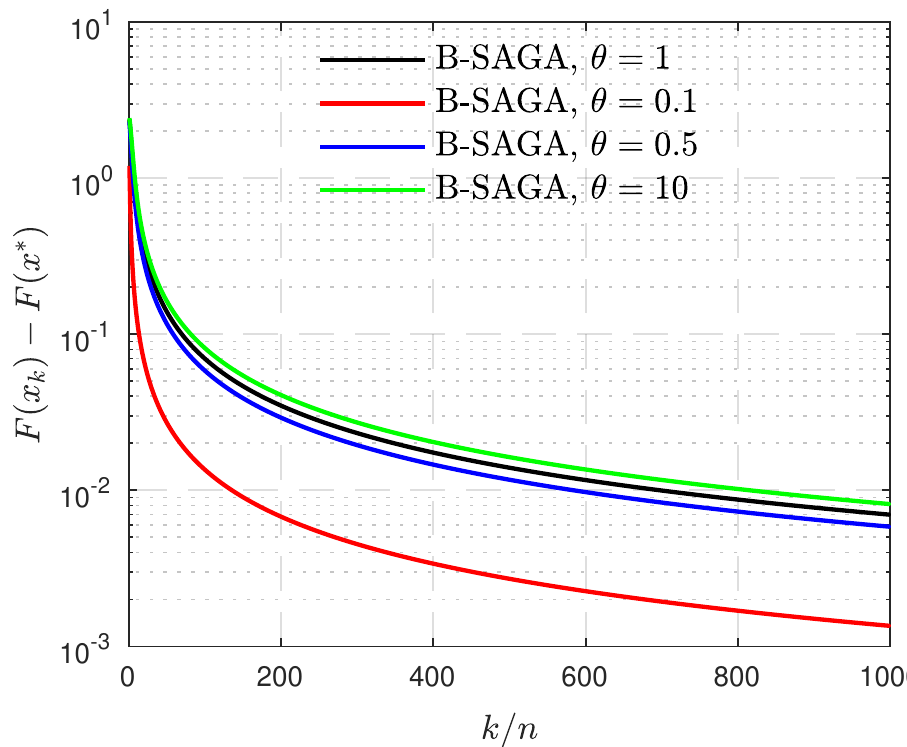} }  \\
%%%%%%%%
\caption{
Performance comparison for solving NN-PCA with different choices of $\theta$ in B-SAGA. The step size for each case is set to $\eta = \frac{1}{5 L n}$. The point $x^\star$ is found by solving the problem using proximal gradient descent to high accuracy.
}
\label{fig:nnpca-saga}
\end{figure}

\begin{figure}[!ht]
\centering
\subfloat[\texttt{australian}]{ \includegraphics[width=0.235\linewidth]{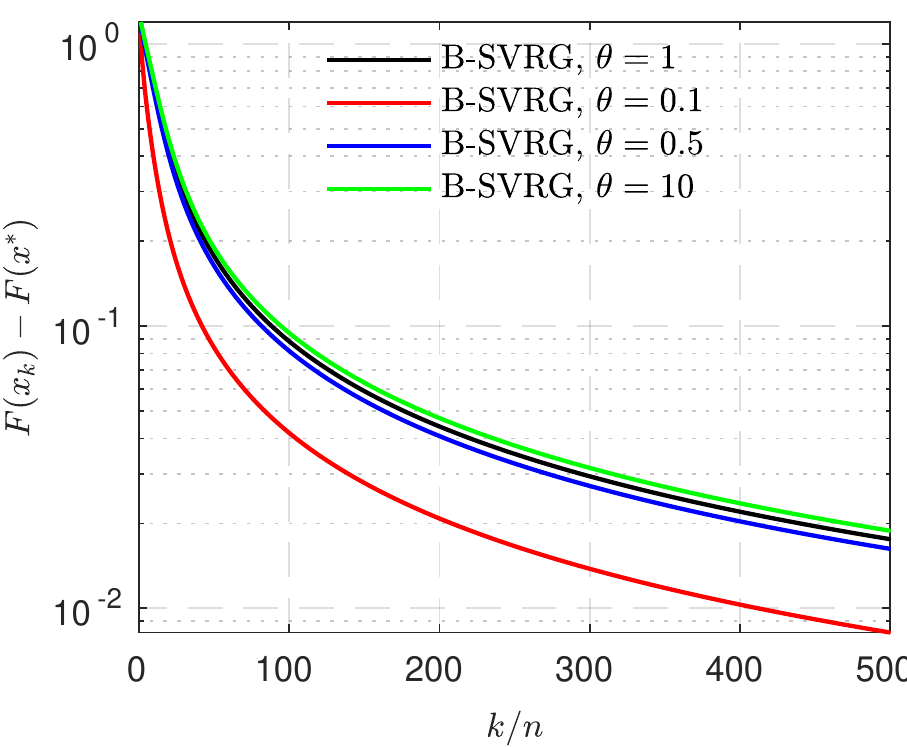} }  %\hspace{1pt}
\subfloat[\texttt{mushrooms}]{ \includegraphics[width=0.235\linewidth]{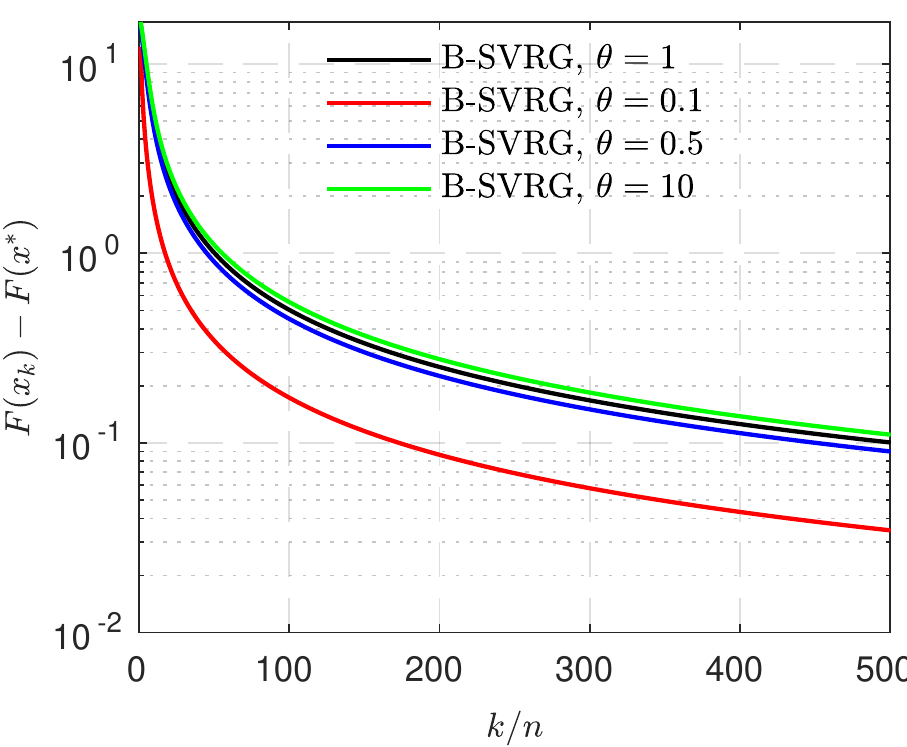} } %\hspace{1pt} 
\subfloat[\texttt{phishing}]{ \includegraphics[width=0.235\linewidth]{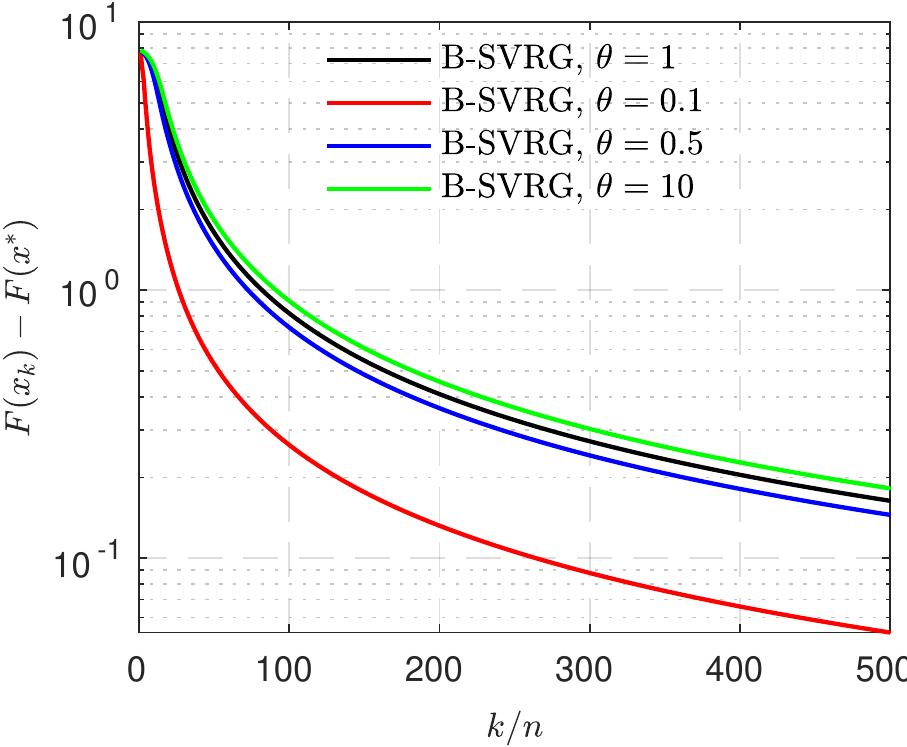} }   
%\hspace{1pt}
\subfloat[\texttt{ijcnn1}]{ \includegraphics[width=0.235\linewidth]{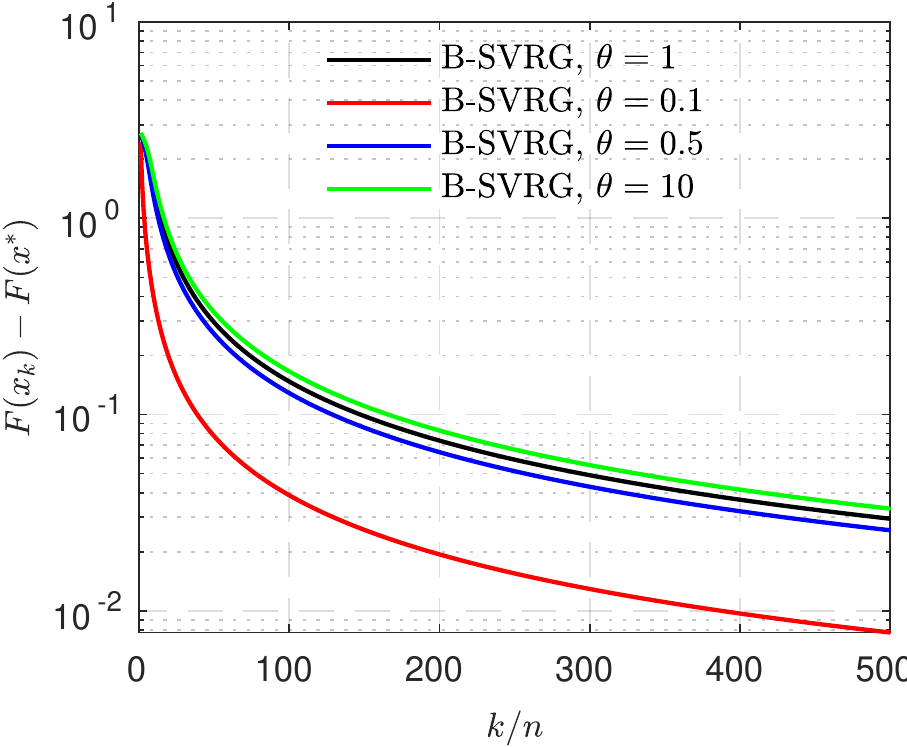} }  \\
%%%%%%%%
\caption{
Performance comparison for solving NN-PCA with different choices of $\theta$ in B-SVRG. The step size for each case is set to $\eta = \frac{1}{5 L n}$. The point $x^\star$ is found by solving the problem using proximal gradient descent.
}
\label{fig:nnpca-svrg}
\end{figure}

For the comparison of all algorithms, B-SAGA and B-SVRG provides the best performance with B-SVRG being slightly faster.

\begin{figure}[!ht]
\centering
\subfloat[\texttt{australian}]{ \includegraphics[width=0.235\linewidth]{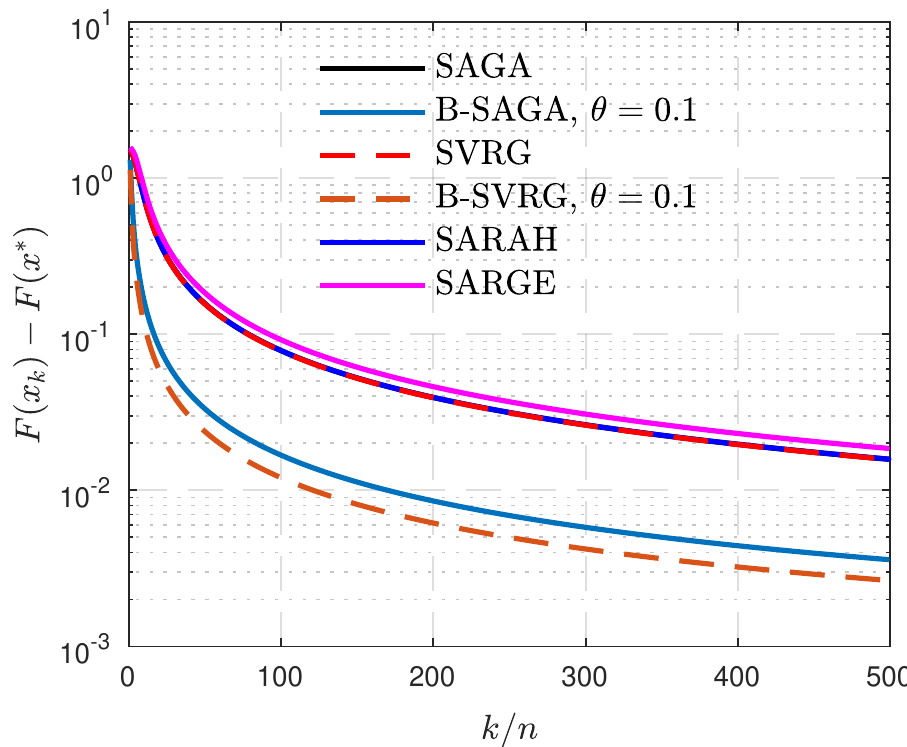} }  %\hspace{1pt}
\subfloat[\texttt{mushrooms}]{ \includegraphics[width=0.235\linewidth]{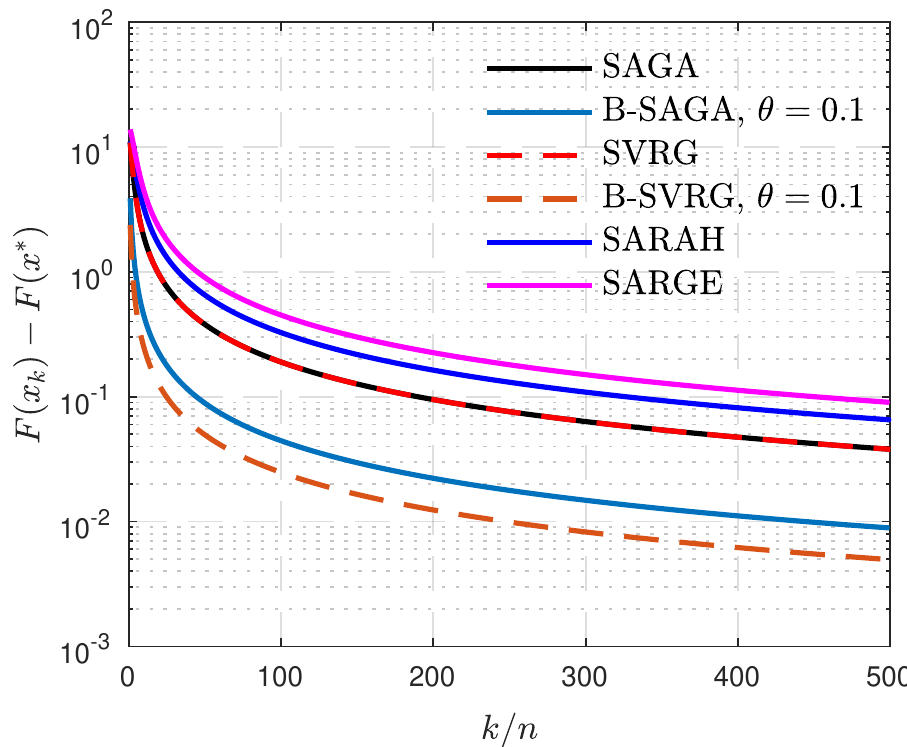} } %\hspace{1pt}
\subfloat[\texttt{phishing}]{ \includegraphics[width=0.235\linewidth]{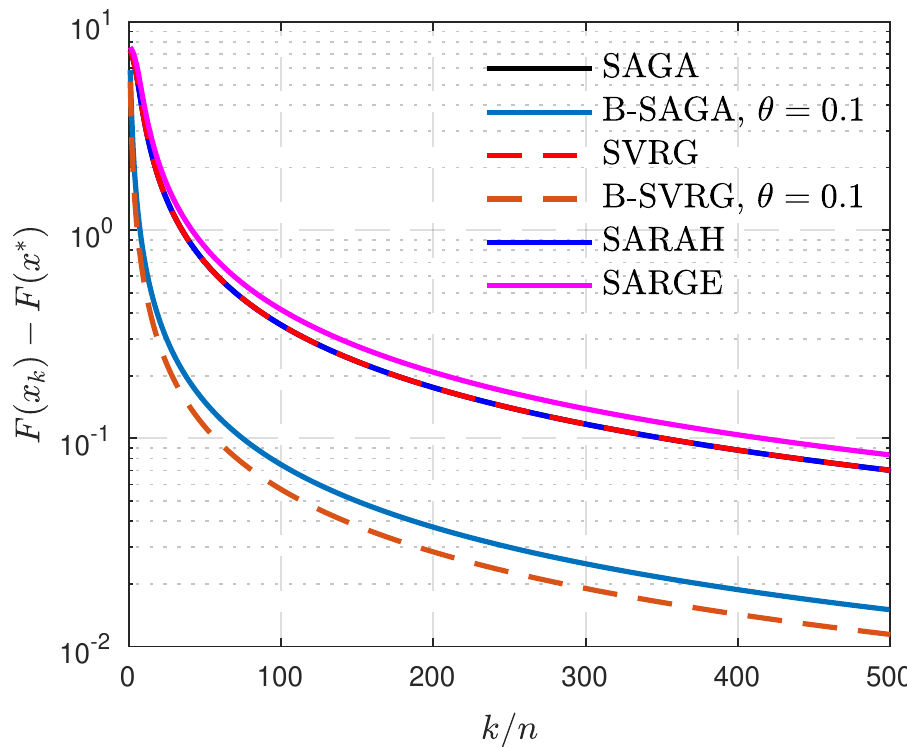} }   
%\hspace{1pt}
\subfloat[\texttt{ijcnn1}]{ \includegraphics[width=0.235\linewidth]{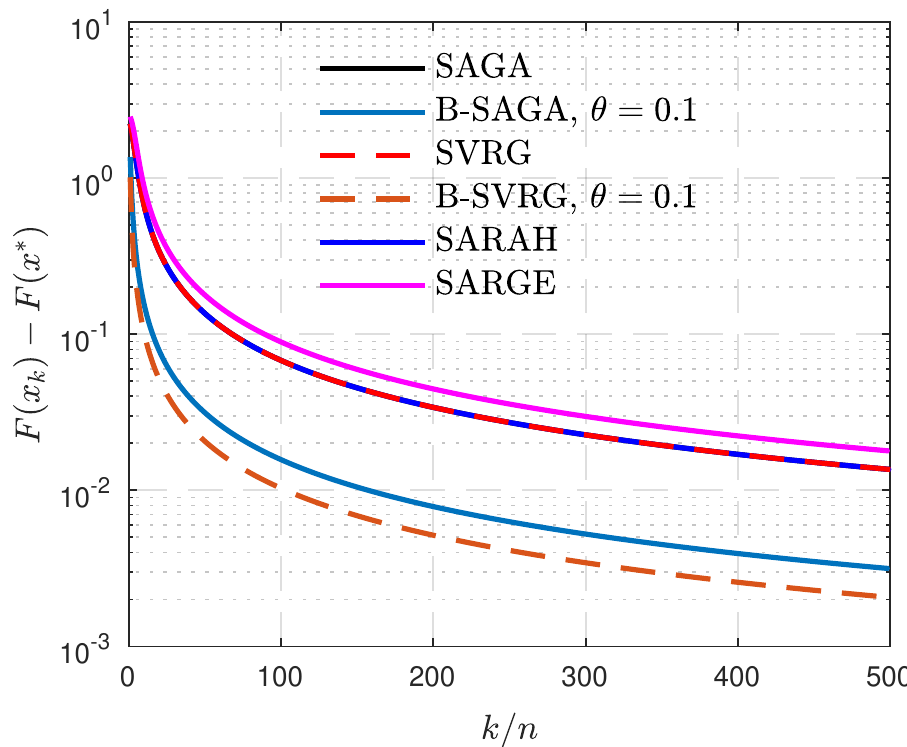} }  \\
%%%%%%%%
\caption{Performance comparison for solving NN-PCA among different algorithms. All step sizes are set to $\frac{1}{5 L n}$. Objective values are averaged over each epoch ($n$ steps).}
\label{fig:nnpca-all}
\end{figure}

\section{Conclusion}

The complicated convergence proofs of biased stochastic gradient methods have restricted researchers to studying unbiased estimators almost exclusively. Our simple framework for proving convergence rates for biased algorithms overcomes this limitation. Our analysis allows for the study of biased algorithms with proximal support for minimizing convex, strongly convex, and non-convex objectives for the first time.

We also show that biased gradient estimators can offer improvements over unbiased estimators in theory and in practice. The B-SAGA and B-SVRG gradient estimators incorporate bias to reduce their mean squared errors and improve their performance in many settings. The bias in recursive gradient estimators, such as SARAH and SARGE, lead to much smaller bounds on their MSE's and faster convergence rates than B-SAGA and B-SVRG.

\section*{Acknowledgements}
CBS acknowledges support from the Leverhulme Trust project on Breaking the Non-Convexity Barrier and on Unveiling the Invisible, the Philip Leverhulme Prize, the EPSRC grant No. EP/M00483X/1, the EPSRC Centre No. EP/N014588/1, the European Union Horizon 2020 research and innovation programmes under the Marie Skłodowska-Curie grant agreement No. 691070 CHiPS and the Marie Skłodowska-Curie grant agreement No 777826, the Cantab Capital Institute for the Mathematics of Information and the Alan Turing Institute.

\vskip 0.2in

\bibliographystyle{acm}
\bibliography{main.bib}

\appendix
\section*{Appendix}

The organization of the appendix is as follows: we prove Theorems \ref{thm:rate_mem} and \ref{thm:rate_recurs} in Appendices \ref{sec:proof_mem} and \ref{sec:proof_recur}, respectively, and we prove Theorem \ref{thm:rate_ncvx} in Appendix \ref{sec:proof_ncvx}. We provide convergence rates for B-SAGA and B-SVRG as special cases of Theorem \ref{thm:rate_mem} in Appendix \ref{sec:saga}, and we provide convergence rates for SARAH and SARGE as special cases of Theorem \ref{thm:rate_recurs} in Appendices \ref{sec:sarah} and \ref{sec:sarge}, respectively.

\section{Proof of Theorem \ref{thm:rate_mem}}\label{sec:proof_mem}

To prove Theorem \ref{thm:rate_mem}, we begin by showing that the BMSE property (Definition \ref{def:bmse}) implies that the MSE of the gradient estimator over $T$ iterations is proportional to $\sum_{k=0}^{T-1} \mathbb{E} \|x_{k+1} - x_k\|^2$.

\begin{lemma}[{MSE bound}]
\label{lem:msebound}
Suppose that the stochastic gradient estimator $\tnabla$ satisfies the \emph{BMSE}$(M_1,M_2,\rho_M, \rho_F, m)$ property, let $\rho = \min\{ \rho_M, \rho_F \}$, and let $\sigma_s$ be any sequence satisfying $\sigma_s (1-\rho)^{m s} \le \sigma_{s-1} (1-\frac{\rho}{2})^{m s}$. For convenience, define $\Theta = \frac{M_1 \rho_F + 2 M_2}{\rho_M \rho_F}$. The MSE of the gradient estimator is bounded as
\beqs
\begin{aligned}
\sum_{s=0}^{S} \sigma_s \sum_{k = m s}^{m (s+1) - 1} \E \ca{ \|\nabla f(x_{k}) - \tnabla_{k}\|^2 } 
\le 2 \Theta L^2 \sum_{s=0}^{S} \sigma_s \sum_{k = m s}^{m (s+1) - 1} \E \ca{ \| x_{k+1} - x_k \|^2 } .
\end{aligned}
\eeqs
\end{lemma}

\begin{proof}
First, we derive a bound on the sequence $\calF_{m s}$ arising in the BMSE property. Summing this sequence from $s=0$ to $s=S$,
\beqs\label{eq:est}
\begin{aligned}
\sum_{s=0}^{S} \sigma_s \calF_{m s} 
&\le \sum_{s=0}^{S} \sum_{\ell=0}^s \tfrac{M_2 \sigma_s (1-\rho_F)^{m (s - \ell)}}{n} \sum_{k = m s}^{m(s+1)-1} \sum_{i=1}^n \E \ca{ \| \nabla f_i(x_{k+1}) - \nabla f_i(x_k) \|^2 } \\
&\symnum{1}{\le} \sum_{s=0}^{S} \sum_{\ell=0}^s \tfrac{M_2 \sigma_\ell (1-\frac{\rho_F}{2})^{m (s - \ell)}}{n} \sum_{k = m s}^{m(s+1)-1} \sum_{i=1}^n \E \ca{ \| \nabla f_i(x_{k+1}) - \nabla f_i(x_k) \|^2 } \\
&\le \sum_{s=0}^{S} \left( \sum_{\ell=0}^\infty ( 1 - \tfrac{\rho_F}{2} )^\ell \right) \tfrac{M_2 \sigma_s}{n} \sum_{k = m s}^{m(s+1)-1} \sum_{i=1}^n \E \ca{ \| \nabla f_i(x_{k+1}) - \nabla f_i(x_k) \|^2 } \\
&= \sum_{s=0}^{S} \tfrac{2 M_2 \sigma_s}{n \rho_F} \sum_{k = m s}^{m(s+1)-1} \sum_{i=1}^n \E \ca{ \| \nabla f_i(x_{k+1}) - \nabla f_i(x_k) \|^2 } .
\end{aligned}
\eeqs
Inequality \numcirc{1} uses the fact that $\sigma_s (1-\rho_F)^{m s} \le \sigma_{s-1} ( 1-\frac{\rho_F}{2} )^{m s}$. With this bound on $\calF_{m s}$, we proceed to bound $\calM_{m s}$ similarly.
\beqs
\begin{aligned}
\sum_{s=0}^{S} \sigma_s \calM_{m s}
&\symnum{1}{\le} \sum_{s=0}^{S} \sigma_s \bPa{ \calF_{m s} + \tfrac{M_1}{n}
\sum_{k = m s}^{m (s+1) - 1} \sum_{i=1}^n \E \ca{ \|\nabla f_i(x_{k+1}) - \nabla f_i(x_k) \|^2 } } + (1-\rho_M)^m \sum_{s=1}^S \sigma_s \calM_{m (s-1)} \\
&\symnum{2}{\le} \sum_{s=0}^{S} \sigma_s \bPa{ \tfrac{M_1 \rho_F + 2 M_2}{n \rho_F} \sum_{k = m s}^{m (s+1) - 1} \sum_{i=1}^n \E \ca{ \| \nabla f_i(x_{k+1}) - \nabla f_i(x_k) \|^2 } } + (1-\tfrac{\rho_M}{2})^m \sum_{s=1}^S \sigma_{s-1} \calM_{m (s-1)} \\
&= \sum_{s=0}^{S} \sigma_s \bPa{ \tfrac{M_1 \rho_F + 2 M_2}{n \rho_F} \sum_{k = m s}^{m (s+1) - 1} \sum_{i=1}^n \E \ca{ \| \nabla f_i(x_{k+1}) - \nabla f_i(x_k) \|^2 }} \\
& \quad + (1-\tfrac{\rho_M}{2})^m \sum_{s=1}^S \sigma_{s-1} \bPa{ \tfrac{M_1 \rho_F + 2 M_2}{n \rho_F} \sum_{k = m (s-1)}^{m s - 1} \sum_{i=1}^n \E \ca{ \| \nabla f_i(x_{k+1}) - \nabla f_i(x_k) \|^2 }} + \cdots \\
&\le \left( \sum_{\ell = 0}^\infty (1 - \tfrac{\rho_M}{2})^{m \ell} \right) \sum_{s=0}^{S} \sigma_s \bPa{ \tfrac{M_1 \rho_F + 2 M_2}{n \rho_F} \sum_{k = m s}^{m (s+1) - 1} \sum_{i=1}^n \E \ca{ \| \nabla f_i(x_{k+1}) - \nabla f_i(x_k) \|^2 }} \\
&\symnum{3}{\le} \sum_{s=0}^{S} \tfrac{2 \sigma_s \Theta}{n} \sum_{k = m s}^{m (s+1) - 1} \sum_{i=1}^n \E \ca{ \| \nabla f_i(x_{k+1}) - \nabla f_i(x_k) \|^2 } \\
&\symnum{4}{\le} 2 \Theta L^2 \sum_{s=0}^{S} \sigma_s \sum_{k = m s}^{m (s+1) - 1} \E \ca{ \| x_{k+1} - x_k \|^2 } .
\end{aligned}
\eeqs
Inequality \numcirc{1} follows uses the fact that $\mathcal{M}_m \le (1 - \rho_M)^m \mathcal{M}_{m (s - 1)}$. Inequality \numcirc{2} uses $\sigma_s (1-\rho_M)^{m s} \le \sigma_{s-1} ( 1-\frac{\rho_M}{2} )^{m s}$, \numcirc{3} uses the same estimate we applied in \eqref{eq:est}, and \numcirc{4} uses the Lipschitz continuity of $\nabla f_i$.
\end{proof}

\begin{prooftext}{Proof of Lemma \ref{lem:couple1main}}
By assumption, $1-\frac{1}{\theta} \ge 0$, so we can apply convexity to obtain
\beqs
\begin{aligned}
& \tfrac{\eta}{\theta} \pa{ f(x_k) - f(\xsol) } + \tfrac{\eta}{n} \pa{ 1 - \tfrac{1}{\theta} } \Pa{ \msum_{i=1}^n f_i(\varphi_k^i) - f_i(\xsol) } \\
&\le \tfrac{\eta}{\theta} \langle \nabla f(x_k), x_k - \xsol \rangle + \tfrac{\eta}{n} \pa{ 1 - \tfrac{1}{\theta} } \sum_{i=1}^n \langle \nabla f_i(\varphi_k^i), \varphi_k^i - \xsol \rangle \\
&= \tfrac{\eta}{\theta} \langle \nabla f(x_k), x_k - \xsol \rangle + \tfrac{\eta}{n} \pa{ 1 - \tfrac{1}{\theta} } \sum_{i=1}^n \langle \nabla f_i(\varphi_k^i), x_k - \xsol \rangle + \tfrac{\eta}{n} \pa{ 1 - \tfrac{1}{\theta} } \sum_{i=1}^n \langle \nabla f_i(\varphi_k^i), \varphi_k^i - x_k \rangle.
\end{aligned}
\eeqs
Because $\tnablak{k}$ is memory-biased,
\beqs
\begin{aligned}
\tfrac{1}{\theta} \nabla f(x_k) + \tfrac{1}{n} \pa{ 1 - \tfrac{1}{\theta} } \sum_{i=1}^n \nabla f_i(\varphi_k^i) 
= \Ek \ca{ \tnablak{k} } .
\end{aligned}
\eeqs
Therefore,
\beqs
\begin{aligned}
&\tfrac{\eta}{\theta} \langle \nabla f(x_k), x_k - \xsol \rangle + \tfrac{\eta}{n} \pa{ 1 - \tfrac{1}{\theta} } \sum_{i=1}^n \langle \nabla f_i(\varphi_k^i), x_k - \xsol \rangle \\
&= \Ek \Ca{ \eta \langle \tnablak{k}, x_k - \xsol \rangle } \\
&= \Ek [ \eta \langle \tnablak{k}, x_k - x_{k+1} \rangle + \eta \langle \tnablak{k}, x_{k+1} - \xsol \rangle ] \\
&\le \Ek \Ca{ \eta \langle \tnablak{k}, x_k - x_{k+1} \rangle - \tfrac{1}{2} \|x_{k+1} - x_k\|^2 + \tfrac{1}{2} \|x_k - \xsol\|^2 - \tfrac{1 + \mu \eta}{2} \|x_{k+1} - \xsol\|^2 - \eta g(x_{k+1}) + \eta g(\xsol) } .
\end{aligned}
\eeqs
The inequality is due to Lemma \ref{lem:prox} with $z = x_{k+1}$, $x = x_k$, $d = \tnablak{k}$, and $y = \xsol$. Combining these two inequalities, we have shown
\beqs\label{eq:first}
\begin{aligned}
& \tfrac{\eta}{\theta} \pa{ f(x_k) - f(\xsol) } + \tfrac{\eta}{n} \pa{ 1 - \tfrac{1}{\theta} } \sum_{i=1}^n \Pa{f_i(\varphi_k^i) - f_i(\xsol) } \\
&\le \Ek \Ca{ \eta \langle \tnablak{k}, x_k - x_{k+1} \rangle - \tfrac{1}{2} \|x_{k+1} - x_k\|^2 - \eta g(x_{k+1}) + \eta g(\xsol) \\
&\qquad + \tfrac{1}{2} \|x_k - \xsol\|^2 - \tfrac{1 + \mu \eta}{2} \|x_{k+1} - \xsol\|^2 + \tfrac{\eta}{n} \pa{ 1 - \tfrac{1}{\theta} } \sum_{i=1}^n \langle \nabla f_i(\varphi_k^i), \varphi_k^i - x_k \rangle } .
\end{aligned}
\eeqs
We bound the first three terms on the right further.
\beqs
\begin{aligned} 
&\eta \langle \tnablak{k}, x_k - x_{k+1} \rangle - \tfrac{1}{2} \|x_{k+1} - x_k\|^2 - \eta g(x_{k+1}) \\
&= \eta ( \langle \nabla f(x_k), x_k - x_{k+1} \rangle - g(x_{k+1}) ) + \eta \langle \tnablak{k} - \nabla f(x_k), x_k - x_{k+1} \rangle - \tfrac{1}{2} \|x_{k+1} - x_k\|^2 \\
&\symnum{1}{\le} \eta \pa{ f(x_k) - F(x_{k+1}) } + \eta \langle \tnablak{k} - \nabla f(x_k), x_k - x_{k+1} \rangle + \pa{ \tfrac{\eta L}{2} - \tfrac{1}{2} } \|x_{k+1} - x_k\|^2 \\
&\symnum{2}{\le} \eta \pa{ f(x_k) - F(x_{k+1}) } + \tfrac{\eta}{2 L \lambda} \| \tnablak{k} - \nabla f(x_k) \|^2 + \pa{ \tfrac{\eta L (\lambda + 1)}{2} - \tfrac{1}{2} } \|x_{k+1} - x_k\|^2.
\end{aligned}
\eeqs
Inequality \numcirc{1} is due to the Lipschitz continuity of $\nabla f$, and inequality \numcirc{2} is Young's. Combining this bound with \eqref{eq:first} and rearranging terms, we have shown that
\beqs
\begin{aligned}
0 &\le - \eta \Ek [ F(x_{k+1}) - F(\xsol) ] + \tfrac{\eta}{2 L \lambda} \Ek \ca{ \|\tnablak{k} - \nabla f(x_k) \|^2 } \\
&\qquad - \tfrac{1 + \mu \eta}{2} \Ek \ca{ \|x_{k+1} - \xsol\|^2 } + \tfrac{1}{2} \|x_k - \xsol\|^2 + ( \tfrac{\eta L ( \lambda + 1 )}{2} - \tfrac{1}{2} ) \Ek \ca{ \|x_{k+1} - x_k\|^2 } \\
&\qquad + \eta \pa{ 1 - \tfrac{1}{\theta} } \bPa{ f(x_k) - \tfrac{1}{n} \sum_{i=1}^n f_i(\varphi_k^i) + \tfrac{1}{n} \sum_{i=1}^n \langle \nabla f_i(\varphi_k^i), \varphi_k^i - x_k \rangle } .
\end{aligned}
\eeqs
We use Lemma \ref{lem:2L} to bound the final term, yielding the desired inequality.
\end{prooftext}

\begin{prooftext}{Proof of Theorem \ref{thm:rate_mem} (Convex Case)}
We begin with the inequality of Lemma \ref{lem:couple1main} with $\mu = 0$. Multiplying the inequality of Lemma \ref{lem:descent} with $z = x_{k+1}$, $x = x_k$, and $d = \tnablak{k}$ by a non-negative constant $\delta$ and adding it to the inequality of Lemma \ref{lem:couple1main}, we obtain
\beqs
\begin{aligned}
& \eta \Ek [ F(x_{k+1}) - F(\xsol) + \delta (F(x_{k+1}) - F(x_k)) ] \\
&\le \tfrac{\eta ( 1 + \delta )}{2 L \lambda} \Ek \ca{ \|\tnablak{k} - \nabla f(x_k) \|^2 } - \tfrac{1}{2} \Ek \ca{ \|x_{k+1} - \xsol\|^2 } + \tfrac{1}{2} \|x_k - \xsol\|^2 \\
&\qquad + \pa{ \tfrac{\eta L (1 + \delta) ( \lambda + 1 )}{2} - \tfrac{1 + 2 \delta}{2} } \Ek \ca{ \|x_{k+1} - x_k\|^2 } + \tfrac{\eta L}{2 n} \pa{ 1 - \tfrac{1}{\theta} } \sum_{i=1}^n \|x_k - \varphi_k^i\|^2.
\end{aligned}
\eeqs
Applying the full expectation operator and summing from $k=0$ to $k = T-1$, we have
\beqs
\begin{aligned}
&\eta \sum_{k=0}^{T-1} \E [ F(x_{k+1}) - F(\xsol) ] + \eta \delta \pa{ \E \ca{ F(x_T) } - F(x_0) } \\
&\le - \tfrac{1}{2} \E \ca{ \|x_T - \xsol\|^2 } + \tfrac{1}{2} \|x_0 - \xsol\|^2 + \sum_{k=0}^{T-1} \E \Ca{ \tfrac{\eta (1 + \delta)}{2 L \lambda} \|\tnablak{k} - \nabla f(x_k) \|^2 \\
& \qquad + \pa{ \tfrac{\eta L (1 + \delta) ( \lambda + 1 )}{2} - \tfrac{1 + 2 \delta}{2} } \|x_{k+1} - x_k\|^2 + \tfrac{\eta L}{2 n} \pa{ 1 - \tfrac{1}{\theta} } \msum_{i=1}^n \|x_k - \varphi_k^i\|^2 } .
\end{aligned}
\eeqs
We use Lemma \ref{lem:msebound} with $\sigma_s = 1$ to bound the MSE, and we use the fact that the gradient estimator is memory-biased to bound the term $1 / n \sum_{i=1}^n \|x_k - \varphi_k^i\|^2$. This leaves
\beqs
\begin{aligned}
\label{eq:finaltermnonpos}
\eta \sum_{k=0}^{T-1} \E [ F(x_{k+1}) - F(\xsol) ] 
&\le - \tfrac{1}{2} \E \ca{ \|x_T - \xsol\|^2 } + \tfrac{1}{2} \|x_0 - \xsol\|^2 + \eta \delta ( F(x_0) - \E \ca{ F(x_T) } ) \\
&+ \pa{ \tfrac{\eta L (1 + \delta) ( \lambda + 1 )}{2} + \tfrac{\Theta \eta L (1 + \delta)}{\lambda} + \tfrac{B_1 \eta L}{2} \pa{ 1 - \tfrac{1}{\theta} } - \tfrac{1 + 2 \delta}{2} } \sum_{k=0}^{T-1} \E \ca{ \|x_{k+1} - x_k\|^2 } .
\end{aligned}
\eeqs
Setting $\lambda = \sqrt{2 \Theta}$ minimizes the coefficient of the term on the final line. With
\beq
\eta 
\le \tfrac{1}{L (1 + 2 \sqrt{2 \Theta} + \tfrac{B_1 (1 - {1}/{\theta})}{1 + \delta}) },
\eeq
the final term in \eqref{eq:finaltermnonpos} is non-positive, so we can drop it from the inequality along with the term $-1/2 \E \|x_T - \xsol\|^2$. Using the fact that $ - F(x_T) \le - F(\xsol)$, this leaves
\beq
\sum_{k=0}^{T-1} \E [ F(x_{k+1}) - F(\xsol) ] 
\le \tfrac{1}{2 \eta} \|x_0 - \xsol\|^2 + \eta \delta \pa{ F(x_0) - F(\xsol) }.
\eeq
We use the convexity of $F$ to rewrite this inequality as a bound on the suboptimality of the average iterate
\beq
\E \ca{ F(\xbar_{T}) - F(\xsol) } 
\le \sfrac{1}{T} \sum_{k=0}^{T-1} \E [ F(x_{k+1}) - F(\xsol) ] 
\le \sfrac{1}{2 \eta T} \|x_0 - \xsol\|^2 + \tfrac{\eta \delta}{T} \pa{ F(x_0) - F(\xsol) } .
\eeq
Setting $\delta = \max \{B_1 (1 - 1/\theta) / \sqrt{2 \Theta} - 1, 0\}$ approximately minimizes the right side, proving the assertion.
\end{prooftext}

\begin{prooftext}{Proof of Theorem \ref{thm:rate_mem} (Strongly Convex Case)}
As in the proof of the convex case, we begin with the inequality of Lemma \ref{lem:couple1main}, multiply the inequality of Lemma \ref{lem:descent} with $z = x_{k+1}$, $x = x_k$, and $d = \tnablak{k}$ by a non-negative constant $\delta$, and add the two inequalities.
\beqs
\begin{aligned}
& \eta \Ek [ F(x_{k+1}) - F(\xsol) + \delta \pa{ F(x_{k+1}) - F(x_k) } ] \\
&\le - \tfrac{1 + \mu \eta}{2} \Ek \ca{ \|x_{k+1} - \xsol\|^2 } + \tfrac{1}{2} \|x_k - \xsol\|^2 + \Ek \bCa{ \tfrac{\eta (1 + \delta)}{2 L \lambda} \|\tnablak{k} - \nabla f(x_k) \|^2 \\
& \qquad + \pa{ \tfrac{\eta L (1 + \delta) ( \lambda + 1 )}{2} - \tfrac{1 + \delta (2 + \mu \eta)}{2} } \|x_{k+1} - x_k\|^2 + \tfrac{\eta L}{2 n} \pa{ 1 - \tfrac{1}{\theta} } \msum_{i=1}^n \|x_k - \varphi_k^i\|^2 }.
\end{aligned}
\eeqs
Applying the full expectation operator, multiplying by $(1 + \mu \eta)^k$, and summing over the epoch $k=m s$ to $k = m (s+1)-1$ for some $s \in \mathbb{N}_0$, we have
\beqs
\begin{aligned}
& \eta \sum_{k=ms}^{m (s+1)-1} (1 + \mu \eta)^k \E [ F(x_{k+1}) - F(\xsol) + \delta (F(x_{k+1}) - F(x_k)) ] \\
&\le - \tfrac{(1 + \mu \eta)^{m (s + 1)}}{2} \E \|x_{m (s + 1)} - \xsol\|^2 + \tfrac{(1 + \mu \eta)^{m s}}{2} \mathbb{E} \|x_{m s} - \xsol\|^2\\
&\qquad + \sum_{k = m s}^{m (s+1)-1} (1 + \mu \eta)^k \E \bCa{ \tfrac{\eta (1 + \delta)}{2 L \lambda} \|\tnablak{k} - \nabla f(x_k) \|^2+ \pa{ \tfrac{\eta L (1 + \delta) ( \lambda + 1 )}{2} - \tfrac{1 + \delta (2 + \mu \eta)}{2} } \|x_{k+1} - x_k\|^2 \\&\qquad + \tfrac{\eta L}{2 n} \pa{ 1 - \tfrac{1}{\theta} } \msum_{i=1}^n \|x_k - \varphi_k^i\|^2 } .
\end{aligned}
\eeqs
Using the fact that $\eta \le \frac{1}{\mu m}$,
\beq
\label{eq:e}
(1 + \mu \eta)^k 
\le (1 + \mu \eta)^{m (s + 1)} 
\le (1 + \mu \eta)^{m s} \lim_{m \to \infty} \pa{ 1 + \tfrac{1}{m} }^m 
= e (1 + \mu \eta)^{m s} 
\le 3 (1 + \mu \eta)^{m s},
\eeq
where $e$ is Euler's number. Therefore,
\beqs
\begin{aligned}
& \eta \sum_{k=ms}^{m (s+1)-1} (1 + \mu \eta)^k \E [ F(x_{k+1}) - F(\xsol) + \delta (F(x_{k+1}) - F(x_k)) ] \\
&\le - \tfrac{(1 + \mu \eta)^{m (s + 1)}}{2} \E \ca{ \|x_{m (s + 1)} - \xsol\|^2 } + \tfrac{(1 + \mu \eta)^{m s}}{2} \|x_{m s} - \xsol\|^2 \\
&\qquad + (1 + \mu \eta)^{m s} \sum_{k = m s}^{m (s+1)-1} \E \bCa{ \tfrac{3 \eta (1 + \delta)}{2 L \lambda} \|\tnablak{k} - \nabla f(x_k) \|^2 + \pa{ \tfrac{3 \eta L (1 + \delta) ( \lambda + 1 )}{2} - \tfrac{1 + \delta (2 + \mu \eta)}{2} } \|x_{k+1} - x_k\|^2 \\&\qquad + \tfrac{3 \eta L}{2 n} \pa{ 1 - \tfrac{1}{\theta} } \msum_{i=1}^n \|x_k - \varphi_k^i\|^2 } .
\end{aligned}
\eeqs
Summing the inequality from epoch $s=0$ to $s = S - 1$,
\beqs
\begin{aligned}
& \eta \sum_{k=0}^{m S - 1} (1 + \mu \eta)^k \E [ F(x_{k+1}) - F(\xsol) + \delta (F(x_{k+1}) - F(x_k)) ] \\
&\le \sum_{s=0}^{S-1} (1 + \mu \eta)^{m s} \sum_{k = m s}^{m (s+1)-1} \E \bCa{ \tfrac{3 \eta (1 + \delta)}{2 L \lambda} \|\tnablak{k} - \nabla f(x_k) \|^2 + \pa{ \tfrac{3 \eta L (1 + \delta) ( \lambda + 1 )}{2} - \tfrac{1+\delta(2 + \mu \eta)}{2} } \|x_{k+1} - x_k\|^2 \\
&\qquad + \tfrac{3 \eta L (1 + \delta)}{2 n} \pa{ 1 - \tfrac{1}{\theta} } \msum_{i=1}^n \|x_k - \varphi_k^i\|^2 } - \tfrac{(1 + \mu \eta)^{m S}}{2} \E \|x_{m s} - \xsol\|^2 + \tfrac{1}{2} \|x_0 - \xsol\|^2 .
\end{aligned}
\eeqs
We use Lemma \ref{lem:msebound} with $\sigma_s = (1 + \mu \eta)^{m s}$ to bound the MSE. Recall $\rho = \min\{ \rho_M, \rho_F \}$ and $\eta \le \frac{\rho}{2 \mu}$. This choice for $\sigma_s$ satisfies the conditions of Lemma \ref{lem:msebound} because $(1 + \mu \eta)^{m s} (1 - \rho)^{m s} \le (1 + \mu \eta)^{m (s-1)} (1 - \rho/2)^{m s}$. We use the fact that the gradient estimator is memory-biased to bound the term $1 / n \sum_{i=1}^n \|x_k - \varphi_k^i\|^2$. This leaves
\beqs
\begin{aligned}
\label{eq:nonpos}
& \eta \sum_{k=0}^{m S - 1} (1 + \mu \eta)^k \E [ F(x_{k+1}) - F(\xsol) + \delta (F(x_{k+1}) - F(x_k)) ] \\
&\le - \tfrac{(1 + \mu \eta)^{m S}}{2} \E \ca{ \|x_{m S} - \xsol\|^2 } + \tfrac{1}{2} \|x_0 - \xsol\|^2 + C \sum_{s=0}^{S-1} (1 + \mu \eta)^{m s} \sum_{k = m s}^{m (s+1)-1} \E \ca{ \|x_{k+1} - x_k\|^2 } ,
\end{aligned}
\eeqs
where $C = { \frac{3 \eta L (1 + \delta) ( \lambda + 1 )}{2} + \frac{3 \Theta \eta L (1 + \delta)}{\lambda} + \frac{3 B_1 \eta L}{2} \pa{ 1 - \frac{1}{\theta} } - \frac{1 + \delta (2 + \mu \eta)}{2} }$. We must choose $\eta, \lambda$, and $\delta$ so that $C \le 0$.
Setting $\lambda = \sqrt{2 \Theta}$ minimizes $C$ over $\lambda$. Using the approximation $\delta(2 + \mu \eta) \ge \delta$, we see that $C$ is non-positive if
\beq
\eta 
\le \tfrac{1}{3 L (1 + 2 \sqrt{2 \Theta} + \frac{ B_1 ( 1 - {1}/{\theta}) }{1 + \delta}) }.
\eeq
Setting $\delta = \max \{B_1(1 - 1/\theta)/\sqrt{2 \Theta} - 1, 0\}$, we are guaranteed that 
\beq
\tfrac{1}{3 L (1 + 3 \sqrt{2 \Theta})} 
\le \tfrac{1}{3 L (1 + 2 \sqrt{2 \Theta} + \frac{ B_1 \pa{ 1 - {1}/{\theta} } }{1 + \delta}) },
\eeq
so the step size in the theorem statement ensures $C \le 0$, and the final term in \eqref{eq:nonpos} is non-positive. Dropping this non-positive term from the inequality, we have
\beqs\label{eq:strcon}
\begin{aligned}
&\eta (1 + \delta) \sum_{k=0}^{m S-1} (1 + \mu \eta)^k \E [ F(x_{k+1}) - F(\xsol) ] + \delta \eta \sum_{k=0}^{m S-1} (1 + \mu \eta)^k \E \ca{ F(x_k) - F(\xsol) } \\
&\le - \tfrac{(1 + \mu \eta)^{m S}}{2} \E \ca{ \|x_{m S} - \xsol\|^2 } + \tfrac{1}{2} \|x_0 - \xsol\|^2 .
\end{aligned}
\eeqs
We would like to show that $1 + \delta \ge (1 + \mu \eta) \delta$ so that the terms on the first line telescope.

We use the fact that $\eta \le \tfrac{\sqrt{2 \Theta}}{B_1 \mu (1 - 1 / \theta)}$ to say
\begin{equation}
    \tfrac{1}{\mu \eta} \ge \tfrac{B_1 (1 - 1 / \theta)}{\sqrt{2 \Theta}} \ge \delta
\end{equation}
Hence,
\beq
\tfrac{1 + \delta}{\delta} \ge 1 + \mu \eta,
\eeq
so inequality \eqref{eq:strcon} simplifies to
\beq
(1 + \mu \eta)^{m S} \E [ \eta \delta ( F(x_{m S}) - F(\xsol) ) + \tfrac{1}{2} \|x_{m S} - \xsol\|^2 ] 
\le \eta \delta \pa{ F(x_0) - F(\xsol) } + \tfrac{1}{2} \|x_0 - \xsol\|^2 , 
\eeq
which implies the result.
\end{prooftext}

\section{Proof of Theorem \ref{thm:rate_recurs}}\label{sec:proof_recur}

The following two lemmas establish an analogue of Lemma \ref{lem:couple1main} for recursively biased estimators.

\begin{lemma}
\label{lem:couple1}
Suppose $\tnabla$ is recursively biased with parameters $\rho_B$ and $\nu$. Suppose $g$ is $\mu$-strongly convex with $\mu \ge 0$, and let $\lambda > 0$ be a constant whose value we determine later. The following inequality holds:
\beqs
\begin{aligned}
0 &\le - \eta \Ek [ F(x_{k+1}) - F(\xsol) ] + \tfrac{\eta}{2 L \lambda} \Ek \ca{ \|\tnablak{k} - \nabla f(x_k) \|^2 } \\
&\qquad - \tfrac{1 + \mu \eta}{2} \Ek \ca{ \|x_{k+1} - \xsol\|^2 } + \tfrac{1}{2} \|x_k - \xsol\|^2 \\&\qquad + \pa{ \tfrac{\eta L ( \lambda + 1 )}{2} - \tfrac{1}{2} } \Ek \ca{ \|x_{k+1} - x_k\|^2 } + \eta (1-\rho_B) \langle \nabla f(x_{k-1}) - \tnabla_{k-1}, x_k - \xsol \rangle . 
\end{aligned}
\eeqs
\end{lemma}

\begin{proof}
Applying the convexity of $f$ yields
\beqs\label{eq:sargecon}
\begin{aligned}
&\eta \pa{ f(x_k) - f(\xsol) } \\
&\le \eta \langle \nabla f(x_k), x_k - \xsol \rangle \\
&= \eta \langle \nabla f(x_k) - ( 1 - \rho_B ) ( \nabla f(x_{k-1}) - \tnabla_{k-1} ), x_k - \xsol \rangle + \eta ( 1 - \rho_B ) \langle \nabla f(x_{k-1}) - \tnabla_{k-1}, x_k - \xsol \rangle .
\end{aligned}
\eeqs
Because the estimator is recursively biased,
\beq
\Ek \ca{ \tnablak{k} } 
= \nabla f(x_k) - ( 1 - \rho_B ) ( \nabla f(x_{k-1}) - \tnabla_{k-1} ).
\eeq
Therefore,
\beqs\label{eq:lemma3}
\begin{aligned}
& \eta \langle \nabla f(x_k) - ( 1 - \rho_B ) ( \nabla f(x_{k-1}) - \tnabla_{k-1} ), x_k - \xsol \rangle \\
&= \Ek [ \eta \langle \tnablak{k}, x_k - \xsol \rangle ] \\
&= \Ek [ \eta \langle \tnablak{k}, x_k - x_{k+1} \rangle + \eta \langle \tnablak{k}, x_{k+1} - \xsol \rangle ] \\
&\le \Ek [ \eta \langle \tnablak{k}, x_k - x_{k+1} \rangle - \tfrac{1}{2} \|x_{k+1} - x_k\|^2 + \tfrac{1}{2} \|x_k - \xsol\|^2 - \tfrac{1}{2} \|x_{k+1} - \xsol\|^2 + \eta g(x_{k+1}) - \eta g(\xsol) ],
\end{aligned}
\eeqs
The inequality is due to Lemma \ref{lem:prox}. The rest of the proof follows the proof of Lemma \ref{lem:couple1main}.
\end{proof}

\begin{prooftext}{Proof of Lemma \ref{lem:bias}}
Because $x_{k-1}$ is independent of $j_{k-1}$, we can use the BMSE property
\beqs
\begin{aligned}
&\E \langle \nabla f(x_{k-1}) - \tnabla_{k-1}, x_k - \xsol \rangle \\
&\symnum{1}{=} \E \ca{ \langle \nabla f(x_{k-1}) - \tnabla_{k-1}, x_k - x_{k-1} \rangle + \langle \nabla f(x_{k-1}) - \E_{k-1} \tnabla_{k-1}, x_{k-1} - \xsol \rangle } \\
&\symnum{2}{\le} \E \Ca{ \tfrac{\epsilon}{2} \| \nabla f(x_{k-1}) - \tnabla_{k-1} \|^2 + \tfrac{1}{2 \epsilon} \|x_k - x_{k-1}\|^2 + \langle \nabla f(x_{k-1}) - \E_{k-1} \tnabla_{k-1}, x_{k-1} - \xsol \rangle } \\
&\symnum{3}{=} \E \Ca{ \tfrac{ \epsilon}{2} \| \nabla f(x_{k-1}) - \tnabla_{k-1} \|^2 + \tfrac{1}{2 \epsilon} \|x_k - x_{k-1}\|^2 + (1 - \rho_B) \langle \nabla f(x_{k-2}) - \tnabla_{k-2}, x_{k-1} - \xsol \rangle } .
\end{aligned}
\eeqs
We can pass the conditional expectation $\mathbb{E}_{k-1}$ into the second inner-product in \numcirc{1} because $x_{k-1}$ is independent of $j_{k-1}$. Inequality \numcirc{2} is Young's, and \numcirc{3} uses the definition of a recursively biased gradient estimator.

This is a recursive inequality, and expanding the recursion gives
\beqs
\begin{aligned}
&\E \langle \nabla f(x_{k-1}) - \tnabla_{k-1}, x_k - \xsol \rangle \\
&\le \sum_{\ell = \nu s + 1}^{k-1} (1 - \rho_B)^{k - \ell - 1} \E \Ca{ \tfrac{ \epsilon}{2} \| \nabla f(x_\ell) - \tnabla_\ell \|^2 + \tfrac{1}{2 \epsilon} \|x_{\ell+1} - x_\ell \|^2 + (1 - \rho_B) \langle \nabla f(x_{\nu s}) - \tnabla_{\nu s}, x_{\nu s + 1} - \xsol \rangle } \\
&\symnum{1}{=} \sum_{\ell = \nu s + 1}^{k-1} (1 - \rho_B)^{k - \ell - 1} \E \Ca{ \tfrac{ \epsilon}{2} \| \nabla f(x_\ell) - \tnabla_\ell \|^2 + \tfrac{1}{2 \epsilon} \|x_{\ell+1} - x_\ell \|^2 } .
\end{aligned}
\eeqs
Equality \numcirc{1} is due to the fact that $\tnabla_{\nu s} = \nabla f(x_{\nu s})$. Taking the absolute value and summing this from $k = \nu s + 1$ to $k = \nu (s + 1) - 1$,
\beqs
\begin{aligned}
&\sum_{k = \nu s + 1}^{\nu (s + 1) - 1} | \E \langle \nabla f(x_{k-1}) - \tnabla_{k-1}, x_k - \xsol \rangle | \\
&\le \sum_{k = \nu s + 1}^{\nu (s + 1) - 1} \sum_{\ell = \nu s + 1}^{k-1} (1 - \rho_B)^{k - \ell - 1} \E \Ca{ \tfrac{\epsilon}{2} \| \nabla f(x_\ell) - \tnabla_\ell \|^2 + \tfrac{1}{2 \epsilon} \|x_{\ell+1} - x_\ell \|^2 } \\
&\le \min \bBa{\nu, \sum_{\ell = 0}^{\infty} (1 - \rho_B)^{\ell} } \sum_{k = \nu s + 1}^{\nu (s + 1) - 1} \E \Ca{ \tfrac{ \epsilon}{2} \| \nabla f(x_k) - \tnablak{k} \|^2 + \tfrac{1}{2 \epsilon} \|x_{k+1} - x_k \|^2 } \\
&\le \min \bBa{\nu, \tfrac{1}{\rho_B} } \sum_{k = \nu s + 1}^{\nu (s + 1) - 1} \E \Ca{ \tfrac{\epsilon}{2} \| \nabla f(x_k) - \tnablak{k} \|^2 + \tfrac{1}{2 \epsilon} \|x_{k+1} - x_k \|^2 } .
\end{aligned}
\eeqs
Summing this inequality from $s = 0$ to $s = S$ completes the proof.
\end{prooftext}

\begin{prooftext}{Proof of Theorem \ref{thm:rate_recurs} (Convex Case)}
To begin, we sum the inequality of Lemma \ref{lem:couple1} and the inequality of Lemma \ref{lem:descent} scaled by $\delta > 0$ with $z = x_{k+1}$, $x = x_k$, and $d = \tnablak{k}$.
\beqs\label{eq:delta}
\begin{aligned}
& \eta \Ek [ F(x_{k+1}) - F(\xsol) + \delta\pa{ F(x_{k+1}) - F(x_k) } ] \\
&\le - \tfrac{1}{2} \Ek \ca{ \|x_{k+1} - \xsol\|^2 } + \tfrac{1}{2} \|x_k - \xsol\|^2 + \Ek \bCa{ \tfrac{\eta (1+\delta)}{2 L \lambda} \|\tnablak{k} - \nabla f(x_k) \|^2 \\
&\qquad + (1 + \delta) \pa{ \tfrac{\eta L ( \lambda + 1 )}{2} - \tfrac{1}{2} } \|x_{k+1} - x_k\|^2 + \eta (1-\rho_B) \langle \nabla f(x_{k-1}) - \tnabla_{k-1}, x_k - \xsol \rangle } .
\end{aligned}
\eeqs
Applying the full expectation operator, setting $\mu = 0$, and summing from $k=0$ to $k = T-1$ where $T = m S$ for some $S \in \mathbb{N}$, we have
\beqs
\begin{aligned}
& \eta \sum_{k=0}^{T-1} \E [ F(x_{k+1}) - F(\xsol) ] + \eta \delta \E [ F(x_T) - F(x_0) ] \\
&\le - \tfrac{1}{2} \E \ca{ \|x_T - \xsol\|^2 } + \tfrac{1}{2} \|x_0 - \xsol\|^2 + \sum_{k=0}^{T-1} \E \bCa{ \tfrac{\eta (1+\delta)}{2 L \lambda} \|\tnablak{k} - \nabla f(x_k) \|^2 \\
&\qquad + (1 + \delta) \pa{ \tfrac{\eta L ( \lambda + 1 )}{2} - \tfrac{1}{2} } \|x_{k+1} - x_k\|^2 + \eta (1-\rho_B) \langle \nabla f(x_{k-1}) - \tnabla_{k-1}, x_k - \xsol \rangle } .
\end{aligned}
\eeqs
We use Lemma \ref{lem:bias} to bound the inner-product bias term.
\beqs
\begin{aligned}
& \eta \sum_{k=0}^{T-1} \E [ F(x_{k+1}) - F(\xsol) ] + \eta \delta \E [ F(x_T) - F(x_0) ] \\
&\le - \tfrac{1}{2} \E \ca{ \|x_T - \xsol\|^2 } + \tfrac{1}{2} \|x_0 - \xsol\|^2 + \sum_{k=0}^{T-1} \E \bCa{ \pa{ \tfrac{\eta (1+\delta)}{2 L \lambda} + \tfrac{B_2 \eta ( 1 - \rho_B ) \epsilon}{2} } \|\tnablak{k} - \nabla f(x_k) \|^2 \\
&\qquad + (1 + \delta) \pa{ \tfrac{\eta L ( \lambda + 1 )}{2} + \tfrac{B_2 \eta (1 - \rho_B)}{2 \epsilon ( 1 + \delta )} - \tfrac{1}{2} } \|x_{k+1} - x_k\|^2 } .
\end{aligned}
\eeqs
To bound the MSE, we use Lemma \ref{lem:msebound} with $\sigma_s = 1$. This leaves
\beqs
\begin{aligned}
\label{eq:finalterm}
& \eta \sum_{k=0}^{T-1} \E [ F(x_{k+1}) - F(\xsol) ] + \eta \delta \E [ F(x_T) - F(x_0) ] \\
&\le - \tfrac{1}{2} \E \ca{ \|x_T - \xsol\|^2 } + \tfrac{1}{2} \|x_0 - \xsol\|^2 
 + w \sum_{k=0}^{T-1} \E \ca{ \|x_{k+1} - x_k\|^2 } ,
\end{aligned}
\eeqs
where $w = { \frac{\eta L ( \lambda + 1 ) (1+\delta)}{2} + \frac{B_2 \eta (1 - \rho_B)}{2 \epsilon} + \frac{\Theta \eta L (1+\delta)}{\lambda} + B_2 \eta L^2 (1 - \rho_B) \epsilon \Theta - \frac{1+\delta}{2} }$. 
To minimize the coefficient of the final term, we set $\lambda = \sqrt{2 \Theta}$ and $\epsilon = (2 L^2 \Theta)^{-1/2}$. This coefficient is then equal to 
\beq
\sqrt{2 \Theta} \eta L (1 + \delta) + \tfrac{\eta L (1+\delta)}{2} + \sqrt{2} (1 - \rho_B) \eta L B_2 \sqrt{\Theta} - \tfrac{1 + \delta}{2},
\eeq
which is non-positive when $\eta 
\le \tfrac{1}{2 \sqrt{2 \Theta} L ( 1 + \frac{(1 - \rho_B) B_2}{1 + \delta} ) + L}$. This ensures that the final term in \eqref{eq:finalterm} is non-positive, so we can drop it from the inequality along with the term $-1/2 \E \|x_T - \xsol\|^2$. This leaves
\beq
\sum_{k=0}^{T-1} \E [ F(x_{k+1}) - F(\xsol) ] 
\le \tfrac{1}{2 \eta} \|x_0 - \xsol\|^2 + \delta \eta \E [F(x_0) - F(x_T)] .
\eeq
By the convexity of $F$ and the fact that $- F(x_T) \le - F(\xsol)$
\beq
\E \ca{ F(\xbar_{T}) - F(\xsol) }
\le \tfrac{1}{T} \sum_{k=0}^{T-1} \E [ F(x_{k+1}) - F(\xsol) ] 
\le \tfrac{1}{2 \eta T} \|x_0 - \xsol\|^2 + \tfrac{\delta \eta}{T} \pa{ F(x_0) - F(\xsol) }.
\eeq
Choosing $\delta = \max \{ (1 - \rho_B) B_2 - 1, 0\}$ approximately minimizes the right side of this inequality, completing the proof.
\end{prooftext}

\begin{prooftext}{Proof of Theorem \ref{thm:rate_recurs} (Strongly Convex Case)}
We begin with inequality \eqref{eq:delta}, but without setting $\mu = 0$.
\beqs
\begin{aligned}
& \eta (1 + \delta) \Ek [ F(x_{k+1}) - F(\xsol) ] + \tfrac{1 + \mu \eta}{2} \Ek \|x_{k+1} - \xsol\|^2 \\
&\le \eta \delta \pa{ F(x_k) - F(\xsol) } + \tfrac{1}{2} \|x_k - \xsol\|^2 + \Ek \bCa{ \tfrac{\eta (1+\delta)}{2 L \lambda} \|\tnablak{k} - \nabla f(x_k) \|^2 \\
& \qquad + (1 + \delta) \pa{ \tfrac{\eta L ( \lambda + 1 )}{2} - \tfrac{1}{2} } \|x_{k+1} - x_k\|^2 + \eta (1-\rho_B) \langle \nabla f(x_{k-1}) - \tnabla_{k-1}, x_k - \xsol \rangle } .
\end{aligned}
\eeqs
Applying the full expectation operator, multiplying by $(1 + \mu \eta)^k$, and summing over the epoch $k=m s$ to $k = m (s+1)-1$ for some $s \in \mathbb{N}_0$, we have
\beqs
\begin{aligned}
& \eta (1+\delta) \sum_{k=ms}^{m (s+1)-1} (1 + \mu \eta)^k \E [ F(x_{k+1}) - F(\xsol) ] + \tfrac{(1 + \mu \eta)^{m (s + 1)}}{2} \E \ca{ \|x_{m (s+1)} - \xsol\|^2 } \\
&\le \eta \delta \sum_{k=ms}^{m (s+1)-1} (1 + \mu \eta)^k \E [ F(x_{k}) - F(\xsol) ] + \tfrac{(1 + \mu \eta )^{m s}}{2} \E \ca{ \|x_{m s} - \xsol\|^2 } \\
&\qquad + \sum_{k = m s}^{m (s+1)-1} (1 + \mu \eta)^k \E \bCa{ \tfrac{\eta (1+\delta)}{2 L \lambda} \|\tnablak{k} - \nabla f(x_k) \|^2 + (1+\delta) \pa{ \tfrac{\eta L ( \lambda + 1 )}{2} - \tfrac{1}{2} } \|x_{k+1} - x_k\|^2 \\
&\qquad + \eta (1 - \rho_B) \langle \nabla f(x_{k-1}) - \tnabla_{k-1}, x_k - \xsol \rangle } .
\end{aligned}
\eeqs
We would like to bound the inner-product bias term using Lemma \ref{lem:bias}, and we can do this after some manipulation. Because $\eta \le \frac{1}{\mu m}$, we have $(1 + \mu \eta)^k \le 3 (1 + \mu \eta)^{m s}$). Using the same estimate as in equation \eqref{eq:e}, we can say
\beq
\begin{aligned}
&\sum_{k = m s}^{m (s+1)-1} (1 + \mu \eta)^k \E \ca{ \langle \nabla f(x_{k-1}) - \tnabla_{k-1}, x_k - \xsol \rangle } \\
&\le 3 (1 + \mu \eta)^{m s} \sum_{k = m s}^{m (s+1)-1} | \E \ca{ \langle \nabla f(x_{k-1}) - \tnabla_{k-1}, x_k - \xsol \rangle } |,
\end{aligned}
\eeq
We can also choose $\delta$ so that $1 + \delta \ge (1 + \mu \eta) \delta$. These simplifications lead to the inequality
\beqs
\begin{aligned}
& (1 + \mu \eta)^{m (s+1)} \E [ \delta \eta (F(x_{m (s+1)}) - F(\xsol)) + \tfrac{1}{2} \|x_{m (s+1)} - \xsol\|^2 ] \\
&\le \delta \eta (1 + \mu \eta)^{m s} \E [ F(x_{m s}) - F(\xsol) ] + \tfrac{(1 + \mu \eta)^{m s}}{2} \E \|x_{m s} - \xsol\|^2 \\
&\qquad + (1 + \mu \eta)^{m s} \Bigg( \sum_{k = m s}^{m (s+1)-1} \E \bCa{ \tfrac{3 \eta (1+\delta)}{2 L \lambda} \|\tnablak{k} - \nabla f(x_k) \|^2 + (1+\delta) \pa{ \tfrac{3 \eta L ( \lambda + 1 )}{2} - \tfrac{1}{2} } \|x_{k+1} - x_k\|^2 } \\&\qquad + 3 \eta (1-\rho_B) | \mathbb{E} \langle \nabla f(x_{k-1}) - \tnabla_{k-1}, x_k - \xsol \rangle | \Bigg) .
\end{aligned}
\eeqs
Summing this inequality from $s = 0$ to $s = S - 1$,
\beqs
\begin{aligned}
& (1 + \mu \eta)^{m S} \E [ \delta \eta (F(x_{m S}) - F(\xsol)) + \tfrac{1}{2} \|x_{m S} - \xsol\|^2 ] \\
&\le \delta \eta ( F(x_0) - F(\xsol) ) + \tfrac{1}{2} \|x_0 - \xsol\|^2 \\
&\qquad + \sum_{s = 0}^{S - 1} (1 + \mu \eta)^{m s} \Bigg( \sum_{k = m s}^{m (s+1)-1} \E \bCa{ \tfrac{3 \eta (1+\delta)}{2 L \lambda} \|\tnablak{k} - \nabla f(x_k) \|^2 + (1+\delta) \pa{ \tfrac{3 \eta L ( \lambda + 1 )}{2} - \tfrac{1}{2} } \|x_{k+1} - x_k\|^2 } \\&\qquad + 3 \eta (1-\rho_B) | \mathbb{E} \langle \nabla f(x_{k-1}) - \tnabla_{k-1}, x_k - \xsol \rangle | \Bigg) .
\end{aligned}
\eeqs
We use Lemma \ref{lem:msebound} with $\sigma_s = (1 + \mu \eta)^{m s}$ to bound the MSE and Lemma \ref{lem:bias} to bound the inner-product bias term.
\beqs
\label{eq:temp}
\begin{aligned}
& (1 + \mu \eta)^{m S} \E [ \delta \eta (F(x_{m S}) - F(\xsol)) + \tfrac{1}{2} \|x_{m S} - \xsol\|^2 ] \\
&\le \delta \eta \pa{ F(x_0) - F(\xsol) } + \tfrac{1}{2} \|x_0 - \xsol\|^2 + w \sum_{s = 0}^{S-1} \sum_{k=m s}^{m (s+1)-1} (1 + \mu \eta)^{m s} \E \|x_{k+1} - x_k\|^2 ,
\end{aligned}
\eeqs
where $w = { \frac{3 \eta L ( \lambda + 1 ) (1+\delta)}{2} + \frac{3 B_2 \eta (1 - \rho_B)}{2 \epsilon} + \frac{3 \Theta \eta L (1+\delta)}{\lambda} + 3 B_2 \eta L^2 (1 - \rho_B) \epsilon \Theta - \frac{1+\delta}{2} }$. 
To minimize the coefficient of the final term, we set $\lambda = \sqrt{2 \Theta}$ and $\epsilon = (2 L^2 \Theta)^{-1/2}$. This coefficient is then equal to
\beq
3 \sqrt{2 \Theta} \eta L (1 + \delta) + \tfrac{3 \eta L (1+\delta)}{2} + 3 \sqrt{2} (1 - \rho_B) \eta L B_2 \sqrt{\Theta} - \tfrac{1 + \delta}{2}.
\eeq
With
\beq
\eta 
\le \tfrac{1}{6 \sqrt{2 \Theta} L ( 1 + \frac{(1 - \rho_B) B_2}{1 + \delta} ) + L}
\eeq
this term is non-positive. Setting $\delta = \max\{ (1 - \rho_B) B_2 - 1, 0\}$, we are assured that
\beq
\eta 
\le \tfrac{1}{3 L (1 + 4 \sqrt{2 \Theta})} 
\le \tfrac{1}{6 \sqrt{2 \Theta} L ( 1 + \frac{(1 - \rho_B) B_2}{1 + \delta} ) + L},
\eeq
so the final term in \eqref{eq:temp} is non-positive, and we can drop it from the inequality. The resulting inequality is
\beqs\label{eq:recursstrcon}
\begin{aligned}
(1 + \mu \eta)^T \E [ \delta \eta (F(x_T) - F(\xsol)) + \tfrac{1}{2} \|x_T - \xsol\|^2 ] \le \delta \eta ( F(x_0) - F(\xsol) ) + \tfrac{1}{2} \|x_0 - \xsol\|^2.
\end{aligned}
\eeqs
All that remains is to show that our choice for $\delta$ satisfies $(1 + \delta) \ge (1 + \mu \eta) \delta$. Using the fact that
\beq
\eta 
\le \tfrac{1}{(1 - \rho_B) B_2 \mu},
\eeq
we can say
\begin{equation}
    \tfrac{1}{\mu \eta} \ge (1 - \rho_B) B_2 \ge \delta.
\end{equation}
This ensures that $(1 + \delta) \ge (1 + \mu \eta) \delta$ and concludes the proof.
\end{prooftext}

\section{Proof of Theorem \ref{thm:rate_ncvx}}\label{sec:proof_ncvx}

Theorem \ref{thm:rate_ncvx} follows immediately from inequality \eqref{eq:saganoncon} and the MSE bound of Lemma \ref{lem:msebound}.

\begin{prooftext}{Proof of Theorem \ref{thm:rate_ncvx}}

Summing inequality \eqref{eq:saganoncon} from $k=0$ to $k=T-1$ and applying the full expectation operator, we obtain
\beqs
\begin{aligned}
&0 \le - \E \ca{ F(x_T) } + F(x_0) + \pa{ L - \tfrac{1}{4 \eta} } \sum_{k=0}^{T-1} \E \ca{ \|\hat{x}_{k+1} - x_k\|^2 } \\&\qquad + \pa{ \tfrac{L}{2} - \tfrac{1}{4 \eta} } \sum_{k=0}^{T-1} \E \ca{ \|x_{k+1} - x_k\|^2 } + 2 \eta \sum_{k=0}^{T-1} \E \ca{ \| \nabla f(x_k) - \tnablak{k} \|^2 } .
\end{aligned}
\eeqs
We bound the MSE using Lemma \ref{lem:msebound} with $\sigma_s = 1$.
\beqs
\begin{aligned}
0 \le - \E \ca{F(x_T)} + F(x_0) + \pa{ L - \tfrac{1}{4 \eta} } \sum_{k=0}^{T-1} \E \|\hat{x}_{k+1} - x_k\|^2 + \pa{ \tfrac{L}{2} + 4 \Theta \eta L^2 - \tfrac{1}{4 \eta} } \sum_{k=0}^{T-1} \E \|x_{k+1} - x_k\|^2.
\end{aligned}
\eeqs
With $\eta \le \frac{\sqrt{16 \Theta + 1} - 1}{16 L \Theta}$, the final term is non-positive, so we can drop it from the inequality. Using the fact that $-F(x_T) \le -F(\xsol)$, our inequality simplifies to
\beq
- \pa{ L - \tfrac{1}{4 \eta} } \sum_{k=0}^{T-1} \E \ca{ \|\hat{x}_{k+1} - x_k\|^2 } 
\le F(x_0) - F(\xsol).
\eeq
Writing the left side in terms of the generalized gradient, we have the bound
\beqs
\begin{aligned}
\sum_{k=0}^{T-1} \E \ca{ \|\mathcal{G}_{\eta / 2}(x_k)\|^2 }
\le \sfrac{16 ( F(x_0) - F(\xsol) )}{\eta ( 1 - 4 \eta L )} .
\end{aligned}
\eeqs
With $x_{\alpha}$ chosen uniformly at random from the set $\{ x_k \}_{k=0}^{T-1}$, this is equivalent to
\beq
\E \ca{ \|\mathcal{G}_{\eta / 2}(x_\alpha)\|^2 }
\le \sfrac{16 ( F(x_0) - F(\xsol) )}{\eta ( 1 - 4 \eta L ) T} .
\eeq
This completes the proof.
\end{prooftext}

\section{Proofs of convergence rates for B-SAGA and B-SVRG}\label{sec:saga}

The following lemma establishes an MSE bound on the B-SAGA and B-SVRG gradient estimators. For the unbiased case $\theta = 1$, this result was essentially first proved in \citep{SAGA}, but the authors ultimately use a looser variance bound.

\begin{lemma}\label{lem:varsaga}
The MSE's of the B-SAGA and B-SVRG gradient estimators satisfy
\beqs\label{eq:varsaga}
\begin{aligned}
\Ek \ca{ \|\tnablak{k} - \nabla f(x_k) \|^2 } 
\le \sfrac{1}{n \theta^2} \sum_{i=1}^n \|\nabla f_i(x_k) - \nabla f_i(\varphi_k^i) \|^2 + \pa{ 1 - \tfrac{2}{\theta} } \| \nabla f(x_k) - \tfrac{1}{n} \msum_{i=1}^n \nabla f_i(\varphi_k^i) \|^2 .
\end{aligned}
\eeqs
\end{lemma}

\begin{proof}
Let $\tnablak{k} \equiv \gradBSAGA_{k}$ or $\gradBSVRG_{k}$. The proof amounts to computing the expectation of the estimator and applying the Lipschitz continuity of $\nabla f_i$.
\beqs
\begin{aligned}
\Ek \ca{ \|\tnablak{k} - \nabla f(x_k) \|^2 } 
&= \Ek \ca{ \| \tfrac{1}{\theta} (\nabla f_{j_k}(x_k) - \nabla f_{j_k}(\varphi_k^{j_k})) - \nabla f(x_k) - \tfrac{1}{n} \msum_{i=1}^n \nabla f_i(\varphi_k^i) \|^2 } \\
&= \sfrac{1}{\theta^2} \Ek \ca{ \|\nabla f_{j_k}(x_k) - \nabla f_{j_k}(\varphi_k^{j_k}) \|^2 } + \| \nabla f(x_k) - \tfrac{1}{n} \msum_{i=1}^n \nabla f_i(\varphi_k^i) \|^2 \\
& \qquad - \sfrac{2}{\theta} \Ek \ca{ \langle \nabla f_{j_k}(x_k) - \nabla f_{j_k}(\varphi_k^{j_k}), \nabla f(x_k) - \tfrac{1}{n} \msum_{i=1}^n \nabla f_i(\varphi_k^i) \rangle } \\
&= \sfrac{1}{n \theta^2} \sum_{i=1}^n \|\nabla f_i(x_k) - \nabla f_i(\varphi_k^i) \|^2 + \pa{ 1 - \tfrac{2}{\theta} } \| \nabla f(x_k) - \sfrac{1}{n} \msum_{i=1}^n \nabla f_i(\varphi_k^i) \|^2,
\end{aligned}
\eeqs
which is the desired result. 
\end{proof}

The following two lemmas establish the constants in the BMSE property for the B-SAGA and B-SVRG estimators.

\begin{prooftext}{Proof of Lemma \ref{lem:sagamse}}
We begin with the inequality of Lemma \ref{lem:varsaga} and consider two cases.

\paragraph{Case 1.} Suppose $\theta \in [1,2]$. In this case the second term in \eqref{eq:varsaga} is non-positive, so we drop it from the inequality. For the remaining term, we use the following bound.
\beqs\label{eq:case1}
\begin{aligned}
&\tfrac{1}{n} \sum_{i=1}^n \E \ca{ \| \nabla f_i(x_k) - \nabla f_i(\varphi_k^i)\|^2 } \\
&\symnum{1}{\le} \tfrac{1+2 n}{n} \sum_{i=1}^n \E \ca{ \|\nabla f_i(x_k) - \nabla f_i(x_{k-1})\|^2 } + \tfrac{1}{n} \pa{1+\tfrac{1}{2 n}} \sum_{i=1}^n \E \ca{ \|\nabla f_i(x_{k-1}) - \nabla f_i(\varphi_k^i)\|^2 } \\
&\symnum{2}{=} \tfrac{1+2 n}{n} \sum_{i=1}^n \E \ca{ \|\nabla f_i(x_k) - \nabla f_i(x_{k-1})\|^2 } + \tfrac{1}{n} \pa{1+\tfrac{1}{2 n}} \pa{1 - \tfrac{1}{n}} \sum_{i=1}^n \E \ca{ \|\nabla f_i(x_{k-1}) - \nabla f_i(\varphi_{k-1}^i)\|^2 } \\
&\symnum{3}{\le} \tfrac{1+2 n}{n} \sum_{i=1}^n \E \ca{ \|\nabla f_i(x_k) - \nabla f_i(x_{k-1})\|^2 } + \tfrac{1}{n} \pa{1-\tfrac{1}{2 n}} \sum_{i=1}^n \E \ca{ \|\nabla f_i(x_{k-1}) - \nabla f_i(\varphi_{k-1}^i)\|^2 } .
\end{aligned}
\eeqs
Inequality \numcirc{1} is the standard inequality $\|a-c\|^2 \le (1+\delta) \|a-b\|^2 + (1+\delta^{-1}) \|b - c\|^2$ (where we let $\delta = \frac{1}{2n}$). Inequality \numcirc{2} follows from the definition of $\varphi_k^i$ and computing the expectation over $j_{k-1}$, and \numcirc{3} uses the fact that $(1+\frac{1}{2 n}) \pa{1 - \tfrac{1}{n}} \le (1 - \frac{1}{2 n})$. Altogether, this gives
\beqs
\begin{aligned}
&\E \ca{ \|\tnablak{k}^{\textnormal{\tiny SAGA}} - \nabla f(x_k) \|^2 } \\
&\le \tfrac{1}{n \theta^2} \sum_{i=1}^n \E \ca{ \| \nabla f_i(x_k) - \nabla f_i(\varphi_k^i) \| } \\
&\le \tfrac{2 n + 1}{n \theta^2} \sum_{i=1}^n \E \ca{ \|\nabla f_i(x_k) - \nabla f_i(x_{k-1}) \|^2 } + \tfrac{1}{n \theta^2} \pa{ 1 - \tfrac{1}{2 n} } \sum_{i=1}^n \E \ca{ \|\nabla f_i(x_{k-1}) - \nabla f_i(\varphi_{k-1}^i)\|^2 } .
\end{aligned}
\eeqs
With $\calM_k = \frac{1}{n \theta^2} \sum_{i=1}^n \E \ca{ \|\nabla f_i(x_k) - \nabla f_i(\varphi_k^i)\|^2 }$, it is clear that the SAGA estimator satisfies the BMSE property with $M_1 = \frac{2n + 1}{\theta^2}, \rho_M = \frac{1}{2 n}, M_2 = 0, \rho_F = 1$, and $m = 1$.

\paragraph{Case 2.} Suppose $\theta > 2$, so that the second term in \eqref{eq:varsaga} is non-negative. Jensen's inequality gives
\beq
\begin{aligned}
\Ek \ca{ \|\tnablak{k} - \nabla f(x_k) \|^2 } 
\le \tfrac{1}{n} \pa{ 1 - \tfrac{1}{\theta} }^2 \sum_{i=1}^n \|\nabla f_i(x_k) - \nabla f_i(\varphi_k^i) \|^2.
\end{aligned}
\eeq
Following the argument of Case 1, it is easy to see that the B-SAGA gradient estimator satisfies the BMSE property with $\calM_k = \frac{1}{n} \pa{ 1 - \tfrac{1}{\theta} }^2 \sum_{i=1}^n \|\nabla f_i(x_k) - \nabla f_i(\varphi_k^i) \|^2$, $M_1 = (2 n + 1) \pa{ 1 - \tfrac{1}{\theta} }^2$, $\rho_M = \frac{1}{2 n}$, $M_2 = 0$, $\rho_F = 1$, and $m = 1$.

To prove that the B-SAGA estimator is memory-biased, we must only compute its expectation.
\beqs
\begin{aligned}
\nabla f(x_k) - \Ek \ca{ \gradBSAGA_{k} } 
& = \nabla f(x_k) - \tfrac{1}{\theta} \Ek [ \nabla f_{j_k}(x_k) - f_i(\varphi_k^{j_k}) ] - \tfrac{1}{n} \sum_{i=1}^n \nabla f_i(\varphi_k^i) \\
& = \pa{ 1 - \tfrac{1}{\theta} } \bPa{ \nabla f(x_k) - \tfrac{1}{n} \sum_{i=1}^n \nabla f_i(\varphi_k^i) } .
\end{aligned}
\eeqs
To compute a value for $B_1$, we follow \eqref{eq:case1} to obtain
\beqs
\begin{aligned}
\tfrac{1}{n} \sum_{i=1}^n \E \ca{ \|x_k - \varphi_k^i\|^2 } 
&\le (2 n + 1) \|x_k - x_{k-1}\|^2 + \tfrac{1}{n} \pa{ 1 - \tfrac{1}{2 n} } \sum_{i=1}^n \E \ca{ \|x_{k-1} - \varphi_{k-1}^i\|^2 } \\
&\le (2 n + 1) \sum_{\ell = 1}^k \pa{ 1 - \tfrac{1}{2 n} }^{k - \ell} \|x_\ell - x_{\ell-1}\|^2.
\end{aligned}
\eeqs
Summing this inequality from $k = 0$ to $k = T-1$, we obtain
\beqs
\begin{aligned}
\frac{1}{n} \sum_{k=0}^{T-1} \sum_{i=1}^n \E \ca{ \|x_k - \varphi_k^i\|^2 } &\le (2 n + 1) \sum_{k=0}^{T-1} \sum_{\ell = 1}^k (1 - \tfrac{1}{2 n})^{k - \ell} \|x_\ell - x_{\ell-1}\|^2 \\
&\le (2 n + 1) \Pa{ \msum_{\ell = 0}^\infty \pa{ 1 - \tfrac{1}{2 n} }^{\ell} } \sum_{k=0}^{T-1} \|x_{k+1} - x_k\|^2 \\
&= 2 n (2 n + 1) \sum_{k=0}^{T-1} \|x_{k+1} - x_k\|^2,
\end{aligned}
\eeqs
which completes the proof.
\end{prooftext}

\begin{prooftext}{Proof of Lemma \ref{lem:svrgmse}}
Suppose $k \in \{ m s, m s + 1, \cdots m (s+1) - 1 \}$ for some $s \in \mathbb{N}_0$. As in the proof of Lemma \ref{lem:sagamse}, we begin with the inequality of Lemma \ref{lem:varsaga} and consider two cases.

\paragraph{Case 1.} Suppose $\theta \in [1,2]$, so that we may drop the second term in \eqref{eq:varsaga}. We can bound the remaining term as follows.
\beqs
\begin{aligned}
\tfrac{1}{n \theta^2} \sum_{i=1}^n \|\nabla f_i (x_k) - \nabla f_i (\varphi_s) \|^2 
&\symnum{1}{\le} \tfrac{1+m}{n \theta^2} \sum_{i=1}^n \|\nabla f_i(x_k) - \nabla f_i(x_{k-1})\|^2 + \tfrac{1+1/m}{n \theta^2} \sum_{i=1}^n \|\nabla f_i(x_{k-1}) - \nabla f_i(\varphi_s)\|^2 \\
&\symnum{2}{\le} \tfrac{1+m}{n \theta^2} \sum_{\ell=m s}^k (1+\tfrac{1}{m})^{k - \ell} \sum_{i=1}^n \|\nabla f_i(x_{\ell+1}) - \nabla f_i(x_{\ell}) \|^2.
\end{aligned}
\eeqs
Inequality \numcirc{1} uses the inequality $\|u - w\|^2 \le (1+1/m) \|u - v\|^2 + (1+m)\|v - w\|^2$, and \numcirc{2} follows from the fact that $x_{m s} = \varphi_s$. Summing this inequality from $k = m s$ to $k = m(s+1)-1$ gives us
\beqs
\begin{aligned}
\sfrac{1}{n \theta^2} \sum_{k = m s}^{m (s+1)-1} \sum_{i=1}^n \|\nabla f_i(x_k) - \nabla f_i(\varphi_s) \|^2 
&\le \sfrac{m+1}{n \theta^2} \pa{ 1 + \tfrac{1}{m} }^m \sum_{k = m s}^{m (s+1)-1} \sum_{\ell=m s}^k \sum_{i=1}^n \|\nabla f_i(x_{\ell+1}) - \nabla f_i(x_{\ell})\|^2 \\
&\le \sfrac{m (m+1)}{n \theta^2} \pa{ 1 + \tfrac{1}{m} }^m \sum_{k = m s}^{m (s+1)-1} \sum_{i=1}^n \|\nabla f_i(x_{\ell+1}) - \nabla f_i(x_{\ell})\|^2 \\
&\le \sfrac{3 m (m + 1)}{n \theta^2} \sum_{k = m s}^{m (s + 1) - 1} \|\nabla f_i(x_{k+1}) - \nabla f_i(x_k)\|^2.
\end{aligned}
\eeqs
The final inequality uses the fact that $( 1 + \frac{1}{m} )^m < \lim_{m \to \infty} ( 1 + \frac{1}{m} )^m = e < 3$. From this inequality, it is clear that the B-SVRG gradient estimator satisfies the BMSE property with $M_1 = \frac{3 m (m+1)}{\theta^2}$, $\rho_M = 1$, $M_2 = 0$, and $\rho_F = 1$.

\paragraph{Case 2.} If $\theta > 2$, then applying Jensen's inequality to \eqref{eq:varsaga} produces
\beq
\Ek \ca{ \|\gradBSVRG_{k} - \nabla f(x_k) \|^2 }
\le \tfrac{1}{n} \pa{ 1 - \tfrac{1}{\theta} }^2 \sum_{i=1}^n \|\nabla f_i(x_k) - \nabla f_i(\varphi_s) \|^2.
\eeq
A similar argument to the one in Case 1 shows that $M_1 = 3 m (m+1) \pa{ 1 - \tfrac{1}{\theta} }^2$, $\rho_M = 1$, $M_2 = 0$, and $\rho_F = 1$.

All that is left is to prove the stated value for $B_1$. Following the proof in Case 1,
\beqs
\begin{aligned}
\sum_{k = m s}^{m (s + 1) - 1} \|x_k - \varphi_s\|^2 \le \sum_{k = m s}^{m (s + 1) - 1} \sum_{\ell = m s}^k (1+m) (1+\sfrac{1}{m})^m \|x_{\ell+1} - x_\ell\|^2 \le 3 m (m + 1) \sum_{k = m s}^{m (s + 1) - 1} \|x_{k + 1} - x_k\|^2.
\end{aligned}
\eeqs
Summing over the epochs $s = 0$ to $s = S$ shows $B_1 = 3 m (m + 1)$.
\end{prooftext}

Combining Lemmas \ref{lem:sagamse} and \ref{lem:svrgmse} with Theorems \ref{thm:rate_mem} and \ref{thm:rate_ncvx} proves convergence rates for B-SAGA and B-SVRG.

\section{Proof of convergence rates for SARAH}\label{sec:sarah}

Lemma \ref{lem:sarahmse} establishes the BMSE constants for the SARAH estimator. The convergence rates of Corollary \ref{thm:sarah} then follow immediately from Theorem \ref{thm:rate_recurs}.

\begin{prooftext}{Proof of Lemma \ref{lem:sarahmse}}
Let $k \in \{ m s+1,ms+2,\cdots,m(s+1) - 1 \}$. The claim follows immediately from the well-known bound on the MSE of the SARAH gradient estimator
\beq
\|\gradSARAH_k - \nabla f(x_k) \|^2 
\le \sfrac{1}{n} \sum_{\ell=m s}^k \sum_{i=1}^n \|\nabla f_i(x_{\ell+1}) - \nabla f_i(x_\ell)\|^2.
\eeq
We refer to \citep{spider}, for example, for a proof of this inequality. Summing over an epoch and applying the estimate
\beq\label{eq:sarahest}
\sfrac{1}{n} \sum_{k=m s}^{m(s+1)-1} \sum_{\ell=m s}^k \sum_{i=1}^n \|\nabla f_i(x_{\ell+1}) - \nabla f_i(x_\ell) \|^2
\le \sfrac{m}{n} \sum_{k=m s}^{m(s+1)-1} \sum_{i=1}^n \|\nabla f_i(x_{k+1}) - \nabla f_i(x_k)\|^2
\eeq
complete the proof.
\end{prooftext}

\section{Proof of convergence rates for SARGE}\label{sec:sarge}

For our analysis, we write the SARGE gradient estimator in terms of the SAGA estimator. Define the estimator
\beq
\tnabla^{\xi\textnormal{\tiny-SAGA}}_k \defeq \nabla f_{j_k}(x_{k-1}) - \nabla f_{j_k}(\xi^{j_k}_k) + \tfrac{1}{n} \sum_{i=1}^n \nabla f_i(\xi^i_k),
\eeq
where the variables $\{\xi_k^i\}_{i=1}^n$ follow the update rules $\xi_{k+1}^{j_k} = x_{k-1}$ and $\xi_{k+1}^i = \xi_k^i$ for all $i \not = j_k$. The SARGE estimator is equal to
\beq
\gradSARGE_{k} = \gradSAGA_{k} - \pa{1 - \tfrac{1}{n}} ( \tnabla^{\xi\textnormal{\tiny-SAGA}}_{k} - \gradSARGE_{k} ).
\eeq
Before we prove Lemma \ref{lem:sargemse}, we require a bound on the MSE of the $\xi$-SAGA gradient estimator that follows immediately from Lemma \ref{lem:varsaga}.

\begin{lemma}
\label{lem:sagaforsarge}
The MSE of the $\xi$-SAGA gradient estimator satisfies the following bound:
\beq
\E \ca{ \|\tnabla^{\xi\textnormal{\tiny-SAGA}}_k - \nabla f(x_{k-1}) \|^2 } 
\le 3 \sum_{\ell = 1}^{k - 1} (1 - \tfrac{1}{2 n} )^{k - \ell - 1} \sum_{i=1}^n \E \ca{ \|\nabla f_i(x_{\ell}) - \nabla f_i(x_{\ell-1}) \|^2 } . 
\eeq
\end{lemma}

\begin{proof}
Following the proof of Lemma \ref{lem:varsaga},
\beqs
\begin{aligned}
\Ek \ca{ \|\tnabla^{\xi\textnormal{\tiny-SAGA}}_k - \nabla f(x_{k-1}) \|^2 }
&= \Ek \ca{ \|\nabla f_{j_k}(x_{k-1}) - \nabla f_{j_k}(\xi_k^{j_k}) - \nabla f(x_{k-1}) + \tfrac{1}{n} \msum_{i=1}^n \nabla f_i(\xi_k^i) \|^2 } \\
&\symnum{1}{=} \frac{1}{n} \sum_{i=1}^n \|\nabla f_i(x_{k-1}) - \nabla f_i(\xi_k^i) \|^2 - \| \nabla f(x_{k-1}) - \tfrac{1}{n} \msum_{i=1}^n \nabla f_i(\xi_k^i) \|^2 \\
&\le \frac{1}{n} \sum_{i=1}^n \|\nabla f_i(x_{k-1}) - \nabla f_i(\xi_k^i) \|^2.
\end{aligned}
\eeqs
Equality \numcirc{1} is the standard variance decomposition. To continue, we follow the proof of Lemma \ref{lem:varsaga}.
\beqs
\begin{aligned}
&\E \ca{ \|\tnabla^{\xi\textnormal{\tiny-SAGA}}_k - \nabla f(x_{k-1}) \|^2 } \\
&\le \tfrac{1}{n} \sum_{i=1}^n \E \ca{ \|\nabla f_i(x_{k-1}) - \nabla f_i(\xi_k^i) \|^2 } \\
&\le \tfrac{1 + 2 n}{n} \sum_{i=1}^n \E \ca{ \|\nabla f_i(x_{k-1}) - \nabla f_i(x_{k-2}) \|^2 } + \tfrac{1}{n} ( 1 + \tfrac{1}{2 n} ) \sum_{i=1}^n \E \ca{ \| \nabla f_i(x_{k-2}) - \nabla f_i(\xi_k^i) \|^2 } \\
&\symnum{2}{=} \tfrac{(1 + 2 n)}{n} \sum_{i=1}^n \E \ca{ \|\nabla f_i(x_{k-1}) - \nabla f_i(x_{k-2}) \|^2 } + \tfrac{1}{n} ( 1 + \tfrac{1}{2 n} ) \pa{1 - \tfrac{1}{n}} \sum_{i=1}^n \E \ca{ \| \nabla f_i(x_{k-2}) - \nabla f_i(\xi_{k-1}^i) \|^2 } \\
&\symnum{3}{\le} 3 \sum_{i=1}^n \E \ca{ \|\nabla f_i(x_{k-1}) - \nabla f_i(x_{k-2}) \|^2 } + \tfrac{1}{n} ( 1 - \tfrac{1}{2 n} ) \sum_{i=1}^n \E \ca{ \| \nabla f_i(x_{k-2}) - \nabla f_i(\xi_{k-1}^i) \|^2 } \\
&\le 3 \sum_{\ell = 1}^{k-1} (1 - \tfrac{1}{2 n} )^{k - \ell - 1} \sum_{i=1}^n \E \ca{ \|\nabla f_i(x_{\ell}) - \nabla f_i(x_{\ell-1}) \|^2 } .
\end{aligned}
\eeqs
Equality \numcirc{2} follows from computing expectations, and \numcirc{3} uses the estimate $\pa{1 - \tfrac{1}{n}}(1+\frac{1}{2 n}) \le (1-\frac{1}{2 n})$.
\end{proof}

Due to the recursive nature of the SARGE gradient estimator, its MSE depends on the difference between the current estimate and the estimate from the previous iteration. The next lemma provides a bound on $\|\gradSARGE_k - \gradSARGE_{k-1}\|^2$.

\begin{lemma}
\label{lem:vbound}
The SARGE gradient estimator satisfies the following bound:
\beqs
\begin{aligned}
\E \ca{ \| \gradSARGE_k - \gradSARGE_{k-1}\|^2 } 
&\le \tfrac{12}{n} \sum_{i=1}^n \E \ca{ \|\nabla f_i(x_k) - \nabla f_i(x_{k-1}) \|^2 } + \tfrac{3}{2 n^2} \E \|\nabla f(x_{k-1}) - \gradSARGE_{k-1} \|^2 \\&\qquad + \tfrac{39}{n^2} \sum_{\ell=1}^k ( 1 - \tfrac{1}{2 n} )^{k - \ell} \sum_{i=1}^n \E \|\nabla f_i(x_{\ell}) - \nabla f_i(x_{\ell - 1}) \|^2 .
\end{aligned}
\eeqs
\end{lemma}

\begin{proof}
To begin, we use the standard inequality $\|a - c\|^2 \le (1+\delta) \|a - b\|^2 + (1+\delta^{-1}) \|b - c\|^2$ for any $\delta > 0$ twice. For simplicity, we set $\delta = \sqrt{3/2} - 1$ and use the fact that $1 + \frac{1}{\sqrt{3/2} - 1} \le 6$ for both applications of this inequality.
\beqs\label{eq:vbound}
\begin{aligned}
&\E \ca{ \| \gradSARGE_k - \gradSARGE_{k-1}\|^2 } \\
&= \E \ca{ \| \gradSAGA_k - \pa{1 - \tfrac{1}{n}} ( \tnabla^{\xi\textnormal{\tiny-SAGA}}_k - \gradSARGE_{k-1} ) - \gradSARGE_{k-1} \|^2 } \\
&\le 6 \E \ca{ \| \gradSAGA_k - \tnabla^{\xi\textnormal{\tiny-SAGA}}_k \|^2 } + \tfrac{\sqrt{3}}{\sqrt{2} n^2} \E \ca{ \|\tnabla^{\xi\textnormal{\tiny-SAGA}}_k - \gradSARGE_{k-1} \|^2 } \\
&\le 6 \E \ca{ \| \gradSAGA_k - \tnabla^{\xi\textnormal{\tiny-SAGA}}_k \|^2 } + \tfrac{6 \sqrt{3}}{\sqrt{2} n^2} \E \ca{ \|\tnabla^{\xi\textnormal{\tiny-SAGA}}_k - \nabla f(x_{k-1}) \|^2 } + \tfrac{3}{2 n^2} \E \ca{ \|\nabla f(x_{k-1}) - \gradSARGE_{k-1} \|^2 }.
\end{aligned}
\eeqs
We use $\tfrac{6 \sqrt{3}}{\sqrt{2} n^2} \le \tfrac{9}{n^2}$ to simplify the coefficient of the second term. We now bound the first two of these three terms separately. Consider the first term.
\beqs
\begin{aligned}
&6 \E \ca{ \| \gradSAGA_k - \tnabla^{\xi\textnormal{\tiny-SAGA}}_k \|^2 } \\
&= 6 \E \ca{ \| \nabla f_{j_k}(x_k) - \nabla f_{j_k}(\varphi_k^{j_k}) + \tfrac{1}{n} \msum_{i=1}^n \nabla f_i(\varphi_k^i) - \nabla f_{j_k}(x_{k-1}) - \nabla f_{j_k}(\xi_k^{j_k}) - \tfrac{1}{n} \msum_{i=1}^n \nabla f_i(\xi_k^i) \|^2 } \\
&\le 12 \E \ca{ \| \nabla f_{j_k}(x_k) - \nabla f_{j_k}(x_{k-1}) \|^2 } \\
&\qquad + 12 \E \ca{ \| \nabla f_{j_k}(\varphi_k^{j_k}) - \nabla f_{j_k}(\xi_k^{j_k}) - \tfrac{1}{n} \msum_{i=1}^n \nabla f_i(\varphi_k^i) + \tfrac{1}{n} \msum_{i=1}^n \nabla f_i(\xi_k^i) \|^2 } \\
&\symnum{1}{=} 12 \E \ca{ \| \nabla f_{j_k}(x_k) - \nabla f_{j_k}(x_{k-1}) \|^2 } \\
&\qquad + 12 \E \ca{ \| \nabla f_{j_k}(\varphi_k^{j_k}) - \nabla f_j(\xi_k^{j_k}) \|^2 } - 12 \E \ca{ \| \tfrac{1}{n} \msum_{i=1}^n \nabla f_i(\varphi_k^i) - \tfrac{1}{n} \msum_{i=1}^n \nabla f_i(\xi_k^i) \|^2 } \\
&\le \tfrac{12}{n} \sum_{i=1}^n \E \ca{ \| \nabla f_i(x_k) - \nabla f_i(x_{k-1}) \|^2 } + 12 \E \ca{ \| \nabla f_{j_k}(\varphi_k^{j_k}) - \nabla f_{j_k} (\xi_k^{j_k}) \|^2 } \\
&\le \tfrac{12}{n} \sum_{i=1}^n \E \ca{ \| \nabla f_i(x_k) - \nabla f_i(x_{k-1}) \|^2 } + 12 \E \ca{ \| \nabla f_{j_k}(\varphi_k^{j_k}) - \nabla f_{j_k} (\xi_k^{j_k}) \|^2 } .
\end{aligned}
\eeqs
Equality \numcirc{1} is the standard variance decomposition, which states that for any random variable $X$, $\E \ca{ \| X - \E X \|^2 } = \E \ca{\|X\|^2} - \|\E \ca{X} \|^2$. The second term can be reduced further by computing the expectation. The probability that $\nabla f_{j_k} (\varphi_k^{j_k}) = \nabla f_{j_{k-1}}(x_{k-1})$ is equal to the probability that $j_k = j_{k-1}$, which is $1 / n$. The probability that $\nabla f_{j_k}(\varphi_k^{j_k}) = \nabla f_{j_{k-2}} (x_{k-2})$ is equal to the probability that $j_k \not = j_{k-1}$ and $j_k = j_{k-2}$, which is $\frac{1}{n} ( 1 - \frac{1}{n} )$. Continuing in this way,
\beqs
\begin{aligned}
&\E \ca{ \|\nabla f_{j_k}(\varphi_k^{j_k}) - \nabla f_{j_k}(\xi_k^{j_k}) \|^2 } \\
& = \tfrac{1}{n} \E \ca{ \| \nabla f_{j_{k-1}} (x_{k-1}) - \nabla f_{j_{k-1}} (x_{k-2}) \|^2 } + \tfrac{1}{n} \pa{1 - \tfrac{1}{n}} \E \ca{ \| \nabla f_{j_{k-2}} (x_{k-2}) - \nabla f_{j_{k-3}} (x_{k-2}) \|^2 } + \cdots \\
& = \tfrac{1}{n} \sum_{\ell = 1}^{k-1} \pa{1 - \tfrac{1}{n}}^{k - \ell - 1} \E \ca{ \| \nabla f_{j_{\ell}}(x_{\ell}) - \nabla f_{j_{\ell}}(x_{\ell-1})\|^2 } .
\end{aligned}
\eeqs
This implies that
\beqs
\begin{aligned}
12 \E \ca{ \| \nabla f_{j_k}(\varphi_k^{j_k}) - \nabla f_{j_k} (\xi_k^{j_k}) \|^2 } 
&\le \tfrac{12}{n^2} \sum_{\ell=1}^{k-1} \pa{1 - \tfrac{1}{n}}^{k-\ell-1} \sum_{i=1}^n \E \ca{ \| \nabla f_i(x_{\ell}) - \nabla f_i(x_{\ell-1})\|^2 } \\
&\le \tfrac{12}{n^2} \sum_{\ell=1}^{k-1} (1-\tfrac{1}{2 n})^{k-\ell-1} \sum_{i=1}^n \E \ca{ \| \nabla f_i(x_{\ell}) - \nabla f_i(x_{\ell-1})\|^2 } .
\end{aligned}
\eeqs
We include the second inequality to simplify later arguments. This completes our bound for the first term of \eqref{eq:vbound}.

For the second term of \eqref{eq:vbound}, we recall Lemma \ref{lem:sagaforsarge}.
\beq
\E \ca{ \|\tnabla^{\xi\textnormal{\tiny-SAGA}}_k - \nabla f(x_{k-1}) \|^2 } 
\le 3 \sum_{\ell = 1}^{k-1} (1 - \tfrac{1}{2 n} )^{k - \ell - 1} \sum_{i=1}^n \E \ca{ \|\nabla f_i(x_{\ell}) - \nabla f_i(x_{\ell-1}) \|^2 } .
\eeq
Combining all of these bounds, we obtain
\beqs
\begin{aligned}
\E \ca{ \| \gradSARGE_k - \gradSARGE_{k-1}\|^2 } 
&\le \tfrac{12}{n} \sum_{i=1}^n \E \ca{ \|\nabla f_i(x_k) - \nabla f_i(x_{k-1}) \|^2 } + \tfrac{3}{2 n^2} \E \ca{ \|\nabla f(x_{k-1}) - \gradSARGE_{k-1} \|^2 } \\&\qquad + \tfrac{39}{n^2} \sum_{\ell=1}^{k-1} ( 1 - \tfrac{1}{2 n} )^{k - \ell - 1} \sum_{i=1}^n \E \ca{ \|\nabla f_i(x_{\ell}) - \nabla f_i(x_{\ell - 1}) \|^2 },
\end{aligned}
\eeqs
which completes the proof.
\end{proof}
Lemma \ref{lem:vbound} allows us to take advantage of the recursive structure of our gradient estimate. With this lemma established, we can prove a bound on the MSE.

\begin{lemma}
\label{lem:sargerecurse}
The SARGE gradient estimator satisfies the following recursive bound:
\beqs
\begin{aligned}
\E \ca{ \|\gradSARGE_k - \nabla f(x_k)\|^2 } 
&\le ( 1 - \tfrac{1}{n} + \tfrac{3}{2 n^2} ) \E \ca{ \| \gradSARGE_{k-1} - \nabla f(x_{k-1}) \|^2 } + \tfrac{12}{n} \sum_{i=1}^n \E \ca{ \|\nabla f_i(x_k) - \nabla f_i(x_{k-1}) \|^2 } \\&\qquad + \tfrac{39}{n^2} \sum_{\ell=1}^{k-1} ( 1 - \tfrac{1}{2 n} )^{k - \ell - 1} \sum_{i=1}^n \E \ca{ \|\nabla f_i(x_{\ell}) - \nabla f_i(x_{\ell - 1}) \|^2 } . 
\end{aligned}
\eeqs
\end{lemma}

\begin{proof}
The beginning of our proof is similar to the proof of the variance bound for the SARAH gradient estimator in \cite[Lem. 2]{sarah}.
\beqs
\begin{aligned}
\Ek \|\gradSARGE_k - \nabla f(x_k)\|^2 
&= \Ek \ca{ \| \gradSARGE_{k-1} - \nabla f(x_{k-1}) + \nabla f(x_{k-1}) - \nabla f(x_k) + \gradSARGE_k - \gradSARGE_{k-1} \|^2 } \\
&= \| \gradSARGE_{k-1} - \nabla f(x_{k-1}) \|^2 + \|\nabla f(x_{k-1}) - \nabla f(x_k) \|^2 + \Ek \ca{ \| \gradSARGE_k - \gradSARGE_{k-1} \|^2 } \\
&\qquad + 2 \langle \nabla f(x_{k-1} ) - \gradSARGE_{k-1}, \nabla f(x_k) - \nabla f(x_{k-1}) \rangle \\
&\qquad - 2 \langle \nabla f(x_{k-1} ) - \gradSARGE_{k-1}, \Ek [ \gradSARGE_k - \gradSARGE_{k-1} ] \rangle \\
&\qquad - 2 \langle \nabla f(x_k) - \nabla f(x_{k-1}), \Ek [ \gradSARGE_k - \gradSARGE_{k-1} ] \rangle.
\end{aligned}
\eeqs
We consider each inner product separately. The first inner product is equal to
\beqs
\begin{aligned}
&2 \langle \nabla f(x_{k-1} ) - \gradSARGE_{k-1}, \nabla f(x_k) - \nabla f(x_{k-1}) \rangle \\
&= - \|\nabla f(x_{k-1} ) - \gradSARGE_{k-1}\|^2 - \|\nabla f(x_k) - \nabla f(x_{k-1})\|^2 + \|\nabla f(x_k) - \gradSARGE_{k-1}\|^2.
\end{aligned}
\eeqs
For the next two inner products, we use the fact that
\beqs
\begin{aligned}
\Ek [ \gradSARGE_k - \gradSARGE_{k-1} ] 
&= \Ek \Ca{ \gradSAGA_k - \pa{1 - \tfrac{1}{n}} \tnabla^{\xi\textnormal{\tiny-SAGA}}_k + \pa{1 - \tfrac{1}{n}} \gradSARGE_{k-1} } - \gradSARGE_{k-1} \\
&= \nabla f(x_k) - ( 1 - \tfrac{1}{n} ) \nabla f(x_{k-1}) - \tfrac{1}{n} \gradSARGE_{k-1} \\
&= \nabla f(x_k) - \nabla f(x_{k-1}) + \tfrac{1}{n} ( \nabla f(x_{k-1}) - \gradSARGE_{k-1} ).
\end{aligned}
\eeqs
With this equality established, we see that the second inner product is equal to
\beqs
\begin{aligned}
& - 2 \langle \nabla f(x_{k-1}) - \gradSARGE_{k-1}, \Ek [ \gradSARGE_k - \gradSARGE_{k-1} ] \rangle \\
&= - 2 \langle \nabla f(x_{k-1}) - \gradSARGE_{k-1}, \nabla f(x_k) - \nabla f(x_{k-1}) \rangle - \tfrac{2}{n} \langle \nabla f(x_{k-1} ) - \gradSARGE_{k-1}, \nabla f(x_{k-1}) - \gradSARGE_{k-1} \rangle \\
&= \| \nabla f(x_{k-1} ) - \gradSARGE_{k-1} \|^2 + \| \nabla f(x_k) - \nabla f(x_{k-1}) \|^2 \\&\qquad - \|\nabla f(x_k) - \gradSARGE_{k-1} \|^2 - \tfrac{2}{n} \| \nabla f(x_{k-1} ) - \gradSARGE_{k-1} \|^2 \\
&= (1 - \tfrac{2}{n} ) \| \nabla f(x_{k-1} ) - \gradSARGE_{k-1} \|^2 + \| \nabla f(x_k) - \nabla f(x_{k-1}) \|^2 - \|\nabla f(x_k) - \gradSARGE_{k-1} \|^2.
\end{aligned}
\eeqs
The third inner product can be bounded using a similar procedure.
\beqs
\begin{aligned}
& - 2 \langle \nabla f(x_k) - \nabla f(x_{k-1}), \Ek [ \gradSARGE_k - \gradSARGE_{k-1} ] \rangle \\
&= - 2 \langle \nabla f(x_k) - \nabla f(x_{k-1}), \nabla f(x_k) - \nabla f(x_{k-1}) \rangle - \tfrac{2}{n} \langle \nabla f(x_k ) - \nabla f(x_{k-1}), \nabla f(x_{k-1}) - \gradSARGE_{k-1} \rangle \\
&\le - 2 \| \nabla f(x_k) - \nabla f(x_{k-1}) \|^2 + \tfrac{1}{n} \|\nabla f(x_k ) - \nabla f(x_{k-1})\|^2 + \tfrac{1}{n} \|\nabla f(x_{k-1}) - \gradSARGE_{k-1} \|^2 \\
&= - (2 - \tfrac{1}{n}) \| \nabla f(x_k) - \nabla f(x_{k-1}) \|^2 + \tfrac{1}{n} \|\nabla f(x_{k-1}) - \gradSARGE_{k-1} \|^2 ,
\end{aligned}
\eeqs
where the inequality is Young's. Altogether and after applying the full expectation operator, we have
\beqs
\begin{aligned}
\E \ca{ \|\gradSARGE_k - \nabla f(x_k)\|^2 } 
&\le ( 1 - \tfrac{1}{n} ) \E \ca{ \| \gradSARGE_{k-1} - \nabla f(x_{k-1}) \|^2 } \\&\qquad - ( 1 - \tfrac{1}{n} ) \E \ca{ \|\nabla f(x_k) - \nabla f(x_{k-1})\|^2 } + \E \ca{ \| \gradSARGE_k - \gradSARGE_{k-1}\|^2 } \\
&\le ( 1 - \tfrac{1}{n} ) \E \ca{ \| \gradSARGE_{k-1} - \nabla f(x_{k-1}) \|^2 } + \E \ca{ \| \gradSARGE_k - \gradSARGE_{k-1}\|^2 } .
\end{aligned}
\eeqs
Finally, we bound the last term on the right using Lemma \ref{lem:vbound}.
\beqs
\begin{aligned}
\E \|\gradSARGE_k - \nabla f(x_k)\|^2 
&\le ( 1 - \tfrac{1}{n} + \tfrac{3}{2 n^2} ) \E \ca{ \| \gradSARGE_{k-1} - \nabla f(x_{k-1}) \|^2 } + \sfrac{12}{n} \sum_{i=1}^n \E \ca{ \|\nabla f_i(x_k) - \nabla f_i(x_{k-1}) \|^2 } \\
&\qquad + \sfrac{39}{n^2} \sum_{\ell=1}^{k-1} ( 1 - \tfrac{1}{2 n} )^{k - \ell - 1} \sum_{i=1}^n \E \ca{ \|\nabla f_i(x_{\ell}) - \nabla f_i(x_{\ell - 1}) \|^2 } 
\end{aligned}
\eeqs
and complete the proof.
\end{proof}

\begin{prooftext}{Proof of Lemma \ref{lem:sargemse}}
It is easy to see that $\rho_B = 1 / n$ by computing the expectation of the SARGE gradient estimator.
\beqs
\begin{aligned}
\nabla f(x_k) - \Ek \ca{ \gradSARGE_k } 
& = \nabla f(x_k) - \Ek [ \gradSAGA_k - \pa{1 - \tfrac{1}{n}} ( \tnabla^{\xi\textnormal{\tiny-SAGA}}_k - \gradSARGE_{k-1} ) ] \\
&= \pa{1 - \tfrac{1}{n}} ( \nabla f(x_{k-1}) - \gradSARGE_{k-1} ).
\end{aligned}
\eeqs
The result of Lemma \ref{lem:sargerecurse} makes it clear that $M_1 = 12$. To determine $\rho_M$, we must first choose a suitable sequence $\calM_k$. Let $\calM_k = \E \ca{ \|\gradSARGE_k - \nabla f(x_k)\|^2 }$. If $n = 1$, then $\calM_k=0$ for all $k$, so it holds trivially that $\calM_k \le (1-\rho_M) \calM_{k-1}$. If $n \ge 2$, then $1 - \frac{1}{n} + \frac{3}{2 n^2} \le 1 - \frac{1}{4 n}$, so Lemma \ref{lem:sargerecurse} ensures that with $\rho_M = \frac{1}{4 n}$, $\calM_k \le (1-\rho_M) \calM_{k-1}$.

Finally, we must compute $M_2$ and $\rho_F$ with respect to some sequence $\calF_k$. Lemma \ref{lem:sargerecurse} motivates the choice
\beq
\calF_k = \sum_{\ell=1}^{k-1} \pa{ 1 - \tfrac{1}{2 n} }^{k - \ell - 1} \sum_{i=1}^n \E \ca{ \|\nabla f_i(x_{\ell}) - \nabla f_i(x_{\ell - 1}) \|^2 } ,
\eeq
and the choices $M_2 = \frac{39}{n}$ and $\rho_F = \frac{1}{2 n}$ are clear.
\end{prooftext}

\end{document}